\documentclass[a4paper, 12pt]{article}

\usepackage{amsmath} 
\usepackage{theorem}
\usepackage{amssymb}
\usepackage{amsfonts}
\usepackage{latexsym}
\usepackage{amscd}
\usepackage{diagramb}
\usepackage{graphicx}
\usepackage{enumerate}
\usepackage{yfonts}
\usepackage{upgreek}
\usepackage{cancel}
\usepackage{hyperref}
\usepackage{pxfonts}  
\usepackage{caption}
\usepackage{multirow}
\usepackage{float}

\usepackage{pdflscape}
\usepackage{afterpage}
\usepackage{capt-of}
\usepackage{lipsum}

\usepackage{stmaryrd}
\usepackage{tikz} 
\usepackage{circledsteps}
\usepackage{quiver} 

    \newsavebox\tikzcdbox
    \newenvironment{tikzcdscale}{%
      \begin{lrbox}{\tikzcdbox}%
      \begin{tikzcd}%
    }{%
      \end{tikzcd}%
      \end{lrbox}%
      \resizebox{\textwidth}{!}{\usebox\tikzcdbox}%
    }

\usepackage{xspace,colortbl}
\usepackage{color}

\usepackage{arydshln} 

\definecolor{verde}{rgb}{0.,0.7,0.}
\definecolor{indigo}{rgb}{.18, .34, .78}
\definecolor{indigo1}{rgb}{.18, .24, .78}
\definecolor{indigo2}{rgb}{.18, .14, .78}
\definecolor{indigo3}{rgb}{.18, 0., .78}
\definecolor{rojo}{rgb}{1,0,0}
\definecolor{negro}{rgb}{0,0,0}
\definecolor{lila}{rgb}{.46, .16, .78}
\definecolor{lila1}{rgb}{.46, .16, .86}
\definecolor{lila2}{rgb}{.56, .16, .86}
	\definecolor{lila3}{rgb}{.63, .16, .78}
\definecolor{lila4}{rgb}{.7, .16, .78}
\definecolor{lila5}{rgb}{.78, .26, .78}
\definecolor{lila6}{rgb}{.6, 0., .78}
\definecolor{grey}{rgb}{0.5,0.5,0.5}

\DeclareMathAlphabet{\mathpzc}{OT1}{pzc}{m}{it}

    \newcommand{\axiom}[1]{\textbf{\hypertarget{#1}{(#1)}}}
    \newcommand{\axiomref}[1]{\hyperlink{#1}{(#1)}}

\newtheorem{thm}{Theorem}

\newtheorem{defn}{Definition}

\theorembodyfont{\rmfamily}

\newtheorem{rem}[thm]{Remark}
\newtheorem{prop}[thm]{Proposition}
\newtheorem{ex}[thm]{Example}
\newtheorem{cor}[thm]{Corollary}

\newcommand{\qed}{\hfill\quad\fbox{\rule[0mm]{0,0cm}{0,0mm}}  \par\bigskip}

\newcommand\SWarrow{\mathrel{\rotatebox[origin=c]{-135}{$\Rightarrow$}}}  
\newcommand\NWarrow{\mathrel{\rotatebox[origin=c]{135}{$\Rightarrow$}}}  
\newcommand\SEarrow{\mathrel{\rotatebox[origin=c]{-45}{$\Rightarrow$}}}  

\newcommand{\x}{\mbox{-}}

\newcommand{\comp}{\circ}
\newcommand{\iso}{\cong}
\newcommand{\ot}{\otimes}

\newcommand{\crta}{\overline}

\newcommand{\Aa}{{\mathbb A}}
\newcommand{\Bb}{{\mathbb B}}
\newcommand{\Cc}{{\mathbb C}}

\def\Nn{{\mathbb N}}

\newcommand{\Ps}{\operatorname Ps}
\newcommand{\C}{{\mathcal C}}
\newcommand{\D}{{\mathcal D}}
\newcommand{\A}{{\mathcal A}}
\newcommand{\B}{{\mathcal B}}
\newcommand{\F}{{\mathcal F}}
\newcommand{\G}{{\mathcal G}}

\newcommand{\M}{{\mathcal M}}
\newcommand{\N}{{\mathcal N}}

\newcommand{\HH}{{\mathcal H}}

\newcommand{\Oo}{{\mathcal O}}
\newcommand{\U}{{\mathcal U}}
\newcommand{\R}{{\mathcal R}}

\newcommand{\Id}{\operatorname {Id}}
\newcommand{\Cat}{\operatorname {Cat}}

\newcommand{\BiMlt}{\operatorname {BiMlt}}
\newcommand{\RepBM}{\operatorname {BiRepBiMlt}} 
\newcommand{\MBic}{\operatorname {MonBicat}}

\newcommand{\id}{\operatorname {id}}
\newcommand{\Lax}{\operatorname{\mathbb L ax}}

\newcommand{\Bicat}{\operatorname {Bicat}}
\newcommand{\Fun}{\operatorname{\mathbb F un}}

\newcommand{\Fi}{\varphi} 
\newcommand{\Kl}{\operatorname {Kl}} 

\newcommand{\Epsilon}{\varepsilon}

\def\ul{\underline}
\def\dul#1{\underline{\underline{#1}}}

\newcommand{\exlabel}[1]{\label{ex:#1}}

\newcommand{\eqlabel}[1]{\label{eq:#1}}
\newcommand{\equref}[1]{(\ref{eq:#1})}
\newcommand{\thlabel}[1]{\label{th:#1}}
\newcommand{\thref}[1]{Theorem~\ref{th:#1}}
\newcommand{\delabel}[1]{\label{de:#1}}
\newcommand{\deref}[1]{Definition~\ref{de:#1}}
\newcommand{\prlabel}[1]{\label{pr:#1}}
\newcommand{\prref}[1]{Proposition~\ref{pr:#1}}
\newcommand{\colabel}[1]{\label{co:#1}}
\newcommand{\coref}[1]{Corollary~\ref{co:#1}}
\newcommand{\rmlabel}[1]{\label{rm:#1}}
\newcommand{\rmref}[1]{Remark~\ref{rm:#1}}
\newcommand{\selabel}[1]{\label{se:#1}}
\newcommand{\seref}[1]{Section~\ref{se:#1}}
\newcommand{\sslabel}[1]{\label{ss:#1}}
\newcommand{\ssref}[1]{Subsection~\ref{ss:#1}}
\newcommand{\ssslabel}[1]{\label{sss:#1}}
\newcommand{\sssref}[1]{Subsection~\ref{sss:#1}}

\newcommand{\plabel}[1]{\label{p:#1}}
\newcommand{\pref}[1]{P.~\ref{p:#1}}

\newcommand*{\threefrac}[3]{%
  \begin{array}{@{\,}c@{\,}}%
    #1\\
    \hline
    #2\\
    \hline
    #3%
  \end{array}%
}

\begin{document}

\title{Bi-representable bi-multicategories } 

\author{Bojana Femi\'c \vspace{6pt} \\
{\small Mathematical Institite of  \vspace{-2pt}}\\
{\small Serbian Academy of Sciences and Arts } \vspace{-2pt}\\
{\small Kneza Mihaila 36,} \vspace{-2pt}\\
{\small 11 000 Belgrade, Serbia} \vspace{-2pt}\\
{\small femicenelsur@gmail.com}}

\date{}
\maketitle

\begin{abstract}
We prove a two-dimensional version of Hermida's equivalence result on representable multicategories and classify 
its stricter versions. The obtained three-dimensional equivalence of the respective totality categories is of 
categories enriched over the multicategory $\mathsf{Bicat}$ of Verity. There is a basic enriched equivalence with 
respect to inherently lax monoidal pseudofunctors, and a one with respect to pseudo-monoidal pseudofunctors, related 
to the latter via the forgetful enriched functors. We extend Slattery's result on a bi-multicategory structure of the 
Kleisli bicategory of a 2-multicategory encompassing bi-representability. 
\end{abstract}


\maketitle

\tableofcontents

\section{Introduction}

It is very well known that there is an equivalence between monoidal categories and representable multicategories. The two 
classes of structures even make biequivalent 2-categories, as shown in \cite{Herm}. In the constructions of 
representable multicategories known in the literature (for 2-categories in \cite[Section I,4]{Gray}, for double categories 
in \cite{Fem:Gray} and  \cite{Fem:Prem}, in  \cite{BL} for Gray-categories), one describes the functors in $\Fun(A, \Fun(B,C))$, for the 
categories in question $A,B,C$, then one characterizes them as elements of $q\x\Fun(A\times B,C)$, which Gray called {\em quasi-functors}, 
and ultimately proves that the latter is in bijection with $\Fun_{st}(A\ot B, C)$ for a naturally occurring product $A\ot B$:
\begin{equation}\eqlabel{isoms-multic}
 \Fun(A, \Fun(B,C))\iso q\x\Fun(A\times B,C) \iso \Fun_{st}(A\ot B, C).
\end{equation}
Here $\Fun_{st}(A\ot B, C)$ is the set of the corresponding strict functors. Such obtained product $A\ot B$ is called {\em Gray product}, in credit to the first above mentioned construction from \cite[Section I,4]{Gray}, if the (horizontal) 1-cells in $\Fun(B,C)$ are {\em natural} transformations, and 
{\em funny product}, if they are {\em unnatural} transformations (\cite{FLK}). From the above listed constructions of representable 
multicategories the one from \cite{Fem:Prem} is of the latter kind, while the rest are of the former. In both cases the multicategory is represented by the occurring product. 

As a matter of fact, these bijections turn out to be isomorphisms in the corresponding ambient category $\Oo$ of the categories $A,B,C$. 
This is not enough to say that one has {\em $\Oo$-enriched multicategories}, as one should verify if the substitutions of the constructed multicategories are functors, {\em i.e.} 1-cells in $\Oo$. Though, the fact that the above isomorphisms are of higher dimension sets the ground for making a natural question, if the obtained Gray product is even a monoidal product in the 2-category whose underlying category 
is $\Oo$. More comprehensively, the appropriate question is then if Hermida's bi-equivalence from \cite[Theorem 9.8]{Herm} has a 
two-dimensional analogue. 

\smallskip

Two-dimensional multicategories are {\em bi-multicategories}, which by definition are $\Cat_2$-multicategories, being $\Cat_2$ 
the 2-category of categories. (The strict version of the notion are 2-multicategories, which first appeared in \cite{Shul-sl} 
and \cite[Definition 5.1]{LSR}. They are $\Cat_1$-multicategories, where $\Cat_1$ is the category of categories. The weaker notion, 
as well as that of bi-representable bi-multicategories were studied in \cite{Phil}.) 
The first motivating conjecture for this research was thus to prove that monoidal bicategories are in an equivalence correspondence with 
bi-representable bi-multicategories. Moreover, these two classes of structures should organize into 
three-dimensional categories. Rather than considering a fully weak totality tricategory, taking inspiration from \cite{Camp}, we 
conjectured the equivalence of the two classes of structures as of 
$\mathsf{Bicat}$-enriched categories, where $\mathsf{Bicat}$ is the multicategory of bicategories from \cite[Section 1.3]{Ver}. 
The conjectured result we proved to be true in 
\thref{main}. In the subsequent two corollaries, in \ssref{Stricter} we formulate equivalences in the two semi-strict cases 
(for Gray monoidal and for bicategorically monoidal 2-categories) and 
in the fully strict case (of strictly monoidal 2-categories). Table \ref{table3} resumes our findings in a concise way. 

The above isomorphisms \equref{isoms-multic} do not include the Cartesian product. There are at least two occurrences of the 
Cartesian product in relation to representability of multicategories in the literature, though with the following twist. Instead 
of representability, one has bicategorical bi-representability via the Cartesian product of the 
multicategories of pseudodouble categories (in \cite[Appendix A]{Camp}, with pseudodouble functors) and of double categories 
(in \cite{Fem:Bif} and \cite{Fem:Gray}, with a subclass of lax double functors.) 
We interpret these two examples through the lens of the above semi-strict equivalence in \coref{mixed}. 


\smallskip

Higher transforms for bi-multicategories were introduced in \cite{Phil}. For bi-representable bi-multicategories we prove in 
\ssref{fun-bicat} that the higher transforms for their underlying bi-multicategories respect the bi-representability in a 
{\em lax-representable way} (see \prref{psf birep} - \prref{modif}). For the pseudofunctors $\F:\M\to\N$ this means that there are 
 pseudonatural transformations $\xi_n: m(\F(-),..,\F(-))\to \F(m(-,..,-))$ and a unary 1-cell $\xi_0: m(-)\to\F(m(-))$ in $\N$ 
constrained by certain invertible modifications, where $m(a_1,...,a_n)$ denotes the universal object in the bi-representability. 
When the $\xi_n$'s are pseudonatural equivalences, the pseudofunctors respect bi-representability in a {\em pseudo-representable way} 
(\coref{fully mon}). Moreover, equivalences of $\mathsf{Bicat}$-categories of monoidal bicategories and bi-representable bi-multicategories 
with respect to lax-representable and pseudo-representable pseudofunctors commute with the forgetful 2-functors (see the commutative 
diagram \equref{commut diag}). 

\smallskip

A bi-multicategory structure was shown to exist on the Kleisli bicategory $\Kl_T(\ul{\M})$ for a 2-multicategory $\M$ and a 
pseudocommutative pseudomonad $T$ on it in \cite[Theorem 6.8]{Slat}, here $\ul{\M}$ denotes the underlying 2-category of $\M$. 
The manuscript has not contemplated bi-representable bi-multicategories. By including bi-representability and applying our 
equivalence result of the $\mathsf{Bicat}$-enriched categories, in \thref{Slat} we extend the latter result, both to 
 monoidal bicategories and bi-representable bi-multicategories to obtain a structure of a bi-representable bi-multicategory 
on the corresponding Kleisli bicategory. 

\medskip

The article is structured as follows. In the next section we recall multicategories (representability, closedness, enrichment 
over them, examples), the notions of 2- and bi-multicategories and of 2- and bi-representability, and we recollect the data defining 
a monoidal bicategory. In Section 3 we show how a bi-representable bi-multicategory yields a monoidal bicategory. Section 4 is the 
technical core of the paper, in it we construct a bi-representable bi-multicategory out of a monoidal bicategory. (We do not use 
strictification result for bicategories, which would have by large simplified the computations, as we will draw the conclusion about 
the strict case in the final section.) In section 5 we prove equivalence of the enriched categories: we prove that the assignments in the previous two sections give an equivalence of the two notions, we describe the functor bicategory $\BiMlt(\M,\N)$ of bi-multicategories 
applied to bi-representable $\M,\N$ showing that the pseudofunctors are automatically {\em lax-representable}, we show that monoidal bicategories make a $\mathsf{Bicat}$-enriched category $\MBic$, we prove that the respective hom-bicategories are bi-equivalent, and 
conclude that $\RepBM$ is $\mathsf{Bicat}$-enriched, too, and equivalent to $\MBic$ as a $\mathsf{Bicat}$-category. The final section we dedicate to applications: we study the strict and two semi-strict versions of the equivalence of enriched categories (\ssref{Stricter}), 
we relate the studied equivalence (with respect to lax-representable and lax monoidal pseudofunctors) to the one contemplating pseudo-representable, respectively, pseudo monoidal pseudofunctors in \ssref{fully mon}, and conclude on bi-representable Kleisli bi-multicategories 
upgrading Slattery's result. 

\smallskip

Some parts of the paper are technically and computationally involved (especially Sections 4 and partially 5.2) and this hinders both the 
writing and the reading from being fluent. As we do want to leave a record of how the computations are done, we presented the necessary technicalities, and we invite the reader to skip the technical parts when he/she feels like to follow the flow of the construction more 
fluently. 

\smallskip

{\bf Notational conventions.} In a bicategory $\B$ the horizontal composition we denote by $[-\vert-]$ read in the lexicographical order, 
and also by the usual $\circ$ notation and concatenation, read in the opposite order. The vertical composition we denote by fractions 
$\frac{\bullet}{\bullet}$. When the annotation of the pastings of 2-cells via the horizontal and vertical compositions gets involved, 
to simplify we will use the notation $\sigma\diamond\tau$ (read from right to left)
to denote the concatenations of the 2-cells $\sigma$ and $\tau$  in the sense that will be obvious from the context.

\section{Multicategories, 
monoidal bicategories and bi-multicategories}

In this section we recall the notions from the title, as well as the notions of enrichment over a multicategory 
and of a 2-representable 2-multicategory. 

\subsection{Multicategories and some of their examples} \sslabel{defex}

We start by recalling multicategories and what it means when they are closed. 

\begin{defn}
A multicategory consists of:
\begin{itemize} 
\item a collection of objects, \vspace{-0,2cm}
\item for each list of objects $a_1,...,a_n$ for $n\geq 0$ and an object $b$, a set $\M_n(a_1,...,a_n;b)$, \vspace{-0,2cm}
\item for each object $a$ an element $1_a\in\M_1(a;a)$, \vspace{-0,2cm}
\item for all lists $\crta{a_i}, i=1,...,n, b_1,...,b_n$ and objects $c$ a function, called {\em substitution}:
$$\M_n(b_1,...,b_n,c)\times\Pi_{j=1}^{n}\M_{k_i}(\crta{a_i};b_i)\to\M_{\sum_{i=1}^n k_i}(\crta{a_1},...,\crta{a_n};c)$$
$$(g, f_1,...,f_n)\mapsto g\circ(f_1,...,f_n)$$
satisfying a natural associativity and two identity axioms. Here $\crta{a_i}$ is a short annotation for a list of objects, 
and the elements of $\M_n(a_1,...,a_n;b)$ are called {\em multimaps}. 
\end{itemize}
\end{defn}

We will recall a diagrammatic representation of associativity in a multicategory in \equref{assoc bi-mc}. 

\begin{defn}
A multicategory $\M$ is said to be {\em left closed} if for all objects $b,c$ there is an object $[b,c]$ and a binary map 
$e_{b,c}:[b,c], b\to c$ so that the induced functions 
\begin{equation} \eqlabel{l-closed}	
\M_n(\crta a;[b,c])\to\M_{n+1}(\crta a,b;c)
\end{equation} 
are bijections for all $n\geq 0$. 

If there exists an object $\{b,c\}$ and a binary map $\crta{e_{b,c}}:b,\{b,c\}\to c$ for which the induced functions 
$\M_n(\crta a;\{b,c\})\to\M_{n+1}(b,\crta a;c)$ are bijections for all $n\geq 0$, then $\M$ is said to be {\em right closed}. 

$\M$ is called {\em biclosed} if it is both left and right closed. 
\end{defn}

\begin{defn}
A multicategory $\M$ is said to be {\em representable} if for any list of objects $\crta a=a_1,..,a_n$ there exists an object %
$m(a_1,..,a_n)$ and a multimap $j_{\crta a}:a_1,..,a_n\to m(a_1,..,a_n)$ that for every object $c$ and lists $\crta x, \crta b$ (of lengths 
$k$ and $l$, respectively) induces bijections 
$$\M_{k+1+l}(\crta x,m(a_1,..,a_n),\crta b;c)\to\M_{k+n+l}(\crta x,a_1,..,a_n,\crta b;c).$$
(Left representability is defined with $k=0$, and right representability with $l=0$.)
\end{defn}

We now recall several multicategories studied in the literature. 

\begin{ex} 
A symmetric closed multicategory of 2-categories was constructed in \cite[Section I,4]{Gray}. 
\end{ex}

\begin{ex} 
A symmetric closed multicategory of bicategories was constructed in \cite[Section 1.3]{Ver}. We will denote it by $\mathsf{Bicat}$. 
\end{ex}

\begin{ex} 
The symmetric closed multicategory $\mathsf{PsDbl}$ of pseudodouble categories (and pseudodouble quasi-functors of $n$-variables) 
was introduced in \cite[Appendix]{Camp}. 
The multicategory $\mathsf{Bicat}$ from the previous example is a sub-multicategory of $\mathsf{PsDbl}$. 
\end{ex}

\begin{ex} \exlabel{Dbl Gray}
In \cite{Fem:Gray} two representable symmetric closed multicategories of double categories were discussed. They are of Gray type and give rise to the Gray monoidal categories 
of double categories $(Dbl, \ot^{st}_{o-l, l-o})$ and $(Dbl, \ot^{st}_{ps})$. (In both cases the double categories of multimaps have for objects double functors, 
$(Dbl, \ot^{st}_{o-l, l-o})$ has for 1h-cells horizontal oplax transformations, for 1v-cells vertical lax transformations and for 2-cells the corresponding 
modifications, whereas $(Dbl, \ot^{st}_{ps})$ has for 1h-cells horizontal pseudonatural transformations, for 1v-cells vertical pseudonatural transformations 
and for 2-cells the corresponding modifications.) Correspondingly, we may denote one multicategory as $\mathsf{Dbl^{st}_{o-l}}$ and the other one as 
$\mathsf{Dbl^{st}_{ps}}$. 
\end{ex}

\begin{ex} 
The two representable symmetric closed multicategories of double categories from \cite{Fem:Prem} $\mathsf{Dbl_f}$ and $\mathsf{Dbl_m}$ 
are of funny type and give rise to the symmetric funny monoidal categories of double categories $(Dbl, \Box_f)$ and $(Dbl, \Box)$, respectively. In the former the objects of the double categories of multimaps 
are {\em purely funny functors}, whereas for the latter they are {\em mixed funny functors}. 
\end{ex}



\begin{ex} 
In \cite{BL} the authors constructed {\em $4$-ary multicategories} of Gray-categories (in \cite[Theorem 4.8]{Lob} it was proved that, 
under the assumption on left representability, there in as equivalence between multicategories and short multicategories). 
\end{ex}

\subsection{Enrichment over a multicategory} 

In this section we recall enrichment over a multicategory. 
Let $\M$ be a multicategory and let $\ul{\M}$ denote its underlying category.

\begin{defn} \delabel{mult-mult}  \cite[Definition 6.1.1]{JY} 
An {\em $\M$-category} or a {\em category enriched over $\M$} consists of:
\begin{itemize} 
\item a collection of objects, \vspace{-0,2cm}
\item for objects $a, b$ an object $\C(a,b)$ in $\M$, \vspace{-0,2cm}
\item for each object $a$ an object $1_a\in\C(a,a)$, \vspace{-0,2cm}
\item for objects $a, b, c$ a binary morphism in $\M$:
$$\C(b,c),\C(a,b)\to\C(a, c)$$
satisfying a natural associativity and two identity axioms in $\ul{\M}$. 
\end{itemize}
\end{defn}

An interesting class of $\M$-categories is given by the following known result, which is referred to as {\em self-enrichment}. 

\begin{prop} \prlabel{Yau} \cite[Theorem 9.1.7]{JY} \\
Let $\M$ be a closed multicategory and let $\ul{\M}$ denote its underlying category. Then $\ul{\M}$ is an $\M$-category with 
the hom-objects $\ul{\M}(a,b):=\M_1(a;b)$ and 
the composition $comp:\ul{\M}(b,c),\ul{\M}(a,b)\to\ul{\M}(a, c)$ the {\em partner} of the ternary multimap 
$$\ul{\M}(b,c), \ul{\M}(a,b), a \stackrel{\langle 1, ev\rangle}{\to} \ul{\M}(b,c), b \stackrel{ev}{\to} c$$     
that is, the following diagram commutes 
\begin{equation} \eqlabel{closed-enriched}	
	\scalebox{0.86}{
\bfig
 \putmorphism(-150,470)(1,0)[\M( b;c),\M( a; b), a`\M( a; c), a`\langle comp,1\rangle]{1250}1a
\putmorphism(-160,470)(0,-1)[\phantom{Y_2}` ` \langle 1, ev\rangle]{380}1l
\putmorphism(-160,430)(0,-1)[\phantom{Y_2}`\M( b;c),b`]{380}0l
\putmorphism(1150,470)(0,-1)[\phantom{Y_2}` `  ev]{380}1r
\putmorphism(1150,430)(0,-1)[\phantom{Y_2}`c. `]{380}0r
 \putmorphism(0,70)(1,0)[\phantom{Y_2}`\phantom{Y_2} `ev]{1150}1b
\efig}
\end{equation}
\end{prop} 


\smallskip

All the examples of multicategories form \ssref{defex} are closed and hence \prref{Yau} applies to them. 

\bigskip

\begin{rem}
The condition on the binary composition map in \prref{Yau} is crucial. Namely, as observed in \cite{Camp}, bicategories form a category enriched over the multicategory $\mathsf{Bicat}$ of bicategories (Verity proved in \cite[Section 1.3]{Ver} that $\mathsf{Bicat}$ is closed, then apply \prref{Yau}). On the other hand, 
there is a natural structure of a {\em ``category'' enriched over the 3-multicategory of bicategories} on the category of bicategories. 
(The semicolon in ``category'' alludes to the fact that one has indeed a higher dimensional category.) Namely, one defines a quasi-pseudofunctor on 
hom-bicategories $c:\Ps(\B,\C), \Ps(\A,\B)\to\Ps(\A,\C)$ via: 1) the two obvious pseudofunctors 
$$c(\G,-):\Ps(\A,\B)\to\Ps(\A,\C) \quad\text{and}\quad c(-,\F):\Ps(\B,\C)\to\Ps(\A,\C)$$
for objects, {\em i.e.} pseudofunctors $\F\in\Ps(\A,\B)$ and $\G\in\Ps(\B,\C)$, and 2) a 2-cell in $\Ps(\A,\C)$, {\em i.e.} modification 
\[
\begin{tikzcd}
	\A && \B && \C \\
	\A && \B && \C
	\arrow["\F", curve={height=-24pt}, from=1-1, to=1-3]
	\arrow[""{name=0, anchor=center, inner sep=0}, "\G", curve={height=-24pt}, from=1-3, to=1-5]
	\arrow[""{name=1, anchor=center, inner sep=0}, curve={height=24pt}, from=1-3, to=1-5]
	\arrow[""{name=2, anchor=center, inner sep=0}, curve={height=-24pt}, from=2-1, to=2-3]
	\arrow[""{name=3, anchor=center, inner sep=0}, "{\F'}"', curve={height=24pt}, from=2-1, to=2-3]
	\arrow["{\G'}"', curve={height=24pt}, from=2-3, to=2-5]
	\arrow["\omega", between={0.2}{0.8}, Rightarrow, nfold, from=0, to=1]
	\arrow["\theta", between={0.2}{0.8}, Rightarrow, nfold, from=2, to=3]
\end{tikzcd}
\Rrightarrow
\begin{tikzcd}
	\A && \B && {} \\
	\A && \B && \C
	\arrow[""{name=0, anchor=center, inner sep=0}, "\F", curve={height=-24pt}, from=1-1, to=1-3]
	\arrow[""{name=1, anchor=center, inner sep=0}, curve={height=24pt}, from=1-1, to=1-3]
	\arrow["\G", curve={height=-24pt}, from=1-3, to=1-5]
	\arrow["{\F'}"', curve={height=24pt}, from=2-1, to=2-3]
	\arrow[""{name=2, anchor=center, inner sep=0}, "{\G'}"', curve={height=24pt}, from=2-3, to=2-5]
	\arrow[""{name=3, anchor=center, inner sep=0}, curve={height=-24pt}, from=2-3, to=2-5]
	\arrow["\theta", between={0.2}{0.8}, Rightarrow, nfold, from=0, to=1]
	\arrow["\omega"', between={0.2}{0.8}, Rightarrow, nfold, from=3, to=2]
\end{tikzcd}
\]
for 1-cells/transformations $\theta\in\Ps(\A,\B)$ and $\omega\in\Ps(\B,\C)$. Then the associativity for the composition $c$ consists of three binary, coming down 
to three unary associativity transformations $\alpha_{\HH,\G,-}, \alpha_{\HH,-,\F},\alpha_{-,\G,\F}$ that are identity on objects (icons), but their 2-cell components 
 $\alpha_{\HH,\G,\omega}, \alpha_{\HH,\theta,\F},\alpha_{\zeta,\G,\F}$ are non-trivial. The transformation component  $\alpha_{\HH,\G,\omega}$ is a 2-cell/modification, which when evaluated at an object $A\in\A$ gives a 2-cell in $\D$ (associativity concerns a ternary map $\Ps(\C,\D), \Ps(\B,\C), \Ps(\A,\B)\to
\Ps(\A,\D)$) whose domain and codomain differ up to the associativity of the composition in the bicategory $\D$, so it can not be the identity. The similar happens with the pentagonator modification for the associativity we are examining. So, one needs a fully 3-dimensional ambient multicategory for this enrichment. What differs in the two enrichment constructions is that {\em in the latter case the composition $c$ does not obey the condition \equref{closed-enriched}}. Though, if one deals with 2-categories instead, one obtains strict associativity and an enrichment over the multicategory $\mathsf{Bicat}$ of bicategories. 
\end{rem}



\subsection{Monoidal bicategories}

As a one-object tricategory a monoidal bicategory $\B$ consists of: 
\begin{enumerate}
\item  \axiom{MB1} identity 1-cells $id_A$ on objects $A$ and identity 2-cells $\Id_f$ on 1-cells $f$;
\item \axiom{MB2} horizontal composition of 1- and 2-cells $[-\vert-]$ (in lexicographical order), associative and unital up to isomorphisms $\alpha, \lambda, \rho$, satisfying a pentagon and a triangle, the interchange law on 2-cells strictly, and such that $[\Id_f\vert \Id_g]=\Id_{[f\vert g]}$;
\item \axiom{MB3} strictly associative and unital vertical composition of 1- and 2-cells; 
\item \axiom{MB4} a pseudofunctor $-\ot-:\B\times\B\to\B$, in particular, there are natural isomorphisms 
$$\xi:[f\ot g\vert f'\ot g']\Rightarrow [f\vert f']\ot[g\vert g']$$
$$\xi_0: id_{A\ot B}\Rightarrow id_A\ot id_B$$
satisfying a hexagon and two squares (we will call $xi$ the ``monoidality interchange''), and on 2-cells we have 
$$[\alpha\vert\alpha']\ot[\beta\vert \beta']= [\alpha\ot \beta\vert \alpha'\ot \beta']$$
$$ \Id_{f\ot g}= \Id_f\ot \Id_g;$$
\item \axiom{MB5} there is a unit object $I$, a 1-cell $\iota:I\to I$ and invertible 2-cells $\nabla: [\iota\vert\iota]\Rightarrow\iota$ 
and $u:id_I\Rightarrow\iota$ satisfying the axioms of a pseudofunctor $1\to\B$;  
\item \axiom{MB6} 
a pseudonatural equivalence $a: \ot(\ot\times 1)\Rightarrow \ot(1\times\ot)$, meaning that there are quasi-invertible 1-cells $a_{A,B,C}:
(A\ot B)\ot C\to  A\ot(B\ot C)$ 
quasi-natural in $A,B,C$, meaning that there are invertible 2-cells $a_{f,g,h}:[(f\ot g)\ot h \vert a_{A',B',C'}]\Rightarrow [a_{A,B,C}\vert f\ot (g\ot h)]$ 
(for $a_{A,B,C}$, and similarly for its quasi-inverse) natural in $f,g,h$;
\item \axiom{MB7} pseudo natural equivalences $l: \ot\comp(I\times\Id_{\B})\to\Id_{\B}$ and $r: \ot\comp(\Id_{\B}\times I) \to \Id_{\B}$ 
(there are 1-cells $l_A: I\ot A\to A$ and $r_A: A\ot I\to A$ and their quasi-inverses, quasi-natural in $A$, {\em i.e.} there are  
 invertible 2-cells $l_f:[id_I\ot f\vert l_B]\Rightarrow [l_A\vert f]$ and $r_f:[f\ot id_I\vert r_B]\Rightarrow [r_A\vert f]$ (and similarly for 
their quasi-inverses), natural in $f$;
\item \axiom{MB8} invertible modifications $\pi,\mu, \tilde\lambda$ and $\tilde\rho$, whereby the pentagon for $a$ commutes up to $\pi$,  $\mu$ relates 
$a$ with $l$ and $r$, $\lambda$ relates $a$ with $l$, and $\rho$ relates $a$ with $r$, so that for $\pi$ hold the axioms (TA1) for a 
non abelian 4-cocycle condition, and (TA2) and (TA3) for left and right normalization for the 4-cocycle $\pi$. 
\end{enumerate}

\subsection{2-multicategories, bi-multicategories and bi-representability} \sslabel{2-mult}

The notion of 2-multicategories goes back to \cite{Shul-sl}, \cite[Definition 5.1]{LSR}. 
A 2-multicategory $\M$ besides objects consists of {\em categories} $\M_n(a_1,...,a_n;b)$, 1-cells $1_a\in\M_1(a;a)$ for all objects $a$ 
and an (associative and unital) substitution {\em functor} 
\begin{equation} \eqlabel{cat-multi}
\M_n(b_1,...,b_n,c)\times\Pi_{i=1}^{n}\M_{k_i}(\crta{a_i};b_i)\to\M_{\sum_{i=1}^n k_i}(\crta{a_1},...,\crta{a_n};c).
\end{equation}

This is a 2-dimensional structure: it has objects, and objects and morphisms of the categories $\M_n(a_1,...,a_n;b)$ are the 1- and 2-cells of $\M$. 
Composition of morphisms in  $\M_n(a_1,...,a_n;b)$ induces the vertical composition of 2-cells in $\M$, and the substitution induces the horizontal 
composition of 1-cells (actually their only composition) and 2-cells. The horizontal composition is strictly associative and unital. The difference between $\M$ 
and a 2-category is that its 1-cells are multimaps and its 2-cells are morphisms of those. In other words, a 2-multicategory has an underlying 2-category 
whose hom-categories are the unary hom-categories of $\M$. A 2-multicategory is a $\Cat_1$-multicategory.

\begin{defn}\delabel{2-rep}
A 2-multicategory $\M$ is said to be {\em 2-representable} if for any list of objects $\crta a=a_1,..,a_n$ there exists an object %
$m(a_1,..,a_n)$ and a multimap $j_{\crta a}:a_1,..,a_n\to m(a_1,..,a_n)$ that for every object $c$ and lists $\crta x, \crta b$ (of lengths 
$k$ and $l$, respectively) induces isomorphisms of categories 
$$\M_{k+1+l}(\crta x,m(a_1,..,a_n),\crta b;c)\to\M_{k+n+l}(\crta x,a_1,..,a_n,\crta b;c).$$
\end{defn}

The notion of a {\em bi-multicategory} was defined in \cite[Definition 4.2.29]{Phil}. In contrast to a 2-multicategory, its substitution has 
a natural family of invertible 2-cells governing associativity and unitality, subject to a pentagon and a triangle law. In other words, 
a bi-multicategory is a multicategory enriched in the 2-category $\Cat_2$ of categories. It possesses an underlying bicategory. 
A {\em bi-representable} bi-multicategory was defined in \cite[Definition 4.2.29]{Phil}. We recall this definition here.

\begin{defn} \delabel{Phil}
A bi-multicategory $\M$ is said to be {\em bi-representable} if for any list of objects $\crta a=a_1,..,a_n$ there exists an object 
$m(a_1,..,a_n)$ that we call the {\em universal object} and a multimap $j_{\crta a}:a_1,..,a_n\to m(a_1,..,a_n)$  that we call the {\em universal arrow}, 
so that
\begin{enumerate}
\item universal arrows are closed under composition, 
\item for every object $b$ the universal arrow $j_{\crta a}$ induces adjoint equivalences 
\begin{equation} \eqlabel{Phil rep}
\M_n(a_1,..,a_n;b)\to\M_1(m(a_1,..,a_n);b)
\end{equation}
specified by a choice of a universal invertible 2-cell $\Epsilon_{\crta a}$. 
\end{enumerate}
\end{defn}

We will use an equivalent definition that is more in accordance to 
the construction of the representability of the multicategories $\mathsf{Dbl^{st}_{o-l}}$ and $\mathsf{Dbl^{st}_{ps}}$ in \cite{Fem:Gray}. 

\begin{defn} 
A bi-multicategory $\M$ is said to be {\em bi-representable} if for any list of objects $\crta a=a_1,..,a_n$ there exists an object 
$m(a_1,..,a_n)$ and a multimap $j_{\crta a}:a_1,..,a_n\to m(a_1,..,a_n)$  
that for every object $c$ and lists $\crta x, \crta b$ (of lengths $k$ and $l$, respectively) induces adjoint equivalences 
\begin{equation} \eqlabel{strong rep}
\M_{k+n+l}(\crta x,a_1,..,a_n,\crta b;c)\to\M_{k+1+l}(\crta x,m(a_1,..,a_n),\crta b;c)
\end{equation}
specified by a choice of a universal invertible 2-cell $\Epsilon_{\crta x, \crta a, \crta b}$. 
The object $m(a_1,..,a_n)$ and the  multimap $j_{\crta a}$ we call the {\em universal object} and {\em strong universal arrow} 
(for the corresponding list). 
\end{defn}

This means that for every multimap $f: \crta x,a_1,..,a_n,\crta b\to c$ there is a multimap $\hat f: 
\crta x,m(a_1,..,a_n),\crta b\to c$ and an invertible 2-cell 
$\Epsilon_{\hat f}: \hat f\comp \langle 1_{\crta x},j_{\crta a},1_{\crta b}\rangle \Rightarrow f$. 
The universality means that for any 1-cell $\xi: \crta x,m(a_1,..,a_n),\crta b\to c$ and any invertible 2-cell 
$\sigma: \xi \comp \langle 1_{\crta x},j_{\crta a},1_{\crta b}\rangle \Rightarrow f$ 
there exists a unique invertible 2-cell $\sigma^\dagger: \xi \Rightarrow \hat f$ so that
\begin{equation} \eqlabel{sigma-dagger}
\frac{[\Id_{\langle 1_{\crta x},j_{\crta a},1_{\crta b}\rangle}\vert\sigma^\dagger]}{\Epsilon_{\hat f}}= \sigma.
\end{equation}

\begin{rem} \rmlabel{unique inv 2-cell}
The invertibility of the universal 2-cell in \cite[Definition 4.2.29]{Phil} was not explicitly stated, though the invertibility of the 2-cells in 
the three above definitions will be crucial for our constructions ({\em e.g.} in \prref{ksi} and the rest of \ssref{fun-bicat}). 
Namely, having an $n$-ary multimap $f:a_1,..,a_n\to c$ and two 1-cells $\xi_1, \xi_2: m(a_1,..,a_n)\to c$ with their respective invertible 2-cells 
$\sigma_i:\xi_i\circ j\Rightarrow f, i=1,2$, we will construct a 2-cell $\xi_1\Rightarrow\xi_2$ as the (necessarily unique invertible) 2-cell 
$(\sigma_2^\dagger)^{-1}\circ\sigma_1^\dagger$ obtained from the unique invertible 2-cells $\sigma_i^\dagger:\xi_i\Rightarrow\Omega(f)$ satisfying 
the universal property. 
\end{rem}




\begin{prop}
The two definitions above of a  bi-representable bi-multicategory are equivalent. 
\end{prop}

\begin{proof}
The proof is fully analogous to the one of \cite[Proposition 8.5]{Herm} for multicategories. Let $\M$ be a bi-multicategory as in the latter definition, 
we should prove that a composite of strong universal arrows satisfies the equivalence \equref{Phil rep}. 
Let us write $\ot x=m(x_1,..,x_n), \ot y =m(y_1,..,y_l)$ and $\crta x=x_1,..,x_n, \crta y=y_1,..,y_l$. We want to prove that the composite 
$j_{\ot x,\ot y}\circ\langle j_{\crta x}, j_{\crta y}\rangle$ is universal for $\crta x,\crta y$. Applying each of the factors gradually we 
obtain a chain of equivalences 
$$\M(\crta a, \crta x,\crta y, \crta b;-)\stackrel{j_{\crta x}}{\simeq}\M(\crta a, \ot x,\crta y, \crta b;-)\stackrel{j_{\crta y}}{\simeq}
\M(\crta a, \ot x,\ot y, \crta b;-)\stackrel{j_{\ot x,\ot y}}{\simeq}\M(\crta a, (\ot x)\ot(\ot y), \crta b;-)$$
proving that the composite fulfills the desired property. 

Knowing that the composites of universal arrows are universal arrows, we need to prove the equivalence \equref{strong rep}. 
Given that $j_{\ot x, \ot a, \ot y}\langle j_{\crta x}, j_{\crta a}, j_{\crta y} \rangle$ is a universal arrow for 
$\crta x, \crta a, \crta y$, for a multimap $f:\crta x,a_1,..,a_n,\crta y\to c$ there is a 1-cell $\hat f: m(\ot x, \ot a, \ot y)\to c$ 
in the right-hand side of \equref{Phil rep} 
and a 2-cell $\Epsilon:\hat f j_{\ot x, \ot a, \ot y}\langle j_{\crta x}, j_{\crta a}, j_{\crta y} \rangle\Rightarrow f$ satisfying 
the universal property. Given that $\langle j_{\crta x}, j_{\crta a}, j_{\crta y} \rangle=\langle j_{\crta x}, id_{\crta a}, j_{\crta y} \rangle 
\langle 1_{\crta x}, j_{\crta a}, 1_{\crta y} \rangle$, let us show that the pair consisting of the 1-cell $\hat f 
 j_{\ot x, \ot a, \ot y}\langle j_{\crta x}, id_{\ot a}, j_{\crta y} \rangle$ in the right-hand side of \equref{strong rep} and the same 2-cell 
$\Epsilon$ fulfills the desired universal property. For that purpose, let $\xi: \crta x, \ot a, \crta y\to c$ be another 1-cell and 
$\sigma: \xi\langle 1_{\crta x}, j_{\crta a}, 1_{\crta y} \rangle\Rightarrow f$ a 2-cell, we want to find a unique invertible 2-cell 
$\hat f j_{\ot x, \ot a, \ot y} \langle j_{\crta x}, id_{\crta a}, j_{\crta y} \rangle\Rightarrow\xi$ which concatenated with $\sigma$ gives 
$\Epsilon$. From the equivalence \equref{Phil rep}, for $\crta x, \ot a, \crta y$ and the given $\xi: \crta x, \ot a, \crta y\to c$ we know 
that there exists a 2-cell $\omega:\hat f j_{\ot x, \ot a, \ot y} \langle j_{\crta x}, id_{\crta a}, j_{\crta y} \rangle\Rightarrow\xi$. 
By the same equivalence but for $\crta x, \crta a, \crta y$ and $f:\crta x, \crta a, \crta y\to c$ we have both pairs $(\hat f, \Epsilon)$ and 
$(\hat f, \sigma\diamond\omega)$ obeying the universal property, so it must be $\sigma\diamond\omega=\Epsilon$, that is, that $\sigma$ concatenated 
with $\omega$ equals $\Epsilon$, as desired. 
\qed\end{proof}

We will fix now some notational conventions. Consider the equivalences 
\begin{equation} \eqlabel{adj equiv}
\M_{n}(a_1,..,a_n;c)\to\M_1(m(a_1,..,a_n);c)
\end{equation}
$$\zeta \mapsto (\crta\zeta, \Epsilon_{\crta\zeta})$$
$$\hspace{-2,3cm}\alpha:\zeta\Rightarrow\zeta' \mapsto \crta\alpha.$$
Let us denote by $\Theta(\zeta \backslash \frac{\Epsilon_{\crta\zeta}}{\sigma} / \sigma^\dagger)$ the triangular prism (below left) expressing 
the corresponding equation  \equref{sigma-dagger} of 2-cells. 
Its top is the 2-cell $\Epsilon_{\crta\zeta}$, its bottom the 2-cell $\sigma$, its two sides are identity 2-cells (one along $\zeta$ and the other along 
$j_{\crta a}$), and the third side is the unique mediating 2-cell $\sigma^\dagger$, which is a morphism on the right-hand side in \equref{adj equiv}.  
(This prism can also be considered in the form of a triangular semi-open seashell (below right): its bottom and top are triangular surfaces made of 2-cells $\sigma$ 
and $\Epsilon_{\crta\zeta}$ (sharing two edges $\zeta$ and $j_{\crta a}$), which are connected by the unique mediating 2-cell $\sigma^\dagger: 
\xi\Rightarrow\crta\zeta$ (between the differing third edges of the surfaces $\Epsilon_{\crta\zeta}$ and $\sigma$) at the place of the semi-opening of the shell.) 
$$
\begin{tikzcd}
	{\crta a} && {m(\crta a)} \\
	{\crta a} & c & {m(\crta a)} \\
	& c
	\arrow["{j_{\crta a}}", from=1-1, to=1-3]
	\arrow[equals, nfold, from=1-1, to=2-1]
	\arrow[""{name=0, anchor=center, inner sep=0}, "\zeta"{description, pos=0.4}, from=1-1, to=2-2]
	\arrow[""{name=1, anchor=center, inner sep=0}, "{\crta\zeta}"{description, pos=0.3}, from=1-3, to=2-2]
	\arrow[equals, nfold, from=1-3, to=2-3]
	\arrow[""{name=2, anchor=center, inner sep=0}, "\zeta", from=2-1, to=3-2]
	\arrow[equals, nfold, from=2-2, to=3-2]
	\arrow[""{name=3, anchor=center, inner sep=0}, "\xi", from=2-3, to=3-2]
	\arrow["\Id"{description}, between={0.2}{0.8}, Rightarrow, nfold, from=0, to=2]
	\arrow["{\Epsilon_{\crta\zeta}}"'{pos=0.7}, between={0.2}{0.8}, Rightarrow, nfold, from=1-3, to=0]
	\arrow["{\sigma^\dagger}"{description}, between={0.2}{0.8}, Rightarrow, nfold, from=3, to=1]
\end{tikzcd}
\quad =\quad
\Theta(\zeta \backslash \frac{\Epsilon_{\crta\zeta}}{\sigma} / \sigma^\dagger)\quad=
\quad
\begin{tikzcd}
	{\crta a} && {m(\crta a)} \\
	\\
	& c
	\arrow[""{name=0, anchor=center, inner sep=0}, "{j_{\crta a}}", from=1-1, to=1-3]
	\arrow[""{name=1, anchor=center, inner sep=0}, "\zeta"'{pos=0.4}, from=1-1, to=3-2]
	\arrow[""{name=2, anchor=center, inner sep=0}, "{\crta\zeta}"{description, pos=0.3}, curve={height=12pt}, from=1-3, to=3-2]
	\arrow[""{name=3, anchor=center, inner sep=0}, "\xi", curve={height=-18pt}, from=1-3, to=3-2]
	\arrow["{\Epsilon_{\crta\zeta}}", between={0.2}{0.8}, Rightarrow, nfold, from=0, to=1]
	\arrow["{\sigma^\dagger}"', between={0.1}{0.9}, Rightarrow, nfold, from=3, to=2]
\end{tikzcd}
$$
The fact that a morphism $\alpha$ in $\M_{n}(a_1,..,a_n;c)$ induces a morphism $\crta\alpha$ in $\M_1(m(a_1,..,a_n);c)$ we will 
express by a triangular prism $P(\alpha\backslash\frac{\Epsilon_{\crta\zeta}}{\Epsilon_{\crta{\zeta'}}}/\crta\alpha)$ below, with lower and upper bases 
$\Epsilon_{\crta{\zeta'}}$ and $\Epsilon_{\crta\zeta}$, respectively, one side (morphism in $\M_{n}(a_1,..,a_n;c)$) the 2-cell $\alpha$, the ``induced'' side 
(the induced morphism in $\M_1(m(a_1,..,a_n);c)$) the 2-cell $\crta\alpha$, and the third side the identity on the connecting 1-cell 
$j_{\crta a}$. 
$$P(\alpha\backslash\frac{\Epsilon_{\crta\zeta}}{\Epsilon_{\crta{\zeta'}}}/\crta\alpha)\quad=\quad 
\begin{tikzcd}
	{\crta a} && {m(\crta a)} \\
	{\crta a} & c & {m(\crta a)} \\
	& c
	\arrow["{j_{\crta a}}", from=1-1, to=1-3]
	\arrow[equals, nfold, from=1-1, to=2-1]
	\arrow[""{name=0, anchor=center, inner sep=0}, "\zeta"{description, pos=0.4}, from=1-1, to=2-2]
	\arrow[""{name=1, anchor=center, inner sep=0}, "{\crta\zeta}"{description, pos=0.3}, from=1-3, to=2-2]
	\arrow[equals, nfold, from=1-3, to=2-3]
	\arrow[""{name=2, anchor=center, inner sep=0}, "\zeta", from=2-1, to=3-2]
	\arrow[equals, nfold, from=2-2, to=3-2]
	\arrow[""{name=3, anchor=center, inner sep=0}, "\xi", from=2-3, to=3-2]
	\arrow["\alpha"{description}, between={0.2}{0.8}, Rightarrow, nfold, from=0, to=2]
	\arrow["{\Epsilon_{\crta\zeta}}"'{pos=0.7}, between={0.2}{0.8}, Rightarrow, nfold, from=1-3, to=0]
	\arrow["{\crta\alpha}"{description}, between={0.2}{0.8}, Rightarrow, nfold, from=3, to=1]
\end{tikzcd}
$$
Note that the 2-cells on the right-hand sides of the prisms map from bottom to top, whereas the distribution of the cells in the annotation of prisms $Theta$ 
and $P$ follows the distribution of those cells in the diagrams. 

\medskip

We could now derive several properties of universal arrows and higher cells on them, but we defer them to \rmref{j and j'}.

\subsection{More data in bi-representability } \sslabel{more data bi-rep}

Before we end this section let us make the following observation that was not spelled out in \cite[Definition 4.2.29]{Phil}, the 
definition of a bi-representable bi-multicategory. The naturality of the adjoint equivalences \equref{Phil rep} and \equref{strong rep} 
entails the following. 

Firstly, the universal object $m_n(a_1,...,a_n):=m(a_1,...,a_n)$ for a list $a_1,...,a_n$ determines a pseudofunctor 
$$m_n:\M_n\to\M_1$$ 
where $\M_n$ is the bicategory whose objects are lists of length $n$ of objects in the multicategory $\M$, 1-cells are $n$-ary lists of 
unary morphisms in $\M$, and 2-cells are $n$-ary lists of unary 2-cells in $\M$. The pseudofunctor $m_1$ is the identity. 

Secondly, the equivalence functor $\M_{k+1+l}(\crta x,m(a_1,..,a_n),\crta b;c)\to\M_{k+n+l}(\crta x,a_1,..,a_n,\crta b;c)$ 
is a 1-cell component in the 2-category $\Cat_2$ of categories of a pseudonatural transformation 
$\Sigma: \M_{k+1+l}(-,m(-),-;-)\Rightarrow\M_{k+n+l}(-,-,-;-): (\M_k\times\M_n\times\M_l)\times\M_1\to\Cat_2$. 
Observe that for 1-cells $\crta p, \crta f, \crta q\in\M_k\times\M_n\times\M_l$ we have contravariant functors 
$\M_{k+1+l}(\crta p, m(\crta f), \crta q;1)$ and $\M_{k+n+l}(\crta p, \crta f, \crta q;1)$ and that similarly we have covariant functors in the last single variable. Thus in particular, the 2-cell component 
of $\Sigma$ at a 1-cell  $\crta p, \crta f, \crta q, r\in(\M_k\times\M_n\times\M_l)\times\M_1$ is a natural transformation of the form 
\begin{equation} \eqlabel{Sigma 2-cell comp}
	\scalebox{0.86}{
\bfig
 \putmorphism(-150,470)(1,0)[\M_{k+1+l}(\crta x',m(a_1',..,a_n'),\crta b';c)`\M_{k+n+l}(\crta x',a_1',..,a_n',\crta b';c)`
   \Sigma_{\crta x,m(a_1',..,a_n'),\crta b;c}]{1900}1a
\putmorphism(-160,470)(0,-1)[\phantom{Y_2}` ` \M_{k+1+l}(\crta p, m(\crta f), \crta q;r)]{380}1l
\putmorphism(1850,470)(0,-1)[\phantom{Y_2}` `  \M_{k+n+l}(\crta p, \crta f, \crta q;r)]{380}1r
 \putmorphism(-150,70)(1,0)[\M_{k+1+l}(\crta x,m(a_1,..,a_n),\crta b;c')`\M_{k+n+l}(\crta x,a_1,..,a_n,\crta b;c'). `
   \Sigma_{\crta x,m(a_1,..,a_n),\crta b;c}]{1900}1b
\put(520,270){\fbox{$\SWarrow \Sigma_{\crta h, \crta f, \crta k;r}$}}
\efig}
\end{equation}
\medskip


\section{From a bi-representable bi-multicategory to a monoidal bicategory} \selabel{bi-bi->mon.bic}

Suppose $\M$ is a bi-representable bi-multicategory, and let $\ul{\M}$ denote its underlying bicategory. Define a monoidal product of two objects in 
$\ul{\M}$ by the corresponding universal object $a_1\ot a_2:=m(a_1,a_2)$ ({\em i.e.} the codomain of the universal arrow for the list 
$a_1, a_2$), and the unit object by the codomain of the universal arrow of the empty list 
$j_{\langle\rangle}:\langle\rangle\to I$. For the monoidal structure equivalences $a, l, r$, consider the unary 1-cells between the codomains of the following pairs of universal arrows 
$$
\langle a_1,a_2,a_3\rangle \stackrel{\langle j_{a_1,a_2},id_{a_3}\rangle}{\to} \langle m(a_1,a_2),a_3\rangle\stackrel{j_{a_1\ot a_2,a_3}}{\to}
(a_1\ot a_2)\ot a_3
$$
$$
\langle a_1,a_2,a_3\rangle \stackrel{\langle id_{a_1},j_{a_2,a_3}\rangle}{\to} \langle a_1, m(a_2,a_3)\rangle\stackrel{j_{a_1,a_2\ot a_3}}{\to}
a_1\ot (a_2\ot a_3)
$$

we denote this one by $a_{a_1,a_2,a_3}: (a_1\ot a_2)\ot a_3\to a_1\ot (a_2\ot a_3)$, 
\begin{equation} \eqlabel{def l}
\langle a\rangle=\langle\rangle\cdot\langle a\rangle \stackrel{\langle j_{\langle\rangle},id_a\rangle}{\to} \langle I,a\rangle \stackrel{j_{I,a}}{\to}I\ot a \qquad \text{and} \qquad \langle a\rangle\stackrel{id_a}{\to} a
\end{equation} 
we denote this one by $l_a:I\ot a\to a$, and 
\begin{equation} \eqlabel{def r}
\langle a\rangle=\langle a\rangle\cdot\langle\rangle \stackrel{\langle id_a,j_{\langle\rangle}\rangle}{\to} \langle a,I\rangle \stackrel{j_{a,I}}{\to}a\ot I \qquad \text{and} \qquad \langle a\rangle\stackrel{id_a}{\to} a
\end{equation}  
we denote this one by $r_a:a\ot I\to a$. These mediating 1-cells are obtained as objects in the right-hand side category in the equivalence 
\equref{adj equiv}, by setting for $m(a_1,..,a_n)$ and $c$ the domain and the codomain of the mediating 1-cell in question, 
respectively. We obtain their quasi-inverses $a^*, l^*, r^*$ similarly, by inverting the domain and the codomain. 

For the empty list and $c=I$ let $(\iota, \Epsilon_\iota)$ be the image in the right-hand side in \equref{adj equiv} of $j:=j_{\langle\rangle}:\langle\rangle\to I$. 
Then $\Epsilon_\iota: \iota j\Rightarrow j$ is an invertible 2-cell so that for $id_I$ and any 2-cell, in particular, $\lambda_j:id_I j\Rightarrow j$ by universality there is a unique 
invertible 2-cell $\sigma^\dagger=u:id_I\Rightarrow\iota$ s.t. $\frac{[\Id_j\vert u]}{\Epsilon_\iota}=\lambda_j$. 
Similarly, there is a unique invertible 2-cell $\nabla: \iota\iota\Rightarrow\iota$ such that $\frac{[\Id_j\vert\nabla]}{\Epsilon_\iota}=
{\Epsilon_{\iota^2}}$. 
The 2-cells $\nabla, \iota$ are shown to fulfill the axioms in \axiomref{MB5} of a monoidal bicategory (we skip this proof as a simpler version of 
a proof we will carry out shortly). 

\smallskip

Quasi-naturality of $a, l, r$ and $a^*, l^*, r^*$ is obtained as follows. Take unary morphisms $f\in\M_1(x;x'), \,\, g\in\M_1(y;y'), \,\,
h\in\M_1(z;z')$. Then the composite 
$$x,y,z\stackrel{\langle f,g,h\rangle}{\to}x',y',z'\stackrel{\langle j_{x',y'},id_{z'}\rangle}{\to}x'\ot y', z'
\stackrel{j_{x'\ot y',z'}}{\to}(x'\ot y') z'$$ 
as an object on the left-hand side in \equref{adj equiv} induces an object $(f\ot g)\ot h$ on the right hand-side, and we can make 
the (horizontal) composite $a_{x',y',z'}((f\ot g)\ot h)$ as the image by the substitution (of two unary mutimaps) living in 
$\M_1((x\ot y)\ot z; x'\ot(y'\ot z'))$. Similarly, we build the composite $(f\ot (g\ot h))a_{x,y,z}$, also an object in $\M_1((x\ot y)\ot z; x'\ot(y'\ot z'))$. By the universal property, we have invertible 2-cells epsilons in the diagrams below
\begin{equation} \eqlabel{Epsilon-1} 
\bfig
 \putmorphism(0,700)(1,0)[x,y,z` `\langle j, 1\rangle ]{500}1a
 \putmorphism(530,700)(1,0)[`(x\ot y)\ot z `j]{500}1a
 \putmorphism(-20,700)(1,-2)[\phantom{A\ot B}` `\langle f,g,h\rangle ]{190}1l
 \putmorphism(100,400)(1,-2)[\phantom{A\ot B}` ` \langle 1,j\rangle]{190}1l
 \putmorphism(240,100)(1,-2)[\phantom{A\ot B}`x'\ot(y'\ot z') ` j]{170}1l
 \putmorphism(-60,730)(2,-1)[\phantom{A\ot B}` `\langle 1,j\rangle]{550}1r
 \putmorphism(350,500)(2,-1)[\phantom{A\ot B}` `]{550}1l
 \putmorphism(350,480)(2,-1)[\phantom{A\ot B}` `j]{550}0l
 \putmorphism(1030,700)(-1,-2)[` `a_{x,y,z}]{240}1r
 \putmorphism(830,300)(-1,-2)[` `f\ot(g\ot h)]{280}1r
\put(590,530){\fbox{$\Epsilon^1_{a_{x,y,z}}$}}
\put(320,180){\fbox{$\Epsilon^2_{f\ot(g\ot h)}$}}
\efig
\qquad\qquad
\bfig
 \putmorphism(0,700)(1,0)[x,y,z` `\langle j, 1\rangle ]{500}1a
 \putmorphism(500,700)(1,0)[`(x\ot y)\ot z `j]{530}1a
 \putmorphism(-40,700)(1,-2)[\phantom{A\ot B}` `\langle f,g,h\rangle ]{190}1l
 \putmorphism(100,400)(1,-2)[\phantom{A\ot B}` ` \langle 1,j\rangle]{190}1l
 \putmorphism(240,100)(1,-2)[\phantom{A\ot B}`x'\ot(y'\ot z') ` j]{170}1l
 \putmorphism(-50,440)(3,-1)[\phantom{A\ot B}` `]{700}1r
 \putmorphism(-180,510)(3,-1)[\phantom{A\ot B}` `\langle j,1\rangle]{700}0r
 \putmorphism(300,320)(3,-1)[\phantom{A\ot B}` `]{650}1r
 \putmorphism(240,380)(3,-1)[\phantom{A\ot B}` `j]{650}0r
 \putmorphism(1030,700)(-1,-2)[` `(f\ot g)\ot h]{280}1r
 \putmorphism(770,200)(-1,-2)[` `a_{x',y',z'}]{200}1r
\put(480,480){\fbox{$\Epsilon^2_{(f\ot g)\ot h}$}}
\put(320,60){\fbox{$\Epsilon^1_{a_{x',y',z'}}$}}
\efig
\end{equation} 
so that the composite 2-cells are $\sigma_1$ (on the left) and $\sigma_2$ (on the right). 
Setting $\zeta$ for the arrow on the left in both diagrams, by universality there is a unique invertible 2-cell $\sigma^\dagger
=a_{f,g,h}: a_{x',y',z'}((f\ot g)\ot h)\Rightarrow (f\ot (g\ot h))a_{x,y,z}$ so that  
$$\frac{[\Id_{j\langle j, 1\rangle}\vert a_{f,g,h}]}{\sigma_1}= \sigma_2$$
{\em i.e.} there is a prism $\Theta(\sigma_2, \sigma_1; \Id_\zeta, a_{f,g,h})$ (a semi-open seashell with bases $\sigma_2$ and $\sigma_1$ and a unique mediating 
2-cell $a_{f,g,h}$). 

To show naturality of the 2-cells $a_{f,g,h}$, take unary 2-cells, {\em i.e.} morphisms in the categories $\alpha\in\M_1(x;x'), \,\, \beta\in\M_1(y;y'), \,\,
\gamma\in\M_1(z;z')$ mapping $\langle\alpha,\beta,\gamma\rangle:\langle f,g,h\rangle\to \langle f',g',h'\rangle$. 
On the basis $\sigma_2$ we can set the prism $\Theta(\zeta\backslash\frac{\sigma_1}{\sigma_2} /a_{f,g,h})$, and then on top of it 
the identity prism with basis $\Epsilon^1_{a_{x,y,z}}$ concatenated by  
$P(\langle\alpha,\langle\beta, \gamma\rangle\rangle\backslash\frac{\Epsilon^2_{f'\ot (g'\ot h')}}{\Epsilon^2_{f\ot (g\ot h)}} /\alpha\ot(\beta\ot\gamma))$. 
(Here $\langle\alpha,\langle\beta, \gamma\rangle\rangle$ denotes $[\langle\alpha,\beta, \gamma\rangle\vert \Id_{j\langle 1,j\rangle}]$.) 
See the diagram below left. 
The composite prism we can denote by  
$$\frac{\Theta(\zeta\backslash\frac{\sigma_1}{\sigma_2} /a_{f,g,h})}{\Id_{\Epsilon^1_{a_{x,y,z}}}\diamond 
P(\langle\alpha,\langle\beta, \gamma\rangle\rangle\backslash\frac{\Epsilon^2_{f'\ot (g'\ot h')}}{\Epsilon^2_{f\ot (g\ot h)}} /\alpha\ot(\beta\ot\gamma))}$$
whereby $\diamond$ denotes the horizontal concatenation of prisms in the obvious sense. Observe that we first set (on the bottom) the prism in the numerator, 
and then (on top of it) the one from the denominator: the fraction follows the distribution of the prisms, while the induced composite 2-cell goes from bottom to top. This composite prism states that the 2-multicell $\langle\alpha,\langle\beta, \gamma\rangle\rangle$ 
from $\M_{3}(x,y,z;x'\ot(y'\ot z') )$ induces the 2-cell $\frac{a_{f,g,h}}{[\Id_{a_{x,y,z}}\vert \alpha\ot (\beta\ot \gamma) ]}$ 
in $\M_1((x\ot y)\ot z;x'\ot(y'\ot z') )$. 
$$
\begin{tikzcd}
	&&&& \bullet & \\
	\\
	\bullet &&& \bullet & \bullet \\
	\\
	\bullet &&& \bullet && \bullet \\
	&&&& \bullet \\
	&& \bullet \\
	\bullet &&& \bullet && \bullet
	\arrow[""{name=0, anchor=center, inner sep=0}, from=1-5, to=3-1]
	\arrow[""{name=1, anchor=center, inner sep=0}, from=1-5, to=3-4]
	\arrow[equals, nfold, from=1-5, to=3-5]
	\arrow[equals, nfold, from=3-1, to=5-1]
	\arrow[""{name=2, anchor=center, inner sep=0}, "{f'\ot(g'\ot h')}"{pos=0.6}, from=3-4, to=3-1]
	\arrow[equals, nfold, from=3-4, to=5-4]
	\arrow[""{name=3, anchor=center, inner sep=0}, from=3-5, to=5-1]
	\arrow[""{name=4, anchor=center, inner sep=0}, from=3-5, to=5-4]
	\arrow[from=3-5, to=5-6]
	\arrow[""{name=5, anchor=center, inner sep=0}, "{f\ot(g\ot h)}"', tail reversed, no head, from=5-1, to=5-4]
	\arrow[equals, nfold, from=5-1, to=8-1]
	\arrow["{a_{f,g,h}}"'{pos=0.3}, color={rgb,255:red,92;green,92;blue,214}, between={0.1}{0.9}, equals, nfold, 2tail reversed, from=5-4, to=8-4]
	\arrow[""{name=6, anchor=center, inner sep=0}, "{a_{x,y,z}}"{pos=0.4}, from=5-6, to=5-4]
	\arrow[equals, nfold, from=5-6, to=8-6]
	\arrow[""{name=7, anchor=center, inner sep=0}, from=6-5, to=7-3]
	\arrow[from=6-5, to=8-6]
	\arrow[""{name=8, anchor=center, inner sep=0}, from=7-3, to=8-1]
	\arrow[from=7-3, to=8-4]
	\arrow[""{name=9, anchor=center, inner sep=0}, "{a_{x',y',z'}}"', tail reversed, no head, from=8-1, to=8-4]
	\arrow[""{name=10, anchor=center, inner sep=0}, "{(f\ot g)\ot h}", from=8-6, to=8-4]
	\arrow["{\Epsilon^2_{f'\ot (g'\ot 'h)}}", between={0.2}{0.8}, Rightarrow, nfold, from=1, to=0]
	\arrow["{\langle\alpha,\langle\beta,\gamma\rangle\rangle}"'{pos=0.3}, between={0}{0.8}, Rightarrow, nfold, from=0, to=3]
	\arrow["{\Epsilon^2_{f\ot (g\ot h)}}"{pos=0.7}, between={0.2}{0.8}, Rightarrow, nfold, from=4, to=3]
	\arrow["{\alpha\ot(\beta\ot\gamma)}", color={rgb,255:red,92;green,92;blue,214}, between={0.1}{0.8}, Rightarrow, nfold, from=5, to=2]
	\arrow["{\Epsilon^1_{a_{x,y,z}}}"', between={0.2}{0.8}, Rightarrow, nfold, from=6, to=4]
	\arrow["{\Epsilon^1_{a_{x',y',z'}}}"'{pos=0.3}, between={0.2}{0.8}, Rightarrow, nfold, from=9, to=8]
	\arrow["{\Epsilon^2_{(f\ot g)\ot h}}"', between={0.2}{0.8}, Rightarrow, nfold, from=10, to=7]
\end{tikzcd}
\quad
\begin{tikzcd}
	&&& \bullet && \\
	& \bullet &&&& \bullet \\
	\bullet && \bullet \\
	&&& \bullet \\
	& \bullet &&&& \bullet \\
	\bullet && \bullet & \bullet \\
	& \bullet &&&& \bullet \\
	&& \bullet
	\arrow[from=1-4, to=2-2]
	\arrow[from=1-4, to=2-6]
	\arrow[""{name=0, anchor=center, inner sep=0}, from=1-4, to=3-3]
	\arrow[equals, nfold, dashed, from=1-4, to=4-4]
	\arrow[from=2-2, to=3-1]
	\arrow[""{name=1, anchor=center, inner sep=0}, "{a_{x,y,z}}", from=2-6, to=3-3]
	\arrow[equals, nfold, from=2-6, to=5-6]
	\arrow[equals, nfold, from=3-1, to=6-1]
	\arrow["{f'\ot (g'\ot h')}", from=3-3, to=3-1]
	\arrow[""{name=2, anchor=center, inner sep=0}, from=4-4, to=5-2]
	\arrow[from=4-4, to=5-6]
	\arrow[equals, nfold, dashed, from=4-4, to=6-4]
	\arrow[""{name=3, anchor=center, inner sep=0}, from=5-2, to=6-1]
	\arrow[""{name=4, anchor=center, inner sep=0}, from=5-2, to=6-3]
	\arrow[""{name=5, anchor=center, inner sep=0}, "{(f'\ot g')\ot h'}"{pos=0.3}, from=5-6, to=6-3]
	\arrow[equals, nfold, from=5-6, to=7-6]
	\arrow["{a_{x',y',z'}}"', tail reversed, no head, from=6-1, to=6-3]
	\arrow["{a_{f',g',h'}}"{pos=0.7}, color={rgb,255:red,92;green,92;blue,214}, Rightarrow, nfold, from=6-3, to=3-3]
	\arrow[equals, nfold, from=6-3, to=8-3]
	\arrow[""{name=6, anchor=center, inner sep=0}, from=6-4, to=7-2]
	\arrow[from=6-4, to=7-6]
	\arrow[from=7-2, to=8-3]
	\arrow[""{name=7, anchor=center, inner sep=0}, "{(f\ot g)\ot h}", from=7-6, to=8-3]
	\arrow["{\Epsilon^2_{f'\ot (g'\ot h')}}"{pos=0.6}, between={0.2}{1}, Rightarrow, nfold, from=0, to=2-2]
	\arrow["{\Epsilon^1_{a_{x,y,z}}}"'{pos=0.3}, between={0.2}{0.8}, Rightarrow, nfold, from=1, to=0]
	\arrow["{\Epsilon^1_{a_{x',y',z'}}}"{pos=0.2}, between={0.2}{0.8}, Rightarrow, nfold, from=4, to=3]
	\arrow["{\Epsilon^2_{(f'\ot g')\ot h'}}"'{pos=0.1}, between={0.2}{0.8}, Rightarrow, nfold, from=5, to=2]
	\arrow["{(\alpha\ot\beta)\ot\gamma}"{pos=0.3}, color={rgb,255:red,92;green,92;blue,214}, between={0.1}{0.9}, equals, nfold, 2tail reversed, from=5, to=7]
	\arrow["{\Epsilon^2_{(f\ot g)\ot h}}"', between={0.2}{0.8}, Rightarrow, nfold, from=7, to=6]
\end{tikzcd}$$
On the other hand, on the basis $\sigma_2$ set the concatenation of the prism identity on $\Epsilon^1_{a_{x',y',z'}}$ and 
$P(\langle\langle\alpha,\beta\rangle, \gamma\rangle\backslash \frac{\Epsilon^2_{(f'\ot g')\ot h'}}{\Epsilon^2_{(f\ot g)\ot h}}/(\alpha\ot\beta)\ot\gamma)$, 
and on top of them the prism $\Theta(\zeta'\backslash\frac{\sigma_1'}{\sigma_2'}/ a_{f',g',h'})$. See the diagram above right. 
The composite prism we can denote as follows
$$\frac{ P(\langle\langle\alpha,\beta\rangle, \gamma\rangle\backslash \frac{\Epsilon^2_{(f'\ot g')\ot h'}}{\Epsilon^2_{(f\ot g)\ot h}}/(\alpha\ot\beta)\ot\gamma)
\diamond \Id_{\Epsilon^1_{a_{x',y',z'}}}} {\Theta(\zeta'\backslash\frac{\sigma_1'}{\sigma_2'}/ a_{f',g',h'})}. 
$$
It states that the 
same 2-multicell $\langle\alpha,\langle\beta, \gamma\rangle\rangle$ 
from $\M_{3}(x,y,z;x'\ot(y'\ot z') )$ induces the 2-cell $\frac{[(\alpha\ot\beta)\ot \gamma \vert \Id_{a_{x',y',z'}}]}{a_{f',g',h'}}$ 
in $\M_1((x\ot y)\ot z;x'\ot(y'\ot z') )$. 
By the uniqueness of the induction of 2-multicells in $\M_1((x\ot y)\ot z;x'\ot(y'\ot z') )$ we obtain naturality of $a_{f,g,h}$, {\em i.e.} 
$$\frac{a_{f,g,h}}{[\Id_{\Epsilon^1_{a_{x,y,z}}}\vert \alpha\ot (\beta\ot \gamma) ]}=
\frac{[(\alpha\ot\beta)\ot \gamma \vert \Id_{\Epsilon^1_{a_{x',y',z'}}}]}{a_{f',g',h'}}.$$

\bigskip

Similarly, we obtain a unique invertible 2-cell $a^*_{f,g,h}: a^*_{x',y',z'}(f\ot (g\ot h))\Rightarrow ((f\ot g)\ot h)a^*_{x,y,z}$ 
natural in $f,g,h$. This completes the construction of a pseudonatural equivalence $a: \ot(\ot\times 1)\Rightarrow \ot(1\times\ot)$ 
from \axiomref{MB6} of a monoidal bicategory. Similarly, pseudo natural equivalences $l: \ot\comp(I\times\Id_{\B})\to\Id_{\B}$ and 
$r: \ot\comp(\Id_{\B}\times I) \to \Id_{\B}$ from \axiomref{MB7} are constructed. (With the same strategy as above one proves that 
$a, l, r$ obey the axioms of pseudonatural equivalences.) 

\bigskip

A unique invertible 2-cell $\pi_{x,y,z,w}: (id_x\ot a_{y,z,w})a_{xy,z,w}(a_{x,y,z}\ot id_w)\Rightarrow a_{x,y,zw}a_{xy,z,w}$ between 
two images living in $\M_1(((x\ot y)\ot z)\ot w; x\ot(y\ot(z\ot w)))$ by the substitution (of unary mutimaps) is constructed  in a similar way as $a_{f,g,h}$ 
above using the universal property of universal arrows. The same holds for unique invertible 2-cells 
\begin{equation} \eqlabel{other modifs}
\mu_{xy}:r_x\ot id_y\Rightarrow(id_x\ot l_y)a_{x,I,y}, \,\,\, \tilde\lambda_{x,y}: l_x\ot id_y\Rightarrow l_{xy}a_{I,x,y}, \,\,\, 
\tilde\rho_{x,y}:(id_x\ot r_y)a_{x,y,I}\Rightarrow r_{xy}.
\end{equation} 
Then any axiom that can be formulated using the compositions of these 2-cells is fulfilled, in particular those of \axiomref{MB8} of a monoidal bicategory. 
We explain this in more details for the axiom (TA1). 
Both $\pi_{x,y,z,w}$ and any variation $\pi'_{x,y,z,w,u}$ of it appearing in the axiom (TA1), as well as the naturality 2-cells of $a$, 
have the form of a semi-open seashell 
$$\frac{[\Id_{j_{x,y,z,w,u}}\vert \pi'_{x,y,z,w,u}]}{\sigma_{codomain}}= \sigma_{domain}$$
whereby $\pi'_{x,y,z,w,u}$ is the unique mediating 2-cell between the third differing edges of the surface 2-cells $\sigma_{domain}$ and $\sigma_{codomain}$, 
respectively. Then the axiom is an equality of the correspondingly composed seashells (stacked vertically one on top of the other), meaning that the induced composite mediating 2-cells - obtained by the corresponding compositions of different $\pi'_{x,y,z,w,u}$s - on both sides of the equality coincide. Since the 
induced mediating 2-cells between same arrows acting on a universal object are unique by universality, the desired equality holds true. 
 (The axioms for $\pi, \mu, \tilde\lambda, \tilde\rho$ to be modifications are proved using the above strategy by which we showed naturality of 
the 2-cells $a_{f,g,h}$, and by which one proves that $a, l, r$ are pseudonatural equivalences.)

\bigskip

As for the functoriality of the tensor product, observe that $j_{x',y'}\langle f,g\rangle$ induces the pair $(f\ot g, \Epsilon_{f\ot g})$ 
and $j_{x'',y''}\langle f',g'\rangle$ induces the pair $(f'\ot g', \Epsilon_{f'\ot g'})$, as on the left in the diagram below, whereas  the 
composite 
$j_{x'',y''}\langle f',g'\rangle\langle f,g\rangle=
j_{x'',y''}\langle [f\vert f'],[g\vert g']\rangle$ 
induces the pair $([f\vert f']\ot[g\vert g'],\Epsilon_{[f\vert f'],[g\vert g']})$, being $[f\vert f']\ot[g\vert g']$ the 1-cell as on the right 
and $\Epsilon_{[f\vert f'],[g\vert g']}$ the 2-cell filling the outer diagram, 
so that there is a unique mediating invertible 2-cell $\sigma^\dagger=\zeta:[f\vert f']\ot[g\vert g']\Rightarrow [f\ot g\vert f'\ot g']$ 
(on the right below) 
\begin{equation} \eqlabel{mon int}
\scalebox{0.86}{
\bfig
 \putmorphism(-150,470)(1,0)[\langle x,y\rangle`x\ot y  `j_{x,y}]{480}1a
   \putmorphism(350,470)(1,0)[\phantom{x\ot y}`x\ot y  `=]{480}1a
\putmorphism(-160,470)(0,-1)[\phantom{Y_2}` `\langle f,g\rangle]{400}1l
\putmorphism(-160,80)(0,-1)[\phantom{Y_2}``\langle f',g'\rangle]{380}1l
\putmorphism(340,80)(0,-1)[\phantom{Y_2}``f'\ot g']{380}1r
\putmorphism(340,470)(0,-1)[\phantom{Y_2}` `f\ot g]{400}1r
\put(-20,280){\fbox{$\Epsilon_{f\ot g}$}}

\putmorphism(-150,70)(1,0)[\langle x',y'\rangle`x'\ot y' `j_{x',y'}]{470}1a
\putmorphism(-150,-300)(1,0)[\langle x'',y''\rangle` x''\ot y''`j_{x'',y''}]{490}1b
  \putmorphism(350,-300)(1,0)[\phantom{x''\ot y''}` x''\ot y''`=]{490}1b
\putmorphism(850,470)(0,-1)[\phantom{Y_2}` `{{[f\vert f']\ot[g\vert g']}} ]{780}1r
\put(560,40){\fbox{$\stackrel{\zeta}{\Leftarrow}$}}
\put(-20,-130){\fbox{$\Epsilon_{f'\ot g'}$}}
\efig}
\end{equation} 
so that the composition of the constituting 2-cells above equals $\Epsilon_{[f\vert f'],[g\vert g']}$, {\em i.e.} so that 
$$\frac{[\Id_{j_{x,y}}\vert \zeta]}{\Epsilon_{f\ot g}\diamond\Epsilon_{f'\ot g'}}= \Epsilon_{[f\vert f'],[g\vert g']}$$
where $\Epsilon_{f\ot g}\diamond\Epsilon_{f'\ot g'}$ is the obvious concatenation of 2-cells. 
The invertible 2-cell $\zeta$ is proved to fulfill the desired hexagonal axiom from \axiomref{MB4} of a monoidal bicategory in a similar way as proved above for 
the (unique mediating invertible) 2-cell $\pi_{x,y,z,w}$. The invertible 2-cell $\zeta_0: id_{x\ot y}\Rightarrow id_x\ot id_y$
satisfying two squares is constructed in an analogous way. This finishes the proof that a bi-representable bi-multicategory induces a monoidal bicategory.

\subsection{Summary of the construction, general features and cliques in a bicategory} \sslabel{cliq}

After the above construction let us stress some principles that we may use in a bi-representable bi-multicategory independently from the construction 
of a monoidal bicategory. Because of later uses we highlight them in the following remark. 

\begin{rem} \rmlabel{j and j'}
Let $\M$ be any bi-representable bi-multicategory with universal objects $m(a_1,...,a_n)$ and universal arrows $j_n:a_1,...,a_n\to m(a_1,...,a_n)$. 
\begin{enumerate} [a)]
\item We may consider different choices of universal objects in $\M$, in particular 
$$m^L(a_1,...,a_n):=m((...m(m(a_1,a_2), a_3), ...), a_n)$$
and
$$m^R(a_1,...,a_n):=m(a_1,m(...,m(a_{n-2}, m(a_{n-1},a_n))...)$$ 
with their corresponding universal arrows 
$$j_n^L:a_1,...,a_n\to m^L(a_1,...,a_n)\quad\text{and}\quad j_n^R:a_1,...,a_n\to m^R(a_1,...,a_n).$$
In this sense the 2-cell $\Epsilon^1_{a_{x,y,z}}$ in \equref{Epsilon-1} maps $a_{x,y,z} j^L_3\Rightarrow j^R_3$. For $n=4$, identifying the compositions of multimaps 
$\langle 1,j\rangle\langle j,1,1\rangle=\langle j,1\rangle\langle 1,1,j\rangle$, one obtains a 2-cell $\Epsilon_{a_{x,y,z,w}}:a_{x,y,z,w} j^L_4\Rightarrow j^R_4$, 
where $a_{x,y,z,w}=a_{x,y,zw}\comp a_{xy,z,w}$. Iterating this process we obtain a 2-cell $\Epsilon_{a_{a_1,...,a_n}}:a_{a_1,...,a_n} j^L_n\Rightarrow j^R_n$, 
 where now $a_{a_1,...,a_n}$ is a composition of $n-2$ 1-cells $a_{\bullet,\bullet,\bullet}$. 

\item In general, for any other choice of universal objects $m'(a_1,...,a_n)$ and universal arrows $j_n':a_1,...,a_n\to m'(a_1,...,a_n)$ there is a 
2-cell $\tau$ as in the diagram below 
\begin{equation} \eqlabel{tau-a}
\begin{tikzcd}
	{} & {a_1,...,a_n} && {m(a_1,...,a_n)} \\
	&& {m'(a_1,...,a_n)}
	\arrow[""{name=0, anchor=center, inner sep=0}, "{j_n}", from=1-2, to=1-4]
	\arrow["{j'_n}"', from=1-2, to=2-3]
	\arrow[""{name=1, anchor=center, inner sep=0}, "{a_{\overline{a_i}}}"', from=1-4, to=2-3]
	\arrow[""{name=2, anchor=center, inner sep=0}, "{a'_{\overline{a_i}}}"{pos=0.3}, curve={height=-24pt}, from=1-4, to=2-3]
	\arrow["\tau"{description}, draw=none, from=0, to=2-3]
	\arrow["\pi"{description}, draw=none, from=2, to=1]
\end{tikzcd}
\end{equation}	
with $a_{a_1,...,a_n}$ being a composition of $n-2$ \quad 1-cells $a_{\bullet,\bullet,\bullet}$. 
Given another path $a_{\crta{a_i}}':m(a_1,...,a_n)\to m'(a_1,...,a_n)$ with another $\tau'$ 
by the universal property there is a unique 2-cell $\pi:a_{\crta{a_i}}'\Rightarrow a_{\crta{a_i}}$ such that $\pi\diamond\tau=\tau'$. 
(Observe that for $n=3$ there is a single path $a_{a_1, a_2, a_3}$, whereas for $n>3$ there are more such paths.) 
\item For a ternary list the universal arrow $j_3$ is either $j^L$ or $j^R$, the universal object $m(x,y,z)$ is either $m^L(x,y,z)$ or $m^R(x,y,z)$, 
and there is a unique arrow $a_{x,y,z}:m^L(x,y,z)\to m^R(x,y,z)$. In case $y=m(-)$, we easily construct from the universal property a 2-cell 
$\mu: m(id_x,l_z)\circ a_{x, m(-), z}\Rightarrow m(r_x, id_z)$, as in \equref{other modifs}, with $l_z:m(m(-),z)\to z$ and $r_x:m(x,m(-))\to x$ 
obtained as in \equref{def l} and \equref{def r}, respectively. 
\item The way we constructed the equivalence 1-cells $a_{x,y,z}, l_z, r_x$ is an example of a more general fact: a multimap $f: a_1,...,a_n\to b$ 
is a universal map if and only if the induced map $\hat f: m_{a_1,...,a_n}\to b$ is an equivalence 1-cell. 
\end{enumerate}
\end{rem}

\bigskip

We may summarize the construction of a monoidal bicategory out of a bi-representable bi-multicategory $\M$ as follows: 
\begin{itemize}
\item for hom-categories: forget the multi-hom categories for $n>1$, \vspace{-0,2cm}
\item for monoidality:  \vspace{-0,12cm}
\begin{enumerate}
\item define all the structure 2-cells via the universal property of bi-representability (existence of a unique 2-cell), 
\item 
all the axioms on those 2-cells hold by the principle: ``all sensible compositions $c_i:f\Rightarrow g: m(a_1,...,a_n)\to b, i\in J$ 
of the 2-cells constructed in the previous item and the coherence 2-cells of $\M$, for some set of indexes $J$, acting on the same pair of 
(composite) 1-cells $f,g$ whose domain is a universal object coincide''; this principle holds by the uniqueness of the 2-cells obtained from the universal property. 
\end{enumerate}
\end{itemize}

\bigskip

The data summarized in the above remark make the core pieces of the notion of a {\em clique} in a bicategory, analogously as in the 1-dimensional proof of \cite{Herm} 
(see the comment after Definition 9.6 thereof). We introduce this notion as a bicategorical analogue of the 1-categorical notion of clique from \cite{JS}. 
(In the existing literature the author found only that in \cite{FS} ``(rebracketing) coherence isomorphisms'', {\em a.k.a} clique morphisms, in a bicategory were mentioned.)

\begin{defn}
A clique in a bicategory $\B$ is a pseudofunctor $\mathcal I\to\B$ from the chaotic category \(\mathcal I\). 
\end{defn}

Unpacking the definition one sees that a clique in $\B$ consists of objects \(a_i\), comparison 1-cells \(c_{ij}:a_i\to a_j\), and invertible 2-cells
$ c_{ijk}:c_{jk}c_{ij}\Rightarrow c_{ik}$ and $\iota_i:1_{a_i}\Rightarrow c_{ii}$ 
satisfying the coherence axioms for a pseudofunctor, which translate into the commutativities
\[\begin{tikzcd}
	{(c_{kl}\circ c_{jk})\circ c_{ij}} && {c_{jl}\circ c_{ij}} \\
	{c_{kl}\circ( c_{jk}\circ c_{ij})} && {c_{il}} \\
	& {c_{kl}\circ c_{ik}}
	\arrow["{c_{jkl}\circ 1}", Rightarrow, nfold, from=1-1, to=1-3]
	\arrow["\alpha"', Rightarrow, nfold, from=1-1, to=2-1]
	\arrow["{c_{ijl}}", Rightarrow, nfold, from=1-3, to=2-3]
	\arrow["{1\circ c_{ijk}}"', Rightarrow, nfold, from=2-1, to=3-2]
	\arrow["{c_{ikl}}"', Rightarrow, nfold, from=3-2, to=2-3]
\end{tikzcd}\]
$$\begin{tikzcd}
	{ id_{a_j}\circ c_{ij}} && {c_{jj}\circ c_{ij}} \\
	& {c_{ij}}
	\arrow["{\iota_{j}\circ 1}", Rightarrow, nfold, from=1-1, to=1-3]
	\arrow["{\lambda_{ij}}"', Rightarrow, nfold, from=1-1, to=2-2]
	\arrow["{c_{ijj}}", Rightarrow, nfold, from=1-3, to=2-2]
\end{tikzcd}
\qquad
\begin{tikzcd}
	{ c_{ij}\circ id_{a_i}} && {c_{ij}\circ c_{ii}} \\
	& {c_{ij}.}
	\arrow["{1\circ\iota_{i}}", Rightarrow, nfold, from=1-1, to=1-3]
	\arrow["{\rho_{ij}}"', Rightarrow, nfold, from=1-1, to=2-2]
	\arrow["{c_{iij}}", Rightarrow, nfold, from=1-3, to=2-2]
\end{tikzcd}$$
In particular, every \(c_{ij}\) is an equivalence.

Given a bi-representable bi-multicategory and a list of objects $a_1,...,a_n$ in it 
between any two choices of the universal objects $m(a_1,...,a_n)$ and $m'(a_1,...,a_n)$ there exists an equivalence 1-cell $a_{a_1,...,a_n}$ 
with a 2-cell $\tau: a_{a_1,...,a_n}\circ j\Rightarrow j'$ satisfying the universal property 
as in the above remark. We take these different 
choices of universal objects as the objects of a clique, the equivalence 1-cells $a_{a_1,...,a_n}$ we take for the comparison 1-cells, 
for the invertible 2-cells $c_{ijk}$ we take the 2-cells $\pi$, existing in a unique  way by the universal property, and similarly we 
construct the 2-cells $\iota_i$. The above axioms for the pseudofunctor on the chaotic category hold by the principle on sensible compositions 
in point 2. above (we may ignore the structure 2-cell components of $\alpha, \lambda, \rho$ by the coherence of (underlying) bicategories). 
We may then record:

\begin{prop}
Given a bi-representable bi-multicategory $\M$ the universal objects in it and their associated higher cells form the data of a clique in the 
underlying bicategory $\ul{\M}$. 
\end{prop}

\section{From a monoidal bicategory to a bi-representable bi-multicategory} \selabel{mon bic - rep}

In this section we construct a bi-representable multicategory out of a monoidal bicategory. The construction is full of technical details 
and we divide it in several steps along the coming subsection. We will first set up the notation and define the substitution functor.  
Then we will define the associativity transformation and prove the pentagon law for it. We will next introduce unity transformations and prove the corresponding triangle laws, and finally argue about bi-representability.

\subsection{The substitution functor}

Let $\B$ be a monoidal bicategory. For objects $A_1,...,A_n\in\B$ let us denote 
$$\textcircled{L}_1^n A_i := (...(A_1\ot A_2)\ot\dots)\ot A_n$$ 
{\em i.e.} their tensor product with the leftmost distribution of parentheses.  
For a list of objects $A_{i1},...,A_{ik_i}$ of ``second order'', as in \equref{mon-multi} below, for $i=1,...,n$, 
the left-most tensor product made out of it we denote by $\textcircled{L}_j A_{ij}=\textcircled{L}_{j=1}^{k_i}A_{ij}$. 
Left-most tensoring 
of all of the objects in such lists of second order we denote shortly by $\textcircled{L}_{ij} A_{ij}=
(\textcircled{L}_{i=1}^{n}\textcircled{L}_{j=1}^{k_j})A_{ij}$. From the latter we will differentiate 
$\textcircled{L}_i(\textcircled{L}_j A_{ij})=(...\big(\textcircled{L}_{j=1}^{k_1}A_{1k_j} \ot 
\textcircled{L}_{j=1}^{k_2} A_{2k_j}\big)\ot ...)\ot \textcircled{L}_{j=1}^{k_n} A_{nj}$. As a matter of fact, we are going to make a choice 
of reassociation isomorphisms 
\begin{equation} \eqlabel{omega-a}
\omega_a: \textcircled{L}_{ij} A_{ij} \to \textcircled{L}_i(\textcircled{L}_j A_{ij})
\end{equation} 
for all lists of second order. 
For another 
choice of a reassociation isomorphism $\omega_b$ by the axiom (TA1) of a tricategory (and a monoidal bicategory) 
there is a unique invertible 2-cell $\pi_{\omega_a, \omega_b}:\omega_a\Rightarrow\omega_b$ obtained by a combination of 
concatenations and tensor products with identities of the suitable versions of the invertible 2-cell $\pi_{A,B,C,D}$ (each two such pastings 
$\pi_{\omega_a, \omega_b}, \pi_{\omega_a, \omega_b}'$ coincide by (TA1)).

Furthermore, we denote by $\ul{\B}$ the associated bi-multicategory having the same objects as $\B$ and whose 
rest of the data is defined as follows. Unary multi-hom categories $\ul{\B}(A;B)$ are the categories $\B(A,B)$, and 
multi-hom categories $\ul{\B}(A_1,...,A_n;B)$ are the categories $\B(\textcircled{L}_1^n A_i,B)$. The identities are those of $\B$. 
A substitution functor 
\begin{equation} \eqlabel{mon-multi}
\ul{\B}_n(B_1,...,B_n;C)\times\big(
\ul{\B}_{k_1}(A_{11},...,A_{1k_1};B_1)\times...\times\ul{\B}_{k_n}(A_{n1},...,A_{nk_n};B_n)\big)
\to\ul{\B}_{\sum_{i=1}^n k_i}(A_{11},...,A_{nk_n};C)
\end{equation} 
on an object $(g,(f_1,...,f_n))$ we define as the composition of 1-cells in $\B$ 
\begin{equation} \eqlabel{def subst}
g\circ\langle f_1,...,f_n\rangle:=g((\textcircled{L}_1^n f_i)\omega_a)
\end{equation}
mapping 
$$\big(\textcircled{L}_{ij} A_{ij}\stackrel{\displaystyle{\omega_a}}{\iso} \textcircled{L}_i(\textcircled{L}_j A_{ij}) 
\xrightarrow{(...(f_1\ot f_2)\ot...)\ot f_n} (...(B_1\ot B_2)\ot...)\ot B_n\big) \stackrel{g}{\to} C$$
and similarly on morphisms $(\gamma, (\Fi_1,...,\Fi_n)):(g,(f_1,...,f_n))\Rightarrow (g',(f'_1,...,f'_n))$. 
In other words, we define substitution \equref{mon-multi} by the commuting diagram 
\begin{equation}  \eqlabel{def subst-diag}
\bfig
 \putmorphism(-420,450)(0,-1)[\B(\underset{i}{\textcircled{L}} B_i,C)\times\underset{i}{\Pi}\B(\underset{j}{\textcircled{L}} A_{ij},B_i) 
  ``]{380}0l
 \putmorphism(-420,470)(0,-1)[\phantom{Y_2} ``1\times\tilde\omega_i]{380}1l	
	\putmorphism(-360,430)(0,-1)[\phantom{Y_2} ` \B(\underset{i}{\textcircled{L}} B_i,C)\times \B(\underset{ij}{\textcircled{L}} A_{ij}, 
	  \underset{i}{\textcircled{L}} B_i)`]{380}0l
	\putmorphism(40,480)(3,-1)[\phantom{Y_2} ` `]{1100}1r
	\putmorphism(-30,530)(3,-1)[\phantom{Y_2} ` `subst]{1100}0r
	\putmorphism(160,90)(1,0)[\phantom{Y_2}`\phantom{Y_2}`comp]{610}1b
	\putmorphism(160,50)(1,0)[\phantom{Y_2}`\B\big(\underset{ij}{\textcircled{L}}A_{ij}, C\big)`]{810}0b
\efig
\end{equation}
where $comp$ is the composition functor of $\B$ and $\tilde\omega_i$ is the composite 
$$\Pi_{i=1}^{n}\B(\textcircled{L}_{j=1}^{k_i}A_{ij},B_i)\stackrel{can}{\to}
\B(\textcircled{L}_{i=1}^{n}\big(\textcircled{L}_{j=1}^{k_i}A_{ij}\big), \textcircled{L}_{i=1}^{n} B_i)
\xrightarrow{\B(\omega_a,1)} 
\B\big((\textcircled{L}_{i=1}^{n}\textcircled{L}_{j=1}^{k_i})A_{ij},
  \textcircled{L}_{i=1}^{n} B_i\big)$$
with $can$ being the obvious canonical functor, mapping 
$$(f_1,...,f_n)\mapsto (\underset{i}{\textcircled{L}}f_i)\omega_a.$$

\subsection{Defining the associativity transformation in the induced bi-multicategory} 	\sslabel{assoc}

For the associativity of substitution in $\ul{\B}$ it is convenient to recall the full diagram expressing associativity of 
substitution in an ordinary multicategory $\M$: 
\begin{equation} \eqlabel{assoc bi-mc}
\scalebox{0.86}{
\bfig
 \putmorphism(-150,470)(1,0)[\M(\underset{i}{\coprod} B_i;C),\underset{i}{\coprod}\M(\underset{j}{\coprod} A_{ij}; B_i),\underset{ij}{\coprod}
  \M(\underset{q}{\coprod} D_{ij}^q; A_{ij}) 
   `\M(\underset{i}{\coprod} B_i;C),\underset{i}{\coprod}\big(\M(\underset{j}{\coprod} A_{ij}; B_i),\underset{j}{\coprod}
	 \M(\underset{q}{\coprod} D_{ij}^q; A_{ij})\big) `\langle 1, int\rangle]{2550}1a
\putmorphism(-160,470)(0,-1)[\phantom{Y_2}` ` \langle subst, 1\rangle]{380}1l
\putmorphism(-160,430)(0,-1)[\phantom{Y_2}`\M(\underset{ij}{\coprod} A_{ij}; C),\underset{ij}{\coprod}\M(\underset{q}{\coprod} D_{ij}^q; A_{ij})`]{380}0l
\putmorphism(2450,470)(0,-1)[\phantom{Y_2}` ` \langle 1,\underset{i}{\coprod} subst\rangle]{380}1r
\putmorphism(2450,430)(0,-1)[\phantom{Y_2}`\M(\underset{i}{\coprod} B_i;C),\underset{i}{\coprod}\M(\underset{q}{\coprod} D_{ij}^q; B_i) `]{380}0r
\putmorphism(-250,60)(3,-1)[\phantom{Y_2}``subst]{1260}1l
\putmorphism(2600,80)(-3,-1)[\phantom{Y_2}``subst]{1300}1r
\putmorphism(-250,-310)(1,0)[\phantom{Y_2}`\M(\underset{q}{\coprod} D_{ij}^q; C)`]{1400}0l
\efig}
\end{equation} 

For $\ul{\B}$ to be a bi-multicategory we should fill an analogous diagram to the above one with a non-trivial transformation, 
which we will denote by $\tilde a$,  
between functors between hom-categories $\ul{\B}(A;B)$. The transformation $\tilde a$ when evaluated at an object of the source hom-category 
gives a 2-cell in $\ul{\B}$ that (together with similarly obtained 2-cells for left and right unity in the usual way) should obey a pentagon and a triangle 
from the end of \cite[Definition 4.2.29]{Phil}.

\begin{rem} \rmlabel{tilde-omega-ij}
Let us clarify that $subst: \ul{\B}(\underset{ij}{\coprod} A_{ij}; C)\times\underset{ij}{\Pi}\ul{\B}(\underset{q}{\coprod} D_{ij}^q; A_{ij})
\to\ul{\B}(\underset{q}{\coprod} D_{ij}^q; C)$ (in the bottom left in the above diagram for $\ul{\B}$) is given as follows 
\begin{equation}  \eqlabel{def dbl subst}
\bfig
 \putmorphism(-420,450)(0,-1)[\B(\underset{i}{\textcircled{L}}\big(\underset{j}{\textcircled{L}} A_{ij}\big),C)\times
     \underset{ij}{\Pi}\B(\underset{q}{\textcircled{L}} D_{ij}^q, A_{ij}) ``]{380}0l
 \putmorphism(-420,470)(0,-1)[\phantom{Y_2} ``\B(\omega_a,1)\times\tilde\omega_{ij}]{380}1l	
	\putmorphism(-360,430)(0,-1)[\phantom{Y_2} ` \B(\underset{ij}{\textcircled{L}} A_{ij},C)\times \B(\underset{ijq}{\textcircled{L}} D_{ij}^q, 
	  \underset{ij}{\textcircled{L}} A_{ij})`]{380}0l
	\putmorphism(40,480)(3,-1)[\phantom{Y_2} ` `]{1100}1r
	\putmorphism(-30,530)(3,-1)[\phantom{Y_2} ` `subst]{1100}0r
	\putmorphism(180,90)(1,0)[\phantom{Y_2}`\phantom{Y_2}`comp]{610}1b
	\putmorphism(180,50)(1,0)[\phantom{Y_2}`\B\big(\underset{ijq}{\textcircled{L}} D_{ij}^q, C\big)`]{810}0b
\efig
\end{equation}
where 
$\tilde\omega_{ij}$ is the following composition
$$\underset{ij}{\Pi}\B(\underset{q}{\textcircled{L}} D_{ij}^q, A_{ij})   \xrightarrow{\underset{i}{\Pi} \tilde\omega_j}
\underset{i}{\Pi}\B(\underset{jq}{\textcircled{L}} D_{ij}^q, \underset{j}{\textcircled{L}} A_{ij})  \xrightarrow{\tilde\omega_i}
\B(\underset{ijq}{\textcircled{L}} D_{ij}^q, \underset{i}{\textcircled{L}}\big(\underset{j}{\textcircled{L}} A_{ij}\big)) 
\xrightarrow{\B(1,\omega_a^{-1})}
\B(\underset{ijq}{\textcircled{L}} D_{ij}^q, \underset{ij}{\textcircled{L}} A_{ij})$$
and $\omega_a^{-1}$ is a quasi-inverse. 
Here $\tilde\omega_j=\B(\omega_{d,j},1)can_j$, 
where $\omega_{d,j}:\underset{jq}{\textcircled{L}} D_{ij}^q\to \underset{j}{\textcircled{L}}\big(\underset{q}{\textcircled{L}} D_{ij}^q\big)$, 
so that if $h_{ij}\in\B(\underset{q}{\textcircled{L}} D_{ij}^q, A_{ij})$, then $\tilde\omega_{ij}$ maps the sequence 
$$(h_{ij})_{ij}\mapsto \big((\underset{j}{\textcircled{L}}h_{ij})\omega_{d,j}\big)_i \mapsto
\big(\big(\underset{i}{\textcircled{L}} \big(\underset{j}{\textcircled{L}} h_{ij}\big)\omega_{d,j}\big)\omega_{d,i}\big)
\mapsto
\omega_a^{-1}\big(\big(\underset{i}{\textcircled{L}} \big(\underset{j}{\textcircled{L}} h_{ij}\big)\omega_{d,j}\big)\omega_{d,i}\big).$$
 
For later use (in \rmref{Omega-k}), at this point let us introduce isomorphism 1-cells $\omega^q_{d;i,j}$ 
	by a fixed choice in the diagram below on the right filled by a unique invertible 2-cell, by the axiom (TA1) 
$$
	\scalebox{0.86}{
\bfig
 \putmorphism(-150,470)(1,0)[\underset{ijq}{\textcircled{L}} D_{ij}^q `\underset{ij}{\textcircled{L}}\big(\underset{q}{\textcircled{L}}  
  D_{ij}^q\big) `\omega_{d,ij}]{1250}1a
\putmorphism(-160,470)(0,-1)[\phantom{Y_2}` ` \omega_{d,i}]{380}1l
\putmorphism(-160,430)(0,-1)[\phantom{Y_2}`\underset{i}{\textcircled{L}}\big(\underset{jq}{\textcircled{L}} D_{ij}^q\big)`]{380}0l
\putmorphism(1150,470)(0,-1)[\phantom{Y_2}` `  \omega^q_{d;i,j}]{380}1r
\putmorphism(1150,430)(0,-1)[\phantom{Y_2}`\underset{i}{\textcircled{L}}\big(\underset{j}{\textcircled{L}}  
   \big(\underset{q}{\textcircled{L}} D_{ij}^q\big)\big). `]{380}0r
 \putmorphism(-30,70)(1,0)[\phantom{Y_2}`\phantom{Y_2} `\underset{i}{\textcircled{L}}\omega_{d,j}]{950}1b
\put(360,240){\fbox{$\SWarrow$}}
\efig}
$$
\end{rem}
	
\medskip

We now draw the mentioned diagram that should define $\tilde a$ (recall from \equref{def subst-diag} 
the definition of $subst$): 
\begin{equation} \eqlabel{tilde alfa}
\scalebox{0.86}{
\bfig
 \putmorphism(-150,470)(1,0)[\B(\underset{i}{\textcircled{L}} B_i,C)\times\underset{i}{\Pi}\B(\underset{j}{\textcircled{L}} A_{ij}, B_i)\times
 \underset{ij}{\Pi}\B(\underset{q}{\textcircled{L}} D_{ij}^q, A_{ij}) 
   `\B(\underset{i}{\textcircled{L}} B_i,C)\times\underset{i}{\Pi}\big(\B(\underset{j}{\textcircled{L}} A_{ij}, B_i)\times\underset{j}{\Pi}
	 \B(\underset{q}{\textcircled{L}} D_{ij}^q, A_{ij})\big) `1\times int]{2600}1a
\putmorphism(-160,470)(0,-1)[\phantom{Y_2}` ` 1\times\tilde\omega_i\times\tilde\omega_{ij}]{380}1l
\putmorphism(-160,430)(0,-1)[\phantom{Y_2}`\B(\underset{i}{\textcircled{L}} B_i,C)\times\B(\underset{ij}{\textcircled{L}} A_{ij}, \underset{i}{\textcircled{L}}B_i)\times \B(\underset{ijq}{\textcircled{L}} D_{ij}^q,\underset{ij}{\textcircled{L}} A_{ij})`]{380}0l
\putmorphism(-160,70)(0,-1)[\phantom{Y_2}` ` comp\times 1]{380}1l
\putmorphism(-160,30)(0,-1)[\phantom{Y_2}`\B(\underset{ij}{\textcircled{L}} A_{ij}, C)\times \B(\underset{ijq}{\textcircled{L}} D_{ij}^q,\underset{ij}{\textcircled{L}} A_{ij})`]{380}0l
\putmorphism(-160,-330)(0,-1)[\phantom{Y_2}` ` comp]{380}1l
\putmorphism(-160,-370)(0,-1)[\phantom{Y_2}`\B(\underset{ijq}{\textcircled{L}} D_{ij}^q, C)`]{380}0l
\putmorphism(20,-700)(1,0)[\phantom{Y_2}`\phantom{Y_2}`comp]{1950}{-1}b
	\putmorphism(2500,470)(0,-1)[\phantom{Y_2}`  `  1\times\underset{i}{\Pi}(1\times\tilde\omega_j)]{380}1r
\putmorphism(2500,430)(0,-1)[\phantom{Y_2}`\B(\underset{i}{\textcircled{L}} B_i,C)\times\underset{i}{\Pi}\big(\B(\underset{j}{\textcircled{L}} A_{ij}, B_i)\times
   \B(\underset{jq}{\textcircled{L}} D_{ij}^q, \underset{j}{\textcircled{L}}A_{ij})\big) `  ]{380}0r
\putmorphism(2500,70)(0,-1)[\phantom{Y_2}`  `  1\times\underset{i}{\Pi}comp]{380}1r
\putmorphism(2500,30)(0,-1)[\phantom{Y_2}`\B(\underset{i}{\textcircled{L}} B_i,C)\times\underset{i}{\Pi}\B(\underset{jq}{\textcircled{L}} D_{ij}^q, B_i)` ]{380}0r
\putmorphism(2500,-330)(0,-1)[\phantom{Y_2}`  `  1\times\tilde\omega_i]{380}1r
\putmorphism(2500,-370)(0,-1)[\phantom{Y_2}`\B(\underset{i}{\textcircled{L}} B_i,C)\times\B(\underset{ijq}{\textcircled{L}} D_{ij}^q, 
  \underset{i}{\textcircled{L}}B_i)` ]{380}0r
\putmorphism(-200,80)(3,-1)[\phantom{Y_2}``]{2320}1l
\putmorphism(-200,120)(3,-1)[\phantom{Y_2}``1\times comp]{2320}0r
\put(1000,40){\fbox{$1\times\tilde\gamma$}}
\put(600,-430){\fbox{$\alpha$}}
\efig}
\end{equation}
whereby on the left we used naturality of the Cartesian product. 
Here $\alpha$ is the associativity of the underlying bicategory $\B$, and $\tilde\gamma$ we will determine after the following considerations. 
Take $(f_1,...,f_n)\in\Pi_{i=1}^n\B(\textcircled{L}_{j=1}^{k_i} A_{ij}, B_i)$ and a string of 
$h_{ij}\in \ul{\B}_{l_{ij}}(D_{ij}^{l_1},...,D_{ij}^{l_{ij}};A_{ij})=\B(\textcircled{L}_{q=1}^{l_{ij}} D_{ij}^q, A_{ij})$ 
for $ij\in\{11,...,1k_1, ..., 1n,...,1k_n, ..., ..., n1,..., nk_1, ..., nn, ..., nk_n\}$. 
Then $\tilde\gamma$ evaluated at $\Pi(f,h):=((f_1,...,f_n), (...,h_{ij},...)_{ij})$ will be a 2-cell 
\begin{equation} \eqlabel{tilde-gama}
\tilde\gamma_{\Pi(f,h)}: 
\big(\underset{i}{\textcircled{L}}f_i\big((\underset{j}{\textcircled{L}}h_{ij})\omega_{d,j}\big)\big)\omega_{d,i}\Rightarrow 
 \big((\underset{i}{\textcircled{L}}f_i)\omega_a\big)\big(\omega_a^{-1}\big(\big(\underset{i}{\textcircled{L}} 
\big(\underset{j}{\textcircled{L}} h_{ij}\big)\omega_{d,j}\big)\omega_{d,i}\big)\big)
\end{equation}
acting between 1-cells $\underset{ijq}{\textcircled{L}} D_{ij}^q\to \underset{i}{\textcircled{L}}B_i$. 

For $g\in\B(\textcircled{L}_{i=1}^n B_i,C)$ the composite transformation $\tilde a$ that we are going to define by the above diagram evaluated at $\Pi(g,f,h):=(g,(f_1,...,f_n), (...,h_{ij},...)_{ij})$ will then be a 2-cell mapping 
$$\tilde a_{\Pi(g,f,h)}: 
g\big(
\big(\underset{i}{\textcircled{L}}f_i\big((\underset{j}{\textcircled{L}}h_{ij})\omega_{d,j}\big)\big)\omega_{d,i}  \big) 
\Rightarrow 
 \big(g\big((\underset{i}{\textcircled{L}}f_i)\omega_a\big)\big)\big(\omega_a^{-1}\big(\big(\underset{i}{\textcircled{L}} 
\big(\underset{j}{\textcircled{L}} h_{ij}\big)\omega_{d,j}\big)\omega_{d,i}\big)\big),$$
which are 1-cells $\underset{ijq}{\textcircled{L}} D_{ij}^q\to C$, so that 
\begin{equation} \eqlabel{tilde a}
\tilde\alpha_{\Pi(g,f,h)}=\frac{1_g\circ\tilde\gamma_{\Pi(f,h)}}{\alpha_{g,f,h}}
\end{equation}
where $f$ and $h$ in $\alpha_{g,f,h}$ are given by $f:=\underset{i}{\textcircled{L}} f_i, \,\, h_i:=\underset{j}{\textcircled{L}} h_{ij}$ and 
$h:=\underset{i}{\textcircled{L}} h_i$. With these abbreviations $\tilde\alpha_{\Pi(g,f,h)}$ becomes a 2-cell 
$$\tilde\alpha_{\Pi(g,f,h)}: g\big(\big(\underset{i}{\textcircled{L}}f_i\big(h_i\omega_{d,j}\big)\big)\omega_{d,i}  \big) 
\Rightarrow 
 \big(g\big(f\omega_a\big)\big)\big(\omega_a^{-1}\big(\big(\underset{i}{\textcircled{L}} h_i \omega_{d,j}\big)\omega_{d,i}\big)\big).$$

\medskip

We define $\tilde\gamma$ by the following diagram 
$$
\scalebox{0.86}{
\bfig
 \putmorphism(-350,470)(1,0)[\underset{i}{\Pi}\B(\underset{j}{\textcircled{L}} A_{ij}, B_i)\times
 \underset{ij}{\Pi}\B(\underset{q}{\textcircled{L}} D_{ij}^q, A_{ij}) 
   `\underset{i}{\Pi}\big(\B(\underset{j}{\textcircled{L}} A_{ij}, B_i)\times\underset{j}{\Pi}
	 \B(\underset{q}{\textcircled{L}} D_{ij}^q, A_{ij})\big) `int]{2570}1a
\putmorphism(-800,470)(0,-1)[\phantom{Y_2}` `\tilde\omega_i\times\tilde\omega_{ij}]{1200}1r
\putmorphism(-160,450)(3,-1)[\phantom{Y_2}` ` 1\times\underset{i}{\Pi}can_j]{980}1r
\putmorphism(240,200)(3,-1)[\phantom{Y_2}`\underset{i}{\Pi}\B(\underset{j}{\textcircled{L}} A_{ij}, B_i)\times \underset{i}{\Pi}\B(\underset{j}{\textcircled{L}}\big(\underset{q}{\textcircled{L}} D_{ij}^q\big),\underset{j}{\textcircled{L}} A_{ij})`]{380}0l
\putmorphism(670,70)(0,-1)[\phantom{Y_2}` `  can_i\times can_i]{380}1l
\putmorphism(600,30)(0,-1)[\phantom{Y_2}`\B(\underset{i}{\textcircled{L}}\big(\underset{j}{\textcircled{L}} A_{ij}\big), 
   \underset{i}{\textcircled{L}} B_i)\times\B(\underset{i}{\textcircled{L}}\big(\underset{j}{\textcircled{L}}\big(\underset{q}{\textcircled{L}} D_{ij}^q\big)\big), \underset{i}{\textcircled{L}}\big(\underset{j}{\textcircled{L}} A_{ij}\big))` ]{380}0l
\putmorphism(670,-350)(-2,-1)[\phantom{Y_2}` ` \B(\omega_a, 1)\times \B((\underset{i}{\textcircled{L}}\omega_{d,j})\omega_{d,i}, 
  \omega_a^{-1})]{780}1l  
\putmorphism(-160,-410)(0,-1)[\phantom{Y_2}` \B(\underset{ij}{\textcircled{L}} A_{ij},
 \underset{i}{\textcircled{L}} B_i)\times\B(\underset{ijq}{\textcircled{L}}D_{ij}^q, \underset{ij}{\textcircled{L}}A_{ij})`]{380}0l
\putmorphism(-160,-770)(0,-1)[\phantom{Y_2}` ` comp]{400}1l
\putmorphism(-160,-840)(0,-1)[\phantom{Y_2}`\B(\underset{ijq}{\textcircled{L}}D_{ij}^q, \underset{i}{\textcircled{L}}B_i)`]{380}0l
	\putmorphism(2220,480)(0,-1)[\phantom{Y_2}`  `  \underset{i}{\Pi}(1\times can_j)]{380}1r
\putmorphism(2310,450)(0,-1)[\phantom{Y_2}` \underset{i}{\Pi}\big(\B(\underset{j}{\textcircled{L}} A_{ij}, B_i)\times
   \B(\underset{j}{\textcircled{L}}\big(\underset{q}{\textcircled{L}} D_{ij}^q\big), \underset{j}{\textcircled{L}}A_{ij})\big) `  ]{380}0r
\putmorphism(2250,70)(0,-1)[\phantom{Y_2}`  `  \underset{i}{\Pi}comp]{380}1r
\putmorphism(2280,30)(0,-1)[\phantom{Y_2}`\underset{i}{\Pi}\B(\underset{j}{\textcircled{L}}\big(\underset{q}{\textcircled{L}} D_{ij}^q\big), 
   \underset{i}{\textcircled{L}} B_i)` ]{380}0r
\putmorphism(2160,-330)(-2,-1)[\phantom{Y_2}`  ` can_i]{760}1l
\putmorphism(510,-360)(2,-1)[\phantom{Y_2}``]{700}1l 
\putmorphism(320,-230)(2,-1)[\phantom{Y_2}`` comp]{1140}0r 
\putmorphism(300,-290)(2,-1)[\phantom{Y_2}`\B(\underset{i}{\textcircled{L}}\big(\underset{j}{\textcircled{L}}\big(\underset{q}{\textcircled{L}} 
   D_{ij}^q\big)\big), \underset{i}{\textcircled{L}} B_i)`]{1000}0l
 \putmorphism(1320,-790)(0,-1)[\phantom{Y_2}``]{380}1l
 \putmorphism(1320,-840)(0,-1)[\phantom{Y_2}`\B(\underset{i}{\textcircled{L}}\big(\underset{jq}{\textcircled{L}} 
   D_{ij}^q\big), \underset{i}{\textcircled{L}} B_i)`\B(\underset{i}{\textcircled{L}}\omega_{d,j},1)]{380}0l
\putmorphism(60,-1180)(-1,0)[\phantom{Y_2}``\B(\omega_{d,i},1)]{900}{-1}r
\putmorphism(2740,70)(0,-1)[\phantom{Y_2}`  `  \underset{i}{\Pi}(1\times\tilde\omega_{d,j})]{680}1r
\putmorphism(2540,-290)(0,-1)[\phantom{Y_2}`\underset{i}{\Pi}\big(\B(\underset{j}{\textcircled{L}} A_{ij}, B_i)\times
   \B(\underset{jq}{\textcircled{L}} D_{ij}^q, \underset{j}{\textcircled{L}}A_{ij})\big)` ]{380}0r
\putmorphism(2740,-730)(0,-1)[\phantom{Y_2}``  \underset{i}{\Pi}comp]{440}1r
\putmorphism(2740,-840)(0,-1)[\phantom{Y_2}` \underset{i}{\Pi}\B(\underset{jq}{\textcircled{L}} D_{ij}^q, \B_i)`  
    ]{380}0r
\putmorphism(1650,-1180)(-1,0)[\phantom{Y_2}`` can_i]{780}{-1}r
\put(-540,140){\fbox{$\theta$}}
\put(1260,-100){\fbox{$\gamma$}}
\put(320,-990){\fbox{$\omega$}}
\put(2200,-990){\fbox{$\Epsilon$}}
\efig}
$$
whereby $\tilde\omega_{d,j}$ on the outer right vertical edge stands shortly for $\B(\tilde\omega_{d,j},1)$. 
Let us first explain the transformation $\gamma$. We define it to be the extension of the interchange transformation of $\B$, seen as a 
one-object tricategory, to more than two variables. Namely, with $\Pi(f,h)$ as above we find that $\gamma_{\Pi(f,h)}$ is a 2-cell 
$$\gamma_{\Pi(f,h)}: \underset{i}{\textcircled{L}}f_i(\underset{j}{\textcircled{L}}h_{ij})\Rightarrow 
 (\underset{i}{\textcircled{L}}f_i)\big(\underset{i}{\textcircled{L}}\big(\underset{j}{\textcircled{L}} h_{ij}\big)\big)$$
between 1-cells $\underset{i}{\textcircled{L}}\big(\underset{j}{\textcircled{L}} \big(\underset{q}{\textcircled{L}} D_{ij}^q\big)\big) 
\to \underset{i}{\textcircled{L}}\big(\underset{j}{\textcircled{L}} A_{ij}) \to \underset{i}{\textcircled{L}} B_i$. 
In shorter terms, we have
$$\gamma_{\Pi(f,h)}: \underset{i}{\textcircled{L}}f_ih_i\Rightarrow (\underset{i}{\textcircled{L}}f_i)(\underset{i}{\textcircled{L}}h_i).$$
The 2-cell $\gamma_{\Pi(f,h)}$ is the vertical composition of 2-cells 
\begin{multline*}
\gamma_{\Pi(f,h)}: \textcircled{L}_{i=1}^n f_ih_i  \stackrel{\xi\textcircled{L} 1_{n-2}}{\Rightarrow} 
\textcircled{L}_{i=1}^{n-1} f^1_i h^1_i \stackrel{\xi\textcircled{L} 1_{n-3}}{\Rightarrow} ...
\stackrel{\xi\ot 1}{\Rightarrow} 
(f^{n-1} h^{n-1})\ot (f_nh_n) \stackrel{\xi}{\Rightarrow} \\
(f^{n-1}\ot f_n)(h^{n-1}\ot h_n) =
 \big(\textcircled{L}_{i=1}^n f_i\big)\big(\textcircled{L}_{i=1}^n h_i\big)
\end{multline*}
whereby $\xi$ denotes the monoidality interchange from \axiomref{MB4}, 
$f^1_1 h^1_1=(f_1\ot f_2)(h_1\ot h_2)$ and the rest $f^1_i h^1_i$ for $i=2,...,n-1$ are $f_{i+1}h_{i+1}$, and so on, ending in 
$f^{n-1}_{n-1} h^{n-1}_{n-1}= \big(\textcircled{L}_{i=1}^n f_i\big)\big(\textcircled{L}_{i=1}^nh_i\big).$
 

\medskip 

To define the two-variable transformation $\theta$, we set its right-hand component at $(h_{ij})_{ij}$ to be given by the composite 2-cell 
(we can call it $\theta^*$)
\begin{equation} \eqlabel{teta star}
\omega_a^{-1}\big[ (\underset{i}{\textcircled{L}} \big(\underset{j}{\textcircled{L}} h_{ij}\big))
 \big((\underset{i}{\textcircled{L}}\omega_{d,j})\omega_{d,i}\big)\big]   \stackrel{1\alpha}{\Rightarrow}
\omega_a^{-1}\big[ \big((\underset{i}{\textcircled{L}} \big(\underset{j}{\textcircled{L}} h_{ij}\big))
 (\underset{i}{\textcircled{L}}\omega_{d,j})\big)\omega_{d,i}\big]    \stackrel{1(\gamma^{-1}_{h_i, \omega_{d,j}}1)}{\Rightarrow}
\omega_a^{-1}\big[ \big(\underset{i}{\textcircled{L}} \big((\underset{j}{\textcircled{L}} h_{ij})\omega_{d,j}\big)\big)\omega_{d,i}\big]
\end{equation}
whereby we recall $\underset{ijq}{\textcircled{L}}D_{ij}^q\stackrel{\omega_{d,i}}{\to} \underset{i}{\textcircled{L}} 
  \big(\underset{jq}{\textcircled{L}} D_{ij}^q\big) \stackrel{\underset{i}{\textcircled{L}} \omega_{d,j}}{\to} 
	   \underset{i}{\textcircled{L}} \big(\underset{j}{\textcircled{L}}\big(\underset{q}{\textcircled{L}}D_{ij}^q\big)\big)$. 
\noindent	This 2-cell $\theta^*$ induces a natural isomorphism transformation 
$$\theta: 
\big( \B(\omega_a, id)\times \B(\omega_{d,i}\comp\omega_{d,j},\omega_a^{-1})\big) (can_i\times can_i)(1\times \underset{i}{\Pi}can_j) 
\Rightarrow  \tilde\omega_i\times\tilde\omega_{ij}$$
mapping
$$\theta_{\Pi(f,h)}: \big( (\underset{i}{\textcircled{L}} f_i)\omega_a, \,\, 
   \omega_a^{-1}\big[ (\underset{i}{\textcircled{L}} \big(\underset{j}{\textcircled{L}} h_{ij}\big))
 \big((\underset{i}{\textcircled{L}}\omega_{d,j})\omega_{d,i}\big)\big] \,\, \big)  	\stackrel{(1,\theta^*)}{\Rightarrow} 
\big( (\underset{i}{\textcircled{L}} f_i)\omega_a, \,\, 
   \omega_a^{-1}\big[ \big(\underset{i}{\textcircled{L}} \big((\underset{j}{\textcircled{L}} h_{ij})\omega_{d,j}\big)\big)\omega_{d,i}\big]  \,\, \big).$$

Observe that the diagram for the to-be-defined transformation $\omega$ commutes strictly when regarding only the naturality of the functor $comp$ with respect to the morphism $(\underset{i}{\textcircled{L}}\omega_{d,j})\omega_{d,i}$ (the variable 
$\underset{i}{\textcircled{L}}\big(\underset{j}{\textcircled{L}}
\big(\underset{q}{\textcircled{L}} D_{ij}^q\big)\big)$), but it commutes up to an isomorphism when regarding that the variable 
$\underset{i}{\textcircled{L}}\big(\underset{j}{\textcircled{L}} A_{ij}\big)$ is affected by the isomorphism transformation 
$comp\Rightarrow comp\circ\big(\B(\omega_a,1)\times\B(1,\omega_a^{-1})\big)$ determined by a unique coherence 2-cell acting 
between $\xi\zeta\Rightarrow(\xi\omega_a)(\omega_a^{-1}\zeta)$, as a morphism in the codomain 
$\B(\underset{ijq}{\textcircled{L}} D_{ij}^q, \underset{i}{\textcircled{L}}B_i)$. It is by this isomorphism transformation 
(and associativities) that we induce $\omega$
$$\omega_{\underset{i}{\textcircled{L}} f_i, \underset{i}{\textcircled{L}} (\underset{j}{\textcircled{L}} h_{ij})}: 
\big(\, \big((\underset{i}{\textcircled{L}} f_i)(\underset{i}{\textcircled{L}} \big(\underset{j}{\textcircled{L}} h_{ij}\big))\big)
\big(\underset{i}{\textcircled{L}} \omega_{d,j}\big) \, \big)\omega_{d,i} \Rightarrow 
\big((\underset{i}{\textcircled{L}} f_i)\omega_a\big) \big(\omega_a^{-1}\big((\underset{i}{\textcircled{L}} \big(\underset{j}{\textcircled{L}} h_{ij}\big) 
\big((\underset{i}{\textcircled{L}} \omega_{d,j})\omega_{d,i}\big)\big).$$
Recalling that $f=\underset{i}{\textcircled{L}} f_i, \,\, h_i=\underset{j}{\textcircled{L}} h_{ij}$ and $h=\underset{i}{\textcircled{L}} h_i$, we can write 
this 2-cell shorter as 
$$\omega_{f, h}: \big((fh)\big(\underset{i}{\textcircled{L}} \omega_{d,j}\big) \, \big)\omega_{d,i} \Rightarrow 
\big(f\omega_a\big) \big(\omega_a^{-1}\big(h\big((\underset{i}{\textcircled{L}} \omega_{d,j})\omega_{d,i}\big)\big)\big).$$

\smallskip

Finally, the transformation $\Epsilon$ gets evaluated at a sequence $(f_i, \underset{j}{\textcircled{L}}h_{ij})_i$ and is defined 
as the vertical composition of 2-cells  
$$\Epsilon_{(f_i,h_i)_i}:  \underset{i}{\textcircled{L}} f_i(h_i\omega_{d,j}) 
\stackrel{\underset{i}{\textcircled{L}} \alpha_{f_i,h_i,\omega_{d,j}}}{\Rightarrow}
\underset{i}{\textcircled{L}} (f_ih_i)\omega_{d,j}
\stackrel{\gamma_{f_ih_i,\omega_{d,j}}}{\Rightarrow}
\big(\underset{i}{\textcircled{L}}(f_ih_i)\big)(\underset{i}{\textcircled{L}}\omega_{d,j}).$$

\medskip

Summing up, the transformation $\tilde a$ is a composition of the transformations: $\alpha$ (associativity of $\B$), $\gamma$ (many-variable extension of the interchange $\xi$ of the monoidal bicategory $\B$), $\Epsilon$ (composition of $\underset{i}{\textcircled{L}} \alpha$ 
and $\gamma$) and transformations $\theta$ and $\omega$ induced by the unique coherence isomorphisms enabling that all the compositions 
make sense. We saw in \equref{tilde a} that $\tilde\alpha_{\Pi(g,f,h)}=\frac{1_g\circ\tilde\gamma_{\Pi(f,h)}}{\alpha}$, and 
$\tilde\gamma_{\Pi(f,h)}$ is the following composition of 2-cells: 
\begin{multline*}
\tilde\gamma_{\Pi(f,h)}: \quad 
\big(\underset{i}{\textcircled{L}} f_i\big(h_i\omega_{d,j}\big)\big)\omega_{d,i} 
\stackrel{\Epsilon_{(f_i,h_i)_i}\circ 1}{\Rightarrow}
\big(\big(\underset{i}{\textcircled{L}}f_ih_i\big)(\underset{i}{\textcircled{L}}\omega_{d,j})\big)\omega_{d,i} 
\stackrel{(\gamma_{\Pi(f,h)}\circ 1)\circ 1}{\Rightarrow}
\big((fh)(\underset{i}{\textcircled{L}}\omega_{d,j})\big)\omega_{d,i} \\
\stackrel{\omega_{f, h}}{\Rightarrow}
\big(f\omega_a\big) \big(\omega_a^{-1}\big[ h \big((\underset{i}{\textcircled{L}}\omega_{d,j})\omega_{d,i}\big)\big]\big)
\stackrel{\theta_{\Pi(f,h)}}{\Rightarrow}
\big( f\omega_a\big) 
\big(\omega_a^{-1}\big[ \big(\underset{i}{\textcircled{L}} h_i\omega_{d,j}\big)\omega_{d,i}\big]\big).
\end{multline*}

As a matter of fact, the pasting of the diagrams for the transformations $1\times\omega$ and $\alpha$ in the diagram defining $\tilde\alpha$ 
is equal to a pasting of the next four diagrams: $nat^*: comp(1\times\B(\omega_{d,i},1))(1\times\B(\underset{i}{\textcircled{L}}\omega_{d,j},1))\Rightarrow
\B(\omega_{d,i},1)\circ\B(\underset{i}{\textcircled{L}}\omega_{d,j},1)\circ comp$, the diagrams for $\alpha$ and $\omega$ and 
$nat':(\B(\omega_a, 1)\times \B((\underset{i}{\textcircled{L}}\omega_{d,j})\omega_{d,i}, \omega_a^{-1}))(comp\times 1)
\Rightarrow(comp\times 1)(\B(\omega_a, 1)\times \B((\underset{i}{\textcircled{L}}\omega_{d,j})\omega_{d,i}, \omega_a^{-1}))$, 
so that the 2-cell \equref{tilde a} can be obtained as the following vertical composition 
\begin{multline*}
\tilde\alpha_{\Pi(g,f,h)}: \quad 
g\big(
\big(\underset{i}{\textcircled{L}}f_i\big(h_i\omega_{d,j}\big)\big)\omega_{d,i}  \big) 
\stackrel{1(\Epsilon_{(f_i,h_i)_i}\circ 1)}{\Rightarrow}
g\big(\big(\big(\underset{i}{\textcircled{L}}f_ih_i\big)(\underset{i}{\textcircled{L}}\omega_{d,j})\big)\omega_{d,i}\big) \\
\stackrel{nat^*}{\Rightarrow}
\big(\big(g\big(\underset{i}{\textcircled{L}}f_ih_i\big)\big)(\underset{i}{\textcircled{L}}\omega_{d,j})\big)\omega_{d,i}
\stackrel{(\frac{(1\times\gamma_{\Pi(f,h)}}{\alpha}\vert_{\Pi(g,f,h)} \circ 1)\circ 1}{\Longrightarrow} 
\big(((gf)h)(\underset{i}{\textcircled{L}}\omega_{d,j})\big)\omega_{d,i} \\
\stackrel{\omega_{gf, h}}{\Rightarrow}
\big((gf)\omega_a\big) \big(\omega_a^{-1}\big(h\big((\underset{i}{\textcircled{L}} \omega_{d,j})\omega_{d,i}\big)\big)\big)
\stackrel{\theta_{\Pi(gf,h)}}{\Rightarrow}
 \big((gf)\omega_a\big)  \big(\omega_a^{-1}\big[ \big(\underset{i}{\textcircled{L}} h_i\omega_{d,j}\big)\omega_{d,i}\big]\big) \\
\stackrel{a^{-1}_{g,f,\omega_a}\circ 1}{\Rightarrow}
 \big(g(f\omega_a)\big)  \big(\omega_a^{-1}\big[ \big(\underset{i}{\textcircled{L}} h_i\omega_{d,j}\big)\omega_{d,i}\big]\big)
\end{multline*}
whereby $nat^*$ is given by the evident consecutive application of suitable associativity 2-cells. 
For practical purposes let us write this composition as follows 
\begin{multline} \label{tilde a decomposed}
\tilde\alpha_{\Pi(g,f,h)}: \quad 
g\big(
\big(\underset{i}{\textcircled{L}}f_i(h_i\omega_{d,j})\big)\omega_{d,i}  \big) 
\stackrel{enter}{\Rightarrow}
\big(\big(g\big(\underset{i}{\textcircled{L}}f_ih_i\big)\big)(\underset{i}{\textcircled{L}}\omega_{d,j})\big)\omega_{d,i} \\
\stackrel{(\frac{(1\times\gamma_{\Pi(f,h)}}{\alpha}\vert_{\Pi(g,f,h)} \circ 1)\circ 1}{\Longrightarrow} 
\big(((gf)h)(\underset{i}{\textcircled{L}}\omega_{d,j})\big)\omega_{d,i} 
\quad  \stackrel{exit}{\Rightarrow}  \quad 
(gf)   \big( \big(\underset{i}{\textcircled{L}} h_i\omega_{d,j}\big)\omega_{d,i}\big)
\end{multline}
treating the composition $\omega_a\omega_a^{-1}$ as equal to identity (recall that $\omega_a^{-1}$ is a quasi-inverse). 

\subsection{The pentagon axiom for the associativity } 	

In the pentagon axiom for the associativity one applies successively $\tilde\alpha$ and evaluates it at four composable multimaps. 
Before dedicating to the axiom itself, let us take a careful look at these multimaps at different levels of the composition.

\subsubsection{Preparing the terrain} \sslabel{prep}

With $g\in \ul{\B}_n(B_1,...,B_n;C)=\B(\textcircled{L}_{i=1}^n B_i,C)$, \,\, 
$(f_1,...,f_n)\in\Pi_{i=1}^n\B(\textcircled{L}_{j=1}^{k_i} A_{ij}, B_i)$ and the $h_{ij}$'s from 
$\ul{\B}_{l_{ij}}(D_{ij}^{l_1},...,D_{ij}^{l_{ij}};A_{ij})=\B(\textcircled{L}_{q=1}^{l_{ij}} D_{ij}^q, A_{ij})$, 
where we set 
$$f:=\underset{i}{\textcircled{L}} f_i: \underset{i}{\textcircled{L}}\big(\underset{j}{\textcircled{L}} A_{ij}\big) 
\to \underset{i}{\textcircled{L}}B_i,$$ 
$$h_i=\underset{j}{\textcircled{L}} h_{ij}: \underset{j}{\textcircled{L}}\big(\underset{q}{\textcircled{L}} D_{ij}^q\big) \to
 \underset{j}{\textcircled{L}} A_{ij}, \qquad 
h=\underset{i}{\textcircled{L}} h_i: \underset{i}{\textcircled{L}} \big(\underset{j}{\textcircled{L}}\big(\underset{q}{\textcircled{L}} D_{ij}^q\big)\big) \to \underset{i}{\textcircled{L}} \big( \underset{j}{\textcircled{L}} A_{ij}\big),$$  
as before, a pentagon axiom involves a further sequence of 1-cells $k_{ij}^q$'s mapping 
$k_{ij}^q: \underset{l}{\textcircled{L}} E_{ij}^{q,l}\to D_{ij}^q$, and we set in a similar fashion 
$$k_{ij}:=\underset{q}{\textcircled{L}} k_{ij}^q: \underset{q}{\textcircled{L}}(\underset{l}{\textcircled{L}} E_{ij}^{q,l})\to 
\underset{q}{\textcircled{L}} D_{ij}^q, \qquad 
k_i:=\underset{j}{\textcircled{L}} k_{ij}: 
\underset{j}{\textcircled{L}}(\underset{q}{\textcircled{L}}(\underset{l}{\textcircled{L}} E_{ij}^{q,l}))\to 
\underset{j}{\textcircled{L}}(\underset{q}{\textcircled{L}} D_{ij}^q),$$ 
and finally 
$$k:=\underset{i}{\textcircled{L}} k_i : \underset{i}{\textcircled{L}} \big( 
\underset{j}{\textcircled{L}}\big(\underset{q}{\textcircled{L}}\big(\underset{l}{\textcircled{L}} E_{ij}^{q,l}\big)\big)\big)\to 
\underset{i}{\textcircled{L}}\big(\underset{j}{\textcircled{L}}\big(\underset{q}{\textcircled{L}} D_{ij}^q\big)\big).$$

\begin{rem} \rmlabel{Omega-k}
As opposed to $k$ (and similarly for $h$ and $f$), we differentiate 
$$\tilde k:= \left( \underset{ijql}{\textcircled{L}} E_{ij}^{q,l} \xrightarrow{\omega_{e,i}} 
\underset{i}{\textcircled{L}}\big( \underset{jql}{\textcircled{L}} E_{ij}^{q,l}\big) \xrightarrow{\underset{i}{\textcircled{L}}\omega_{e,j}} 
\underset{i}{\textcircled{L}}\big(\underset{j}{\textcircled{L}}\big(\underset{ql}{\textcircled{L}} E_{ij}^{q,l}\big)\big) 
 \xrightarrow{\underset{i}{\textcircled{L}}\big(\underset{j}{\textcircled{L}} \omega_{e;q,l}\big)} 
\underset{i}{\textcircled{L}} \big( 
\underset{j}{\textcircled{L}}\big(\underset{q}{\textcircled{L}}\big(\underset{l}{\textcircled{L}} E_{ij}^{q,l}\big)\big)\big)
\stackrel{k}{\to} \underset{i}{\textcircled{L}}\big(\underset{j}{\textcircled{L}}\big(\underset{q}{\textcircled{L}} D_{ij}^q\big)\big) \right).$$
Set $$\Omega^k:=\big(\underset{i}{\textcircled{L}}\big(\underset{j}{\textcircled{L}} \omega_{e;q,l}\big)\big)
\big(\underset{i}{\textcircled{L}} \omega_{e,j}\big)\omega_{e,i}$$
so that $\tilde k=k\Omega^k$, and similarly $\tilde h=h\Omega^h$ and $\tilde f=f\Omega^f=f\omega_a$. 
(Observe that there is a 2-cell $\gamma^{-1}1$ from $\Omega^k$ to 
$\big(\underset{i}{\textcircled{L}} \big((\underset{j}{\textcircled{L}}\omega_{e;q,l})\omega_{e,j}\big)\big)\omega_{e,i}$.) We 
will use the 1-cell $\Omega^k$ at the end of \sssref{comput pentag}. 


\bigskip

Observe that although we have composable strings of 1-cells $\Omega^k, k, h, f, g$ and similarly $\Omega^h, h, f, g$, 
the composition of the latter does not appear neither in the domain nor in the codomain of $\tilde a_{\Pi(g,f,h)}$, but it does appear 
(up to associativity) 
in the domain of the 2-cell $exit$ in the composition defining $\tilde\alpha_{\Pi(g,f,h)}$ (see (\ref{tilde a decomposed})). 
\end{rem}

\subsubsection{Proving the pentagon axiom} \ssslabel{comput pentag}

Our strategy to prove the  pentagon axiom for $\tilde\alpha$ consists of the following four steps: 
\begin{enumerate}
\item show that the ``ternary hexagon'' for $\gamma$ commutes, 
\item show that the ``$i$-ary hexagon for $\gamma$'' commutes
\item show that the pentagon for $\frac{1\gamma}{\alpha}$ commutes, 
\item divide the pentagon for $\tilde\alpha_{\Pi(g,f,h)}=\frac{1_g\circ\tilde\gamma_{\Pi(f,h)}}{\alpha_{g,f,h}}$ into a pentagon obtained from the pentagon for $\frac{1\gamma}{\alpha}$ and diagrams commuting by naturality, pentagon for $\alpha$ and $i$-ary hexagons. 
\end{enumerate}

By the hexagon we mean the axiom for the pseudofunctor $\ot:\B\times\B\to\B$ on composable 1-cells $(f,f'), (h,h'), (k,k')$, which is the law for the interchange $\xi$ between the horizontal composition and the monoidal product of $\B$. It concerns 1-cells of the form $f\ot f'$ and their compositions. By the ``ternary hexagon'' we mean 
its extension to $(f\ot f')\ot f''$ and compositions of such 1-cells. By the ``$i$-ary hexagon for $\gamma$'' we mean the hexagonal rule 
as in the picture below, where $f_i, h_i, k_i$ are as in \ssref{prep} (a more appropriate name would be 
an ``$n$-ary hexagon'', with $i=1,...,n$, but the name we chose is more allusive). 
\[\begin{tikzcd}
	{\underset{i}{\textcircled{L}} f_i(h_ik_i)} & {f(\underset{i}{\textcircled{L}}h_ik_i)} \\
	{\underset{i}{\textcircled{L}} (f_i h_i)k_i} & {f(hk)} \\
	{(\underset{i}{\textcircled{L}} f_i h_i)k} & {(fh)k}
	\arrow["{\gamma_{.,..}1}", Rightarrow, nfold, from=1-1, to=1-2]
	\arrow["{\underset{i}{\textcircled{L}}\alpha}"', Rightarrow, nfold, from=1-1, to=2-1]
	\arrow["{1\gamma}", Rightarrow, nfold, from=1-2, to=2-2]
	\arrow["{\gamma_{..,.}}"', Rightarrow, nfold, from=2-1, to=3-1]
	\arrow["\alpha", Rightarrow, nfold, from=2-2, to=3-2]
	\arrow["{\gamma 1}"', Rightarrow, nfold, from=3-1, to=3-2]
\end{tikzcd}\]
The ternary hexagon is easily directly proved. We have in this case $\gamma=\xi(\xi\ot 1)$, by definition of $\gamma$. One divides the ternary hexagon 
in four subdiagrams (whose edges are 2-cells). If the starting edge is the 1-cell $(f(hk)\ot f'(h'k'))\ot f''(h''k'')$, the first, upper-left  subdiagram is 
``hexagon on the starting edge $f(hk)\ot f'(h'k')$ tensor $a_{f'',h'',k''}$, the upper-right and lower-left subdiagrams commute by naturality of $\xi$ in one and the other variable, and finally the fourth subdiagram (on bottom-right) is another simple (binary) hexagon. 

The ``$i$-ary hexagon for $\gamma$'' could be written out and divided into many subdiagrams commuting, by similar reasons as in the ternary case, but this is tedious. 
Instead, we may use induction in $n$, with $f=\textcircled{L}_{i=1}^n f_i$, starting with the edge 
$\big(\textcircled{L}_{i=1}^{n-1} f_i(h_ik_i)\big)\ot f_n(h_nk_n)$, see the diagram \ref{hex_i} in the Appendix. 
One divides the corresponding hexagon - similarly as in the case of the ternary hexagon - 
into four subdiagrams: ``hexagon for $n-1$ tensor identity'' (commuting by the induction hypothesis), upper-right subdiagram with the starting edge 
$(\textcircled{L}_{i=1}^{n-1} f_i)(\textcircled{L}_{i=1}^{n-1} h_ik_i) \ot f_n(h_nk_n)$ commutes by naturality of $\xi$ with respect to $\gamma$, 
the bottom-left subdiagram commutes obviously: it is $(\gamma 1\ot 1)(1\ot\alpha)=(1\ot \alpha)(\gamma 1\ot 1)$, and finally the bottom-right subdiagram is the usual (binary hexagon). 

We prove that the pentagon for $\frac{1\gamma}{\alpha}$ commutes by dividing it into subdiagrams commuting for the indicated reasons:
\[\begin{tikzcd}
	{g(\underset{i}{\textcircled{L}} f_i\big(\underset{j}{\textcircled{L}} h_{ij}k_{ij}\big))} & {g(f\big(\underset{i}{\textcircled{L}}\big(\underset{j}{\textcircled{L}} h_{ij}k_{ij}\big)\big))} & {(gf)\big(\underset{i}{\textcircled{L}}\big(\underset{j}{\textcircled{L}} h_{ij}k_{ij}\big)\big)} \\
	{g(\underset{i}{\textcircled{L}} f_i\big((\underset{j}{\textcircled{L}} h_{ij})(\underset{j}{\textcircled{L}} k_{ij})\big))} & {g(f(\underset{i}{\textcircled{L}} h_ik_i))} & {(gf)\big(\underset{i}{\textcircled{L}}h_ik_i\big)} \\
	{g(\underset{i}{\textcircled{L}} (f_i h_i)k_i)} & {g(f(hk))} & {(gf)(hk)} \\
	{g((\underset{i}{\textcircled{L}} f_i h_i)k_i)} & {g((fh)k)} & {((gf)h)k} \\
	{(g(\underset{i}{\textcircled{L}} f_i h_i))k} & {(g(fh))k}
	\arrow[""{name=0, anchor=center, inner sep=0}, "{1\gamma_{.,..}}", Rightarrow, nfold, from=1-1, to=1-2]
	\arrow["{1(\underset{i}{\textcircled{L}}1\gamma)}"', Rightarrow, nfold, from=1-1, to=2-1]
	\arrow[""{name=1, anchor=center, inner sep=0}, "{\alpha_{., ., ..}}", Rightarrow, nfold, from=1-2, to=1-3]
	\arrow["{1(1(\underset{i}{\textcircled{L}}1\gamma))}", Rightarrow, nfold, from=1-2, to=2-2]
	\arrow["{1(\underset{i}{\textcircled{L}}\gamma)}", Rightarrow, nfold, from=1-3, to=2-3]
	\arrow[""{name=2, anchor=center, inner sep=0}, "{1\gamma_{., ..}}"', Rightarrow, nfold, from=2-1, to=2-2]
	\arrow["{1(\underset{i}{\textcircled{L}}\alpha)}"', Rightarrow, nfold, from=2-1, to=3-1]
	\arrow["{1(1\gamma)}", Rightarrow, nfold, from=2-2, to=3-2]
	\arrow["{1\gamma}", Rightarrow, nfold, from=2-3, to=3-3]
	\arrow["{1\gamma_{..,.}}"', Rightarrow, nfold, from=3-1, to=4-1]
	\arrow[""{name=3, anchor=center, inner sep=0}, "{\alpha_{.,.,..}}", Rightarrow, nfold, from=3-2, to=3-3]
	\arrow["{1\alpha}", Rightarrow, nfold, from=3-2, to=4-2]
	\arrow["{\alpha_{..,.,.}}", Rightarrow, nfold, from=3-3, to=4-3]
	\arrow[""{name=4, anchor=center, inner sep=0}, "{1(\gamma 1)}", Rightarrow, nfold, from=4-1, to=4-2]
	\arrow["{ \alpha_{.,..,.}}"', Rightarrow, nfold, from=4-1, to=5-1]
	\arrow["{\alpha_{.,..,.}}", Rightarrow, nfold, from=4-2, to=5-2]
	\arrow[""{name=5, anchor=center, inner sep=0}, "{(1\gamma)1}"', Rightarrow, nfold, from=5-1, to=5-2]
	\arrow[""{name=6, anchor=center, inner sep=0}, "{\alpha 1}"', Rightarrow, nfold, from=5-2, to=4-3]
	\arrow["{nat_{\gamma_{.,\bullet}}}"{description, pos=0.4}, draw=none, from=0, to=2]
	\arrow["{nat_{\alpha_{.,.,\bullet}}}"{description}, draw=none, from=1, to=3]
	\arrow["{1 hex_i}"{description}, draw=none, from=2, to=4]
	\arrow["{pent. \B}"{description}, draw=none, from=3, to=6]
	\arrow["{nat_{\alpha_{.,\gamma,.}}}"{description}, draw=none, from=4, to=5]
\end{tikzcd}\]
where $nat_{\gamma_{.,\bullet}}$ denotes naturality of $\gamma$ in the second variable with respect to $\gamma$, and 
$nat_{\alpha_{.,.,\bullet}}$ denotes naturality of $\alpha$ in the third variable with respect to 
$\frac{\underset{i}{\textcircled{L}}\gamma}{\gamma}$. 

\medskip

The final step is to prove the pentagon for $\tilde\alpha_{\Pi(g,f,h)}=\frac{1_g\circ\tilde\gamma_{\Pi(f,h)}}{\alpha_{g,f,h}}$. 
This pentagon starts with the vertex  $g\big[\big(\underset{i}{\textcircled{L}} f_i\big[\big(\underset{j}{\textcircled{L}} h_{ij}
(k_{ij}\omega_{e;q,l})\big)\omega_{e,j}\big]\big)\omega_{e,i}\big]$ as in the picture below. To apply $\tilde\alpha_{.,.,..}$ to it 
in the form of 
(\ref{tilde a decomposed}), we should apply $enter_{.,.,..}, (\frac{1\times\gamma_{.,..}}{\alpha_{.,.,..}}1)1$ and 
 $exit_{.,.,..}$. Since there are more terms now, $enter$ and $exit$ are defined via iterations and there are two 
different ways to write down this sequence of compositions. 
Afterward one should apply  $\tilde\alpha_{..,.,.}$, using  $enter_{..,.,.}, (\frac{1_{..}\times\gamma}{\alpha_{..,.,.}}1)1$ and 
 $exit_{..,.,.}$, also expressed via iterations. In the end, 
the two different ways of writing down the composition $\frac{\tilde\alpha_{.,.,..}}{\tilde\alpha_{..,.,.}}$ appearing in the pentagon for 
$\tilde\alpha$ are shown to be equal (using naturalities, pentagon for $\alpha$ and $i$-ary hexagon). The analogous happens on the 
other side of the pentagon. Moreover, at the joints one $exit$ cancels out with the coming $enter$. (To understand why, observe that 
$enter$ is basically $\Epsilon$, {\em i.e.} $\gamma$, up to associativities. Similarly, $exit$ is basically $\gamma^{-1}$.) Consequently, 
one may choose the way of writing down the pentagon for $\tilde\alpha$ as follows 
\[\begin{tikzcd}
	{g\large( \big(\underset{i}{\textcircled{L}} f_i[\big(\underset{j}{\textcircled{L}} h_{ij}\big(        (\underset{q}{\textcircled{L}} k_{ij}^q )\omega_{e;q,l}\big)\big) \omega_{e,j}]\big)\omega_{e,i} \large)} & \\
	{\big(g(\underset{i}{\textcircled{L}} f_i\big(\underset{j}{\textcircled{L}} h_{ij}k_{ij}\big))\big)\Omega^k} & {\big((gf)h)k\big)\Omega^k } \\
	& {\big((gf)h\big) \big(\big[\underset{i}{\textcircled{L}} \big(\underset{j}{\textcircled{L}}         (\underset{q}{\textcircled{L}} k_{ij}^q )\omega_{e;q,l}\big) \omega_{e,j}\big]\omega_{e,i} \big)}
	\arrow["enter", Rightarrow, nfold, from=1-1, to=2-1]
	\arrow["{(pentagon for \frac{1\gamma}{\alpha})1_{\Omega^k}}", Rightarrow, nfold, from=2-1, to=2-2]
	\arrow["exit", Rightarrow, nfold, from=2-2, to=3-2]
\end{tikzcd}\]
with $\Omega^k$ from \rmref{Omega-k}, where the middle 2-cell is actually the indicated pentagonal diagram, so the whole desired pentagon  commutes.

\subsection{Unity transformations and the triangle law}

The left and right unity 2-cells 
$$id_B\circ f\Rightarrow f \quad\text{and}\quad f\circ\langle\crta{id_{A_i}} \rangle\Rightarrow f$$ 
for $f\in\ul{\B}(A_1,...,A_n;B)$ defined via $\hat f\in\B(\underset{i}{\textcircled{L}} A_i,B)$ we define as 
$\lambda_{\hat f}: id_B\circ \hat f\Rightarrow \hat f$ 
and 
$$\frac{[\tilde\zeta_0\vert \Id_{\hat f}]}{\rho_{\hat f}}:\hat f(\underset{i}{\textcircled{L}} id_{A_i})\Rightarrow\hat f$$  
whereby  $\tilde\zeta_0:\underset{i}{\textcircled{L}} id_{A_i}\Rightarrow id_{\underset{i}{\textcircled{L}} A_i}$ is the induced 2-cell from the canonical 2-cell $\zeta_0:id_A\ot id_B\Rightarrow id_{A\ot B}$ in $\B$. The triangle law for $\ul{\B}$ is proved to hold by the triangle law in $\B$ and naturality of $\alpha_{\hat f, id_{\underset{i}{\textcircled{L}} A_i}, \underset{i}{\textcircled{L}} g_i}$ with respect to the 2-cell $\tilde\zeta_0$, whereby $g_i\in\ul{\B}(\crta{D_{in}};A_i)$. 
This defines a bi-multicategory $\ul{\B}$.

\subsection{Bi-representability} \sslabel{birep-ity}

We finally show that $\ul{\B}$ is bi-representable in the sense of \deref{Phil}. Since any multimap in $\ul{\B}_n(A_1,...,A_n;B)$ is by construction a morphism 
$f$ in $\B(\textcircled{L}_1^n A_i,B)$, the equivalence \equref{Phil rep} could actually hold by identity, whereby $m(A_1,...,A_n)=\textcircled{L}_1^n A_i$,  
$j_{A_1,...,A_n}=id_{\textcircled{L}_1^n A_i}$ 
and the 2-cells $\Epsilon$ are identities on morphisms $f$. Though, for (other choices of) universal arrows $j':A_1,...,A_n=
\textcircled{L}_1^n A_i\to m'(A_1,...,A_n)$ we permit all arrows obtained from the objects $A_1,...,A_n$ via the components of the equivalences $a, l, r$ of $\B$ (we include $l,r$ in case some of the $A_i$'s is the unit object $I$). This way: the construction agrees 
with \rmref{j and j'} part d), 
other choices of universal objects $m'(A_1,...,A_n)$ include tensor product of the objects $A_1,...,A_n$ with a different distribution of parentheses,  
and we indeed have equivalences of categories $\ul{\B}_n(A_1,...,A_n;B)\simeq\ul{\B}_1(m'(A_1,...,A_n);B)$. Both types of universal arrows 
are closed under composition. 
The form of the equivalence \equref{strong rep} is defined via the isomorphisms $\omega_a$ from \equref{omega-a}. 

\smallskip

We summarize the construction of an induced bi-representable bi-multicategory out of a monoidal bicategory $\B$ in the following Table: 
\begin{table}[H]
\begin{center}
\begin{tabular}{ c c } 
To construct/show & \hspace{0,8cm} use / comment \\ [0.5ex]
\hline
lists (multi-homs) & monoidal product $\ot$ (homs on such products) \\ [1ex]   
substitution & monoidal associativity $a$, composition $comp$ of $\B$  \\ [1ex]
associativity $\tilde\alpha$ & associativity $\alpha$ and monoidal interchange $\xi$ of $\B$ \\ [1ex]
bi-representability &  equivalence, if so is the mon. assoc. $a$ \\
 & identity, if so is $a$ \\ [1ex]   
\end{tabular}
\caption{From a monoidal bicategory $\B$ to bi-rep. bi-mult.}
\label{table1}
\end{center}
\end{table}

\section{An equivalence of enriched categories} \selabel{eq}

In this section we prove that monoidal bicategories and bi-representable bi-multicategories form equivalent $\mathsf{Bicat}$-enriched categories. We will firstly prove the equivalence at the level of objects, then we will study the functor bicategory of (bi-representable) 
bi-multicategories (\ssref{fun-bicat}). In \ssref{mon bicat} we will show that monoidal bicategories indeed form a $\mathsf{Bicat}$-category. Next, in \ssref{equiv homs} we will prove that the two hom-bicategories are biequivalent, from what we will conclude in \ssref{equiva enrich} that bi-representable bi-multicategories also form a $\mathsf{Bicat}$-category, and that they both are equivalent.

\subsection{Equivalence at the level of objects} \sslabel{level objects}

It is easily seen that the assignments in the two previous sections are inverse to each other up to equivalence. Namely, starting from a monoidal bicategory $\B$ after both assignments one ends up in $\B$ itself, with the identical monoidal structure, so the monoidal pseudofunctor structure $\B\to\dul{\B}$ is given by identities. On the other hand, if one starts from a bi-representable bi-multicategory 
$\M$, after both assignments one ends up with a bi-representable bi-multicategory $\dul{\M}$ 
with the category of multimaps 
$\dul{\M}_n(a_1, ..., a_n;b)=\ul{\M}(\underset{i}{\textcircled{L}}a_i,b)=\M_1(m^L(a_1,...,a_n);b)\simeq\M_n(a_1,...,a_n;b)$, by bi-representability of $\M$, 
whereby in the second equality we define a morphism $\crta\zeta$ in $\ul{\M}(\underset{i}{\textcircled{L}}a_i,b)$ identically as $\dul{\zeta}\in
\M_1(m^L(a_1,...,a_n);b)$ corresponding by the equivalence to an initial $\zeta\in\M_n(a_1,...,a_n;b)$, with $m^L$ as in \rmref{j and j'}. 
To define a pseudofunctor on bi-multicategories $\M\to\dul{\M}$, we take it to be identity on objects and unary 1-cells, and on the $n$-ary multihoms we define 
it by the category equivalence given by bi-representability of $\M$
$$\Omega:\M_n(a_1,...,a_n;b)\to\M_1(m^L(a_1,...,a_n);b).$$ 
We also need to give transformations $\Omega^2, \Omega^0$ governing the composition of $n$-ary 1-cells in $\dul{\M}$ and identity 1-cells. 
For this reason let us show in more detail the above equivalence on multihom-categories. 
From the above we see that the two universal objects $m(a_1,...,a_n)$ in $\M$ and $m'(a_1,...,a_n)=m^L(a_1,...,a_n)$ in $\dul{\M}$ (we know $\dul{\M}$ is 
bi-representable) are equivalent (via a 1-cell $k$ and its quasi-inverse $k^*$), and that there is a universal 2-cell $\Epsilon_1: \dul{\zeta}k\Rightarrow\crta\zeta$, as in the picture below
$$\bfig
 \putmorphism(-420,480)(1,0)[\phantom{a_1,...,a_n} `m(a_1,...,a_n)`]{680}0a
 \putmorphism(-420,450)(1,0)[a_1,...,a_n `\phantom{m(a_1,...,a_n)}`j]{680}1a 
\putmorphism(-400,460)(2,-1)[\phantom{Y_2} ` `]{1300}1l
\putmorphism(-460,460)(2,-1)[\phantom{Y_2} ` `\zeta]{1300}0l
 \putmorphism(330,450)(1,-1)[ `b.`\crta\zeta]{540}1l

 \putmorphism(320,580)(1,0)[\phantom{a_1,...,a_n} ``id]{1500}1a
 \putmorphism(320,450)(1,0)[\phantom{a_1,...,a_n} `m^L(a_1,...,a_n)`k]{750}1b
 \putmorphism(1190,450)(1,0)[\phantom{a_1,...,a_n} `\phantom{m(a_1,...,a_n)}`k^*]{750}1b 
 \putmorphism(1190,480)(1,0)[\phantom{a_1,...,a_n} `m(a_1,...,a_n)`]{750}0b
 \putmorphism(900,450)(0,-1)[\phantom{a_1,...,a_n} ``\dul{\zeta}]{500}1r 
 \putmorphism(1860,460)(-3,-2)[\phantom{a_1,...,a_n} ``]{970}1r
 \putmorphism(1940,460)(-3,-2)[\phantom{a_1,...,a_n} ``\crta{\dul{\zeta}}]{970}0r
\put(1080,490){\fbox{$\omega$}}
\put(180,260){\fbox{$\Epsilon$}}
\put(650,260){\fbox{$\Epsilon_1$}}
\put(1120,230){\fbox{$\Epsilon_2$}}
\efig
$$	
In conclusion, $\zeta$ from $\M$ is mapped into $\dul{\zeta}:m^L(a_1,...,a_n)\to b$ in $\dul{\M}$ (accompanied with the universal 2-cell 
obtained by concatenation $\Epsilon\diamond \Epsilon_1$). Thus, both $\Omega(g)\langle \Omega(f_1),...,\Omega(f_n)\rangle$ and 
$\Omega(g\langle f_1,...,f_n\rangle)$ are given by unary 1-cells $\dul{\zeta_1}$ and $\dul{\zeta_2}$ as in the above diagram, and by universality there is a unique 2-cell between them, which we take to be $\Omega^2_{g, \crta f}$. Since $n$-ary 1-cells in $\dul{\M}$ are 
given by unary 1-cells on the universal objects in $\M$ (similar to $g, f_1,...,f_n, h_{ij}'s$ from \ssref{prep}), the 1-cells in the 
vertexes of the hexagonal law for $\Omega^2$ are 1-cells whose domains and codomains are such universal objects (up to the last codomain $c$ 
in $g:b_1,...,b_n\to c$), and the arrows in the hexagon are these unique 2-cells (for the 1-cells in question in every arrow). By the uniqueness of the 2-cells between 1-cells on universal objects, the two compositions of 2-cells in the hexagon are equal, which proves this law. 
Since $\Omega$ is identity on 1-cells, we take $\Omega^0$ to be identity. 

In the rest of the section we are going to prove that the categories of bi-representable bi-multicategories and of monoidal bicategories 
are enriched over the multicategory $\mathsf{Bicat}$ of bicategories, and that they are equivalent as such.

\subsection{The functor bicategory of (bi-representable) bi-multicategories} \sslabel{fun-bicat}

Pseudofunctors, transformations and modifications of bi-multicategories were defined in \cite[Definition 4.2.33]{Phil}. It is straightforwardly shown that pseudofunctors (and even lax functors) and (oplax/pseudo) transformations are closed under composition. (We will be interested in pseudonatural transformations.) If a transformation $\theta:\F\Rightarrow \G$ is given by a pair $(\theta_x, \theta_f)$ 
for a 1-cell $f:x\to y$, and a transformation $\theta':\G\Rightarrow \HH$ is given by a pair $(\theta'_x, \theta'_f)$, their composition is given by a pair 
$(\theta'_x\theta_x, \frac{\theta}{\theta'}\vert_f)$ whereby $\frac{\theta}{\theta'}\vert_f$ is a 2-cell given as a vertical composition of 2-cells $\alpha, 
[\theta_{id_x}\vert 1], \alpha^{-1}, [1\vert\theta'_{\id_X}], \alpha$ in $\M$. It is also straightforward to show that both horizontal and vertical composition of modifications are modifications. They are given by horizontal, respectively vertical composition of their respective component 2-cells. The composition of transformations is associative up to an isomorphism, as so is the composition of their 1-cell components living in the bi-multicategory in question. This proves that 
pseudofunctors $\M\to\N$, pseudonatural transformations and modifications of bi-multicategories $\M, \N$ form a bicategory $\BiMlt(\M,\N)$. If the codomain $\N$ 
is a 2-multicategory, $\BiMlt(\M,\N)$ is henceforth a 2-category. 

\smallskip

Similarly to our observations from \ssref{more data bi-rep}, we draw more information from the definition of a pseudofunctor $\F:\M\to\N$ 
of bi-multicategories \cite[Definition 4.2.33]{Phil}. Namely, one has that the functors 
$\F_{\crta a;b}: \M_n(\crta a; b)\to\N_n(\crta{\F(a_i)};\F(b))$ are 1-cell components in $\Cat_2$ of a pseudonatural transformation 
$\F_{-;1}: \M_{n}(-;b)\Rightarrow\N_n(\crta{\F(-)};\F(b)): \M_n \to\Cat_2$. 
Its 2-cell components $\F_{\crta f;1}$ for a 1-cell $\crta f\in\M_n$ are natural transformations and have the form 
\begin{equation} \eqlabel{nat F}
	\scalebox{0.86}{
\bfig
 \putmorphism(-150,470)(1,0)[\M_{k+n+l}(\crta x',a_1',..,a_n',\crta b';c)`\N_{k+n+l}(\crta{\F(x')},\F(a_1'),..,\F(a_n'),\crta{\F(b')};\F(c))`
   \F_{\crta x',\crta{a'},\crta b';c}]{1900}1a
\putmorphism(-160,470)(0,-1)[\phantom{Y_2}` ` \M_{k+n+l}(\crta p, \crta f, \crta q;r)]{380}1l
\putmorphism(1850,470)(0,-1)[\phantom{Y_2}` `  \N_{k+n+l}(\crta p,\crta{\F(f_i)}, \crta q;r)]{380}1r
 \putmorphism(-150,70)(1,0)[\M_{k+n+l}(\crta x,a_1,..,a_n,\crta b;c')`\N_{k+n+l}(\crta{\F(x)},\F(a_1),..,\F(a_n),\crta{\F(b)};\F(c')). `
   \F_{\crta x,\crta{a},\crta b;c'}]{1900}1b
\put(520,270){\fbox{$\SWarrow \F_{\crta p, \crta f, \crta q;r}$}}
\efig}
\end{equation}

\subsubsection{Properties of pseudofunctors on bi-representable bi-multicategories} \sslabel{fun-birep-bimult} 

We are interested in the full sub-bicategory of $\BiMlt(\M,\N)$ of pseudofunctors of {\em bi-representable} bi-multicategories. 
It will turn out that for $\M$ and $\N$ bi-representable this sub-bicategory coincides with the whole bicategory $\BiMlt(\M,\N)$, as we will show 
in this and the next two subsections. 
In this subsection we prove a series of properties of the pseudofunctors in $\BiMlt(\M,\N)$ acting on bi-representable bi-multicategories $\M$ and $\N$. 
Most importantly, we will show that such pseudofunctors induce a certain family of natural transformations and modifications. 

In the next proposition we prove the existence of those transformations. 
Recall from \ssref{more data bi-rep} that universal objects of bi-representable bi-multicategories determine 
pseudofunctors $m_n:\M_n\to\M_1$. We will omit writing the index $n$.

\begin{prop} \prlabel{ksi}
Let $\M$ and $\N$ be  bi-representable bi-multicategories and $\F:\M\to\N$ a pseudofunctor of their underlying bi-multicategories. 
Then $\F$ induces for every $n\in\Nn_{\geq 1}$ a pseudonatural transformation 
$$\xi_n: m(\F(-),..,\F(-))\to \F(m(-,..,-)): \M_n\to\N_1$$
and a unary 1-cell $\xi_0: m(-)\to\F(m(-))$ in $\N$. 
 \end{prop}

\begin{proof}
The 1-cell components $\xi_n(\crta a)$ of $\xi_n$ are induced as follows
\begin{equation} \eqlabel{j-xi}
\bfig
 \putmorphism(330,450)(1,-1)[ `\F(m(a_1,...,a_n)))`\F_{\crta a}(j_{\crta a})]{540}1l
\putmorphism(320,450)(1,0)[\F(a_1),...,\F(a_n) `m(\F(a_1),...,\F(a_n))`j^\N_{\crta{\F(a_i)}}]{1200}1a
 \putmorphism(1500,460)(-1,-1)[\phantom{a_1,...,a_n} ``\xi_n(\crta a)]{550}1r 
\put(830,190){\fbox{$\sigma_{\crta a} $}}
\efig
\end{equation}	
that is, as the image by the equivalence 
$$\N_n(\F(a_1),...,\F(a_n);\F(m(a_1,...,a_n)))\stackrel{\Omega^\N_{\crta a;\F(m(\crta a))}}{\to}
  \N_1(m(\F(a_1),...,\F(a_n));\F(m(a_1,...,a_n)))$$ 
of $\F_{\crta a}(j_{\crta a})$, the image by $\F$ of the $n$-ary universal map $j_{\crta a}$ in $\M$. 
(The unary 1-cell $\xi_0: m(-)\to\F(m(-))$ corresponds to $\xi_n(\crta a)$ for $n=0$ and the empty list in place of $\crta a$.) 
The 2-cell components $\xi_n(\crta{f_i})$ of $\xi_n$ are 
the unique invertible 2-cells $\xi_n(\crta b)m(\F(\crta{f_i}))\Rightarrow\F(\Epsilon_{m(\crta{f_i})})\xi_n(\crta{a})$ at the place labeled 
by $\NWarrow$ on the left below obtained from the universal property (recall \rmref{unique inv 2-cell}) so that the following equality of 2-cells holds 
\begin{equation} \eqlabel{transf xi}
	\scalebox{0.86}{
\bfig
 \putmorphism(330,450)(1,-1)[ `\F(m(a_1,...,a_n)))`\F_{\crta a}(j_{\crta a})]{540}1l
\putmorphism(320,450)(1,0)[\F(a_1),...,\F(a_n) `m(\F(a_1),...,\F(a_n))`j^\N_{\crta{\F(a_i)}}]{1200}1a
 \putmorphism(1500,460)(-1,-1)[\phantom{a_1,...,a_n} ``\xi_n(\crta a)]{550}1r 
\put(830,240){\fbox{$\sigma_{\crta a} $}}

\putmorphism(260,450)(0,-1)[ `\F(b_1),...,\F(b_n)`\crta{\F(f_i)}]{900}1l
\putmorphism(890,-90)(0,-1)[ `\F(m(b_1,...,b_n)))`\F(m(\crta{f_i}))]{900}1r
\put(380,-270){\fbox{$\F(\Epsilon_{m(\crta{f_i})})$}}
\putmorphism(1580,450)(0,-1)[ ``m(\crta{\F(f_i)})]{900}1r
\putmorphism(1680,450)(0,-1)[ `m(\crta{\F(b_i)})`]{900}0r
\put(1110,-300){\fbox{$\NWarrow$}}
 \putmorphism(250,-460)(1,-1)[ ``\F(j_{\crta b})]{540}1l
 \putmorphism(1600,-450)(-1,-1)[\phantom{a_1,...,a_n} ``\xi_n(\crta b)]{550}1r 

\efig}\quad
=\quad
	\scalebox{0.86}{
\bfig
\putmorphism(320,450)(1,0)[\F(a_1),...,\F(a_n) `m(\F(a_1),...,\F(a_n))`j^\N_{\crta{\F(a_i)}}]{1200}1a
\putmorphism(320,450)(0,-1)[ `\F(b_1),...,\F(b_n)`\crta{\F(f_i)}]{400}1l
\putmorphism(1520,450)(0,-1)[ `m(\F(b_1),...,\F(b_n))`m(\crta{\F(f_i)})]{400}1r
\putmorphism(320,0)(1,0)[\phantom{\F(a_1),...,\F(a_n)} `\phantom{m(\F(a_1),...,\F(a_n))}`j^\N_{\crta{\F(b_i)}}]{1200}1a
\put(830,-230){\fbox{$\sigma_{\crta b} $}}
\put(720,300){\fbox{$\Epsilon_{m(\crta{\F(f_i)})} $}}
 \putmorphism(330,0)(1,-1)[ `\F(m(a_1,...,a_n)))`\F(j_{\crta b})]{540}1l
 \putmorphism(1500,10)(-1,-1)[\phantom{A_1,...,A_n} ``\xi_n(\crta b)]{550}1r 
\efig}
\end{equation}
where $\F(\Epsilon_{m(\crta{f_i})})$ on the left is the image by $\F$ of the invertible 2-cell 
$$\scalebox{0.86}{
\bfig
 \putmorphism(-150,470)(1,0)[a_1,..,a_n ` m(a_1,..,a_n)`j_{\crta a}]{900}1a
\putmorphism(-160,470)(0,-1)[\phantom{Y_2}` `\crta{f_i} ]{380}1l
\putmorphism(740,470)(0,-1)[\phantom{Y_2}` `  m(\crta{f_i})]{380}1r
 \putmorphism(-150,70)(1,0)[b_1,..,b_n`m(b_1,..,b_n)` j_{\crta b}]{900}1b
\put(100,270){\fbox{$\SWarrow \Epsilon_{m(\crta{f_i})}$}}
\efig}
$$
defining $m(\crta{f_i})$. 
The transformation property is 
then directly and easily proved ({\em e.g.} for the compositionality, to the left diagram in \equref{transf xi} paste from below the 
diagram for $\F(\Epsilon_{m(\crta{g_i})})$ on the left and the one for $\xi_n(\crta{g_i})$ on the right, and apply \equref{transf xi} 
on the right-hand side). Another way of proving it is applying the principle of the unique total 2-cell by the universality, between 1-cells acting on a universal object, which expresses  the equality holding in the transformation laws. 
\qed\end{proof}


We now pass to proving the existence of the above-mentioned modifications induced by $\F$.  
The following is an obvious observation. Recall the 2-cells $\tau$ from \equref{tau-a} 
and let $\F^2$ and $\F^0$ denote the structure transformations {\em i.e.} their 2-cell components of the pseudofunctor structure of $\F$. 

\begin{prop} \prlabel{delta}
In view of \rmref{j and j'}, consider additional choices of universal objects $m'(\bullet)$, arrows $j'_n$ and let 
$\xi'_n$ and $\sigma'$ be the corresponding 1- and 2-cells inherent to $\F$, both in $\M$ and in $\N$. Given the composite 2-cells as bellow 
$$ 
\scalebox{0.88}{\bfig
 \putmorphism(330,450)(1,-1)[ `m'(\crta{\F(a_i)})`j'_n]{540}1r
\putmorphism(140,450)(1,0)[\F(a_1),...,\F(a_n) `m(\F(a_1),...,\F(a_n))`j^\N_n]{1280}1a
 \putmorphism(1500,460)(-1,-1)[\phantom{A_1,...,A_n} ``a_{\crta{\F(a_i)}}]{550}1r 
\put(880,270){\fbox{$\tau $}}
\putmorphism(890,-90)(0,-1)[ `\F(m'(\crta{a_i}))`\xi'_n]{350}1r
\putmorphism(-140,450)(1,-1)[ ``\F(j'_n)]{880}1l
\put(400,0){\fbox{$\sigma'$}}
\efig}
\qquad
\scalebox{0.88}{\bfig
 \putmorphism(330,450)(1,-1)[ `\F(m(\crta{a_i})))`\F(j_n)]{540}1r
\putmorphism(140,450)(1,0)[\F(a_1),...,\F(a_n) `m(\F(a_1),...,\F(a_n))`j^\N_n]{1280}1a
 \putmorphism(1500,460)(-1,-1)[\phantom{A_1,...,A_n} ``\xi_n]{550}1r 
\put(980,270){\fbox{$\sigma$}}
\putmorphism(890,-90)(0,-1)[ `\F(m'(\crta{a_i}))`\F(a_{\crta{a_i}})]{350}1r
\putmorphism(-140,450)(1,-1)[ ``\F(j'_n)]{880}1l
\put(400,0){\fbox{$\frac{\F^2}{\F(\tau)}$}}
\efig}
$$
the universal property of bi-representability renders a unique invertible 2-cell between their right-hand side edges
\begin{equation}\eqlabel{delta_n za niz a_i}
\F(a_{\crta{a_i}}) \xi_n\Rightarrow \xi'_n a_{\crta{\F(a_i)}}
\end{equation}
so that its composition with the left composite-triangle above equals the right composite-triangle. Given another path 
$a_{\crta{a_i}}':m(a_1,...,a_n)\to m'(a_1,...,a_n)$ and composing with $\F(\pi)$ from \rmref{j and j'}, we similarly obtain a unique invertible 2-cell 
\begin{equation}\eqlabel{delta_n za niz a_i'}
\F(a_{\crta{a_i}}') \xi_n\Rightarrow \xi'_n a_{\crta{\F(a_i)}}
\end{equation}
with a similar property. (Observe that \equref{delta_n za niz a_i} is relevant for $n=3$, while \equref{delta_n za niz a_i'} makes sense for $n>3$.) 
%
\end{prop}


\medskip

We actually have:

\begin{prop} \prlabel{get modif delta for free}
For $n=3$ the 2-cell \equref{delta_n za niz a_i} is a 2-cell component of a modification 
$$\delta_3:\F(a_{-,-,-}) \xi_3(-,-,-)\Rrightarrow \xi'_3(-,-,-) a_{\F(-),\F(-),\F(-)}$$
and for $n>3$ the 2-cell \equref{delta_n za niz a_i'} is a 2-cell component of an invertible modification 
$$\delta_n:\F(a_{\crta{(-)}}') \xi_n(\crta{(-)})\Rrightarrow \xi'_n(\crta{(-)}) a_{\crta{\F(-)}}$$
between pseudonatural transformations of pseudofunctors $m(\F(-),..,\F(-))\to \F(m'(-,..,-))$. 
\end{prop}



\begin{proof}
To check that $\delta_n$ is a modification, we evaluate it at a list of objects $\crta{a_i}=a_1,..,a_n$ and the pseudonatural transformations on which $\delta_n$ acts we evaluate at a list of morphisms $\crta{f_i}:\crta{a_i}\to\crta{b_i}$, the desired axiom is then the identity below in the form of prism 
\[\begin{tikzcdscale}[every label/.append style = {font = \large}]
	&& {m(\overline{\F(a_i)})} &&&& {m(\overline{\F(a_i)})} \\
	& {m'(\overline{\F(a_i)})} &&&& {\F(m(\overline{a_i}))} \\
	{\F(m'(\overline{a_i}))} &&&& {\F(m'(\overline{a_i}))} \\
	\\
	&& {m(\overline{\F(b_i)})} &&&& {m(\overline{\F(b_i)})} \\
	& {m'(\overline{\F(b_i)})} &&&& {\F(m(\overline{b_i}))} \\
	{\F(m'(\overline{b_i}))} &&&& {\F(m'(\overline{b_i}))}
	\arrow[equals, nfold, from=1-3, to=1-7]
	\arrow[""{name=0, anchor=center, inner sep=0}, "{a_{\overline{\F(a_i)}}}"', from=1-3, to=2-2]
	\arrow["{m(\overline{\F(f_i)})}"'{pos=0.3}, from=1-3, to=5-3]
	\arrow[""{name=1, anchor=center, inner sep=0}, "{\xi_n(\overline{a_i})}"'{pos=0.7}, from=1-7, to=2-6]
	\arrow["{m(\overline{\F(f_i)})}", from=1-7, to=5-7]
	\arrow[""{name=2, anchor=center, inner sep=0}, "{\xi_n'(\overline{a_i})}"', from=2-2, to=3-1]
	\arrow["{m'(\overline{\F(f_i)})}"'{pos=0.4}, from=2-2, to=6-2]
	\arrow["{\delta_n(\overline{a_i})(j_n')}"', between={0.4}{0.6}, Rightarrow, nfold, from=2-6, to=2-2]
	\arrow[""{name=3, anchor=center, inner sep=0}, "{\F(a_{\overline{a_i}}')}"'{pos=0.6}, from=2-6, to=3-5]
	\arrow["{\F(m(\overline{f_i}))}"'{pos=0.4}, from=2-6, to=6-6]
	\arrow[equals, nfold, from=3-1, to=3-5]
	\arrow["{\F(m'(\overline{f_i}))}"', from=3-1, to=7-1]
	\arrow["{\F(m'(\overline{f_i}))}"'{pos=0.3}, from=3-5, to=7-5]
	\arrow[equals, nfold, from=5-3, to=5-7]
	\arrow[""{name=4, anchor=center, inner sep=0}, "{a_{\overline{\F(b_i)}}}"', from=5-3, to=6-2]
	\arrow[""{name=5, anchor=center, inner sep=0}, "{\xi_n(\overline{b_i})}", from=5-7, to=6-6]
	\arrow[""{name=6, anchor=center, inner sep=0}, "{\xi_n'(\overline{b_i})}"', from=6-2, to=7-1]
	\arrow["{\delta_n(\overline{b_i})(j_n')}"', between={0.4}{0.6}, Rightarrow, nfold, from=6-6, to=6-2]
	\arrow[""{name=7, anchor=center, inner sep=0}, "{\F(a_{\overline{b_i}}')}", from=6-6, to=7-5]
	\arrow[equals, nfold, from=7-5, to=7-1]
	\arrow["{a_{\overline{\F(f_i)}}}"{description}, between={0.3}{0.7}, Rightarrow, nfold, from=0, to=4]
	\arrow["{\xi_n(\overline{f_i})}"{description}, between={0.3}{0.7}, Rightarrow, nfold, from=1, to=5]
	\arrow["{\xi'_n(\overline{f_i})}"{description}, between={0.4}{0.7}, Rightarrow, nfold, from=2, to=6]
	\arrow["{\F(a_{\overline{f_i}}')}"{description}, between={0.4}{0.7}, Rightarrow, nfold, from=3, to=7]
\end{tikzcdscale}\]
(its front and back face are identities). This identity indeed holds true, as the two compositions of 2-cells that should be equal (the top 
composed with the right-hand side lateral, and the left-hand side lateral composed with the bottom) are mapping 
\begin{equation} \eqlabel{cylinder to prove}
\F(a_{\crta{b_i}}') \circ\xi_n(\crta{b_i})\circ m(\overline{\F(f_i)})\Rightarrow \F(m'(\overline{f_i}))\circ \xi'_n\circ a_{\crta{\F(a_i)}}: 
\quad m(\overline{\F(a_i)})\to \F(m'(\overline{b_i}))
\end{equation}
and as such they must be equal by the universal property. Namely, in the following diagram
\[\begin{tikzcdscale}
	&&&&&& {m(\overline{\F(a_i)})} \\
	{\overline{\F(a_i)}} \\
	&& {\varepsilon_{back}} && {m(\overline{\F(a_i)})} \\
	& {\F(\varepsilon_{front})} \\
	&&&&&& {m(\overline{\F(b_i)})} \\
	{\overline{\F(b_i)}} \\
	&& {} && {m(\overline{\F(b_i)})}
	\arrow[""{name=0, anchor=center, inner sep=0}, "{\xi'_n\circ a_{\crta{\F(a_i)}}}"{description, pos=0.6}, curve={height=30pt}, from=1-7, to=3-5]
	\arrow[""{name=1, anchor=center, inner sep=0}, "{\F(a_{\crta{a_i}}') \circ\xi_n(\crta{a_i})}"{description}, curve={height=-30pt}, from=1-7, to=3-5]
	\arrow["{m(\overline{\F(f_i)})}", from=1-7, to=5-7]
	\arrow["{j^{\mathcal{N}}_{\overline{a_i}}}", curve={height=-30pt}, from=2-1, to=1-7]
	\arrow[""{name=2, anchor=center, inner sep=0}, "{\F(j'_{\overline{a_i}})}", from=2-1, to=3-5]
	\arrow["{\overline{\F(f_i)}}"', from=2-1, to=6-1]
	\arrow["{\F(m'(\overline{f_i}))}"{pos=0.2}, from=3-5, to=7-5]
	\arrow[""{name=3, anchor=center, inner sep=0}, "{\xi'_n\circ a_{\crta{\F(b_i)}}}"{description}, curve={height=30pt}, from=5-7, to=7-5]
	\arrow[""{name=4, anchor=center, inner sep=0}, "{\F(a_{\crta{b_i}}') \circ\xi_n(\crta{b_i})}"{description}, curve={height=-30pt}, from=5-7, to=7-5]
	\arrow["{j^{\mathcal{N}}_{\overline{b_i}}}", curve={height=-30pt}, from=6-1, to=5-7]
	\arrow[""{name=5, anchor=center, inner sep=0}, "{\F(j'_{\overline{b_i}})}"', from=6-1, to=7-5]
	\arrow["{\delta_n(\overline{a_i})(j_n')}"{description}, between={0.3}{0.7}, Rightarrow, nfold, from=1, to=0]
	\arrow["{(\sigma'\diamond\tau)_{\overline{a_i}}}"', between={0.4}{0.6}, Rightarrow, nfold, from=0, to=2]
	\arrow["{\delta_n(\overline{b_i})(j_n')}"{description}, between={0.3}{0.7}, Rightarrow, nfold, from=4, to=3]
	\arrow["{(\sigma'\diamond\tau)_{\overline{b_i}}}"', between={0.3}{0.6}, Rightarrow, nfold, from=3, to=5]
\end{tikzcdscale}\]
observe that the back rectangular presents a 2-cell $\varepsilon_{back}$ determining the (vertically posed) 1-cell $m(\overline{\F(f_i)})$ 
on the right, similarly that the left front square presents a 2-cell $\F(\varepsilon_{front})$ determining the 1-cell $\F(m'(\overline{f_i}))$ (here $\varepsilon_{front}$ determines $m'(\overline{f_i})$), and the cylinder on the right presents the above prism. The concatenation 
of 2-cells $sigma'\diamond\tau$ is the one from \prref{delta}. Then for the given 
1-cell $\F(j'_{\overline{b_i}})\circ \overline{\F(f_i)}$ (bottom-left) and the universal arrow $j^{\mathcal{N}}_{\overline{a_i}}$ above, 
we have two pairs of a 2-cell and a 1-cell: $(\sigma'\diamond\tau)_{\overline{a_i}} \diamond 
\F(\Epsilon_{front})$ determining $\F(m'(\overline{f_i}))\circ \big(\xi'_n\circ a_{\crta{\F(a_i)}}\big)$, and $\Epsilon_{back}\diamond 
(\sigma'\diamond\tau)_{\overline{b_i}} \diamond \delta_n(\overline{b_i})$ determining $\big(\F(a_{\crta{b_i}}') \circ\xi_n(\crta{b_i})\big)\circ m(\overline{\F(f_i)})$, 
then the universal property states that there is a unique  invertible 2-cell between the latter two 1-cells, {\em i.e.} \equref{cylinder to prove}, meaning that the two 2-cells determining the above prism/cylinder must be equal. 
\qed\end{proof}

\bigskip

Moreover, in an analogous way one proves: 

\begin{prop} \prlabel{}
There is an invertible modification 
$$\delta_n(-;*)(\bullet):\N_1(\xi_n(-); *)\circ\F_{m(-);*}\circ\Omega^\M_{-;*}(\bullet)\Rrightarrow\Omega^\N_{-;*}\circ\F_{-;*}(\bullet)$$ 
with 2-cell components natural transformations 
\begin{equation} \eqlabel{modif delta}
\scalebox{0.86}{
\bfig
 \putmorphism(-150,470)(1,0)[\M_1(m(a_1,..,a_n);c)` \M_{n}(a_1,..,a_n;c) `
   \Omega^\M_{\crta{a_i};c}]{2350}{-1}a
\putmorphism(-160,450)(0,-1)[\phantom{Y_2}`\N_{1}(\F(m(a_1,..,a_n));\F(c)) ` \F_{m(\crta{a_i});c}]{400}1l
\putmorphism(-160,50)(0,-1)[\phantom{Y_2}` ` \N_1(\xi_n(\crta a);1)]{400}1l
\putmorphism(2250,470)(0,-1)[\phantom{Y_2}` `  \F_{\crta{a_i};c}]{800}1r
 \putmorphism(0,-350)(1,0)[\N_1(m(\F(a_1),..,\F(a_n));\F(c))` \N_n(\F(a_1),..,\F(a_n);\F(c)) `\Omega^\N_{\crta{a_i};c}]{2250}{-1}a
\put(900,40){\fbox{$\SEarrow \delta_{\crta{a_i};,c}$}}
\efig}
\end{equation}
for any list of objects $\crta{a_i}=a_1,..,a_n$, and object $c$.
\end{prop}

By adding lists $\crta{x_i}$ and $\crta{b_i}$ in the above proposition, one deduces that there is an invertible modification
$\delta_{k+n+l;1}:\N_1(-, \xi_n(-), -; -)\F_{-,m(-),-;-}\Omega^\M_{-,-,-;-}\Rrightarrow\Omega^\N_{-,-,-;-}\F_{-,-,-;-}$ 
whose 2-cell components evaluated at $\zeta:\crta{x_i}, \crta{a_i},\crta{b_i}\to c$ 
take the form  
\begin{equation}\eqlabel{delta_n forma modif- sa x i b}
\delta_{k+n+l;1}(\crta{x_i},\crta{a_i},\crta{b_i};c)(\zeta): 
\F(\Omega^\M(\zeta))\circ\langle 1_k,\xi_n,1_l\rangle\Rightarrow\Omega^\N(\F(\zeta)).
\end{equation}
(The idea of the proof is that in the analogous prism to the first one in the proof of \prref{get modif delta for free} one adds lists 
$\crta{\F(x_i)}$ and $\crta{\F(b_i)}$; the vertically posed 2-cells 
change into the ternary 2-cells of the form 
$\langle 1_k,\bullet,1_l\rangle$ and the prism gets slightly more complex by obtaining a few more non-trivial faces.)



\bigskip

Similarly as in \prref{delta}, we construct invertible 2-cells 
\begin{equation} \eqlabel{l i r}
\F(l_a)\circ\xi_2\circ m(\xi_0,id_a)\Rightarrow l_{\F(a)}\quad\text{and}\quad \F(r_a)\circ\xi_2\circ m(id_a,\xi_0)\Rightarrow r_{\F(a)}
\end{equation}
and similarly as in \prref{get modif delta for free}, we have:

\begin{prop} \prlabel{modifs delta L i R}
The 2-cells \equref{l i r} are 2-cell components of invertible modifications
$$\F(l_{-})\circ\xi_2\circ m(\xi_0,id_{-})\Rightarrow l_{\F(-)}\quad\text{and}\quad \F(r_{-})\circ\xi_2\circ m(id_{-},\xi_0)\Rightarrow r_{\F(-)}.$$
\end{prop}

\medskip

In view of the diagram \equref{modif delta} and given the direction of the pseudonatural transformations 
 $\xi_n: m(\F(-),..,\F(-))\to \F(m(-,..,-))$ we may conclude: 

\begin{cor} \colabel{respect univ o}
Let $\M$ and $\N$ be  bi-representable bi-multicategories. Any 
pseudofunctor between their underlying bi-multicategories preserves the universal objects in a lax monoidal way up to equivalence. 
\end{cor}

\subsubsection{Iterative presentation for $\xi_n$ and the summary for the pseudofunctors} \sslabel{fun-rep-bimult summary}

We next want to present each of the four 1-cells in \equref{delta_n za niz a_i} for $n=3$ and in \equref{delta_n za niz a_i'} for $n>3$ 
as a path of 1-cells. This will help us finalize the description of 
pseudofunctors of bi-representable bi-multicategories. We immediately may observe:

\begin{cor} \colabel{path a_i}
Given two paths $a_{\crta{a_i}}, a_{\crta{a_i}}':m(a_1,...,a_n)\to m'(a_1,...,a_n)$ as in \rmref{j and j'} there is a unique (composite) invertible 2-cell
$$
path\big(\F(a'_{\crta{a_i}})\big) \xi_n\stackrel{\crta{\F^2}\circ 1}{\Rightarrow} \F(a'_{\crta{a_i}})\xi_n \stackrel{\equref{delta_n za niz a_i'}}{\Rightarrow} 
  \xi'_n a_{\crta{\F(a_i)}}
$$
where $\crta{\F^2}: path\big(\F(a'_{\crta{a_i}})\big)\Rightarrow \F(a'_{\crta{a_i}})$ is the composite of any choice of the relevant binary coherence 2-cells 
of the pseudofunctor $\F$ (unique by the hexagonal rule for $\F$). 
\end{cor}

It remains to replace the arrows $\xi_n$ with paths of arrows depending on $\xi_2$. We achieve this in the next proposition.

\begin{prop} \prlabel{lax mon mult}  \plabel{lax mon mult}
\begin{enumerate}
\item 
Let $j_n^L:a_1,...,a_n\to m^L(a_1,...,a_n)$ and $j_n^R:a_1,...,a_n\to m^R(a_1,...,a_n)$ be as in \rmref{j and j'} and let 
$(\xi^L_n, \sigma^L)$ and $(\xi^R_n, \sigma^R)$  be the corresponding pairs of 1- and 2-cells inherent to $\F$ with respect to $j_n^L$ and $j_n^R$. 
By the universal property there is a unique invertible 2-cell marked on the right below so that when composed with the rest of the labeled 2-cells the 
whole composition equals $\sigma^L$. 
\begin{equation}\eqlabel{sigma-1}
\bfig
 \putmorphism(0,700)(1,0)[\F(x),\F(y),\F(z)` \F(x)\F(y),\F(z) `\langle j, 1\rangle ]{1000}1a
 \putmorphism(1000,700)(1,0)[\phantom{\F(x)\F(y),\F(z)}`(\F(x)\F(y)) \F(z) `j]{900}1a
 \putmorphism(190,700)(1,-1)[\phantom{A\ot B}` `\langle \F(j),1\rangle ]{350}1l 
 \putmorphism(400,700)(1,-2)[\phantom{A\ot B}`\F(xy),\F(z) ` ]{190}0l 
 \putmorphism(710,320)(1,-1)[\phantom{A\ot B}`` \F(j)]{350}1l 
 \putmorphism(920,340)(1,-2)[\phantom{A\ot B}`\F((xy) z) ` ]{190}0l
 \putmorphism(1140,-40)(0,-1)[\phantom{A\ot B}`\F((xy) z) ` =]{370}1l
 \putmorphism(970,700)(-1,-2)[` `]{180}1r 
 \putmorphism(940,700)(-1,-2)[` `\langle\xi,1\rangle]{180}0r 
 \putmorphism(1460,320)(-1,-2)[` `\xi]{180}1r

 \putmorphism(600,320)(1,0)[\phantom{\F(xy),\F(z)}`\F(xy)\F(z) ` j]{900}1a

 \putmorphism(1730,700)(-1,-2)[` `\xi 1]{180}1r

\putmorphism(-280,700)(1,-1)[ ``\F(j^L_3)]{1180}1l
\putmorphism(2030,700)(-1,-2)[ ``\xi_3^L]{580}1r
\put(1360,-50){\fbox{$\NWarrow$}}

\put(500,500){\fbox{$\langle\sigma,1\rangle$}}
\put(1300,520){\fbox{$\Epsilon$}}
\put(1050,140){\fbox{$\sigma$}}
\put(530,-50){\fbox{$\frac{[(1,\F^0)\vert 1]}{\F(\F^2)}$}}
\efig
\end{equation}
(Here concatenation of objects and arrows $xy, \xi 1$ is short for $m_2(x,y)$ and $m_2(\xi,1)$, respectively.)
Analogously, there is a unique invertible 2-cell $\xi^R_3\Rightarrow \xi\comp (1\xi)$. 
\item 
For any $n\in\Nn_{\geq 3}$ there are unique invertible 2-cells (coming from the universal property) with the domain $\xi_{n+1}$ and codomains 
\begin{multline*}
m_{n+1}(\crta{\F(a_i)})\xrightarrow{\displaystyle a^{\N,i}_\bullet} m_n(\F(a_1),...,m_2(\F(a_i),\F(a_{i+1})),...,\F(a_{n+1})) \\
  \xrightarrow{\displaystyle m_n(1,\xi_2^i,1)} 
	m_n(\F(a_1),...,\F(m_2(a_i,a_{i+1}),...,\F(a_{n+1}))   \\
	\xrightarrow{\displaystyle \xi_n}	 \F(m_n(a_1,...,m_2(a_i,a_{i+1}),...,a_{n+1}))  
	\xrightarrow{\displaystyle \F((a^{\M,i}_\bullet)^*)} \F(m_{n+1}(\crta{a_i}))
	\end{multline*}
	and 
\begin{multline*}
m_{n+1}(\crta{\F(a_i)})\xrightarrow{\displaystyle a'^\N_\bullet} m_2(m_n(\crta{\F(a_i)}),\F(a_{n+1})) 
  \xrightarrow{\displaystyle m_2(\xi_n,1)} m_2(\F(m_n(a_i)),\F(a_{n+1}))   \\
	\xrightarrow{\displaystyle \xi_2}	 \F(m_2(m_n(\crta{a_i}),a_{n+1})) \xrightarrow{\displaystyle \F((a'^\M_\bullet)^*)} \F(m_{n+1}(\crta{a_i}))
		\end{multline*}

and
\begin{multline*}
m_{n+1}(\crta{\F(a_i)})\xrightarrow{\displaystyle a''^\N_\bullet} m_2(\F(a_1),m_n(\F(a_2),...,\F(a_{n+1}))) \\
  \xrightarrow{\displaystyle m_2(1,\xi_n)} m_2(\F(a_1),\F(m_n(a_2,...,a_{n+1})))   \\
	\xrightarrow{\displaystyle \xi_2}	 \F(m_2( a_1,m_n(a_2,...,a_{n+1})) \xrightarrow{\displaystyle \F((a''^\M_\bullet)^*)} \F(m_{n+1}(\crta{a_i}))
		\end{multline*}
		relating  $\xi_{n+1}$ with $\xi_n$ and $\xi_2$, where $a^{\M,i}_\bullet,a'^\M_\bullet,a''^\M_\bullet$ are the obvious versions of the arrow $a_{\crta{a_i}}$ in \rmref{j and j'} and $a^*$ denotes a quasi-inverse, and $\xi^i_2$ denotes that $\xi_2$ is applied at the position $i$ of the $n$-tuple in question. 
\item For any $n\in\Nn_{\geq 3}$ there is a unique invertible 2-cell 
$$\xi_{n+1}\Rightarrow \F(a^\M_\bullet)\circ \xi_2\circ m_2(1,\xi_2^{i_1})\circ m_3(1,\xi_2^{i_2},1)\circ m_4(1,\xi_2^{i_3},1)\circ...\circ m_n(1,\xi_2^{i_{n-1}},1)
  \circ (a^\N_\bullet)^*$$ 
	where 
	$\xi_2^{i_{l-1}}$ denotes that $\xi_2$ is applied at the position $i_{l-1}$ in the $l$-tuple in question. 
\item For $\xi_n^L$ and $\xi_n^R$ as in part 1) there are unique invertible 2-cells $\xi_n^L\Rightarrow \xi\comp (\xi 1)\circ ...\circ(\xi 1...1)$ and 
$\xi_n^R\Rightarrow \xi\comp (1\xi)\circ ...\circ(1...1\xi)$. 
\end{enumerate}
\end{prop}

\begin{proof}
The first part is clear. The second part is proved by manipulation with the following 
2-cells: $\sigma$ from \equref{j-xi}, $\tau$ from \equref{tau-a}, the 2-cell from \prref{delta}, the structure 2-cells $\F^2, \F^0$ of $\F$ and naturality 
2-cells. The third part follows by iterating the first 2-cell in the second part and noticing the following commutativity relations (up to an invertible 2-cell): 
\begin{itemize}
\item $a^{\N,j}_\bullet\circ m_n(1,\xi_2^i,1)\iso m_{n-1}(1,\xi_2^i,1)\circ a^{\N,j}_\bullet$ for $j$ different than $i,i-1$, 
\item $a^{\N,i-1}_\bullet \circ m_n(1,\xi_2^i,1)\iso m_{n-1}(1,m_2(1,\xi_2^i),1)\circ a^{\N,i-1}_\bullet$, and  
\item $a^{\N,i}_\bullet \circ m_n(1,\xi_2^i,1)\iso m_{n-1}(1,m_2(\xi_2^i,1),1)\circ a^{\N,i}_\bullet$ 
\end{itemize}  
of 1-cells acting on the object $m_n(\F(a_1),...,m_2(\F(a_i),\F(a_{i+1})),...,\F(a_{n+1}))$.  
The part 4) follows from 3) or by iterating 1). 
\qed\end{proof}

\begin{defn} \delabel{four commutativities}
The three invertible commutativity 2-cells from the above proof of part 3) and the obvious 
$$m_n(1,\xi_2^i,1)\circ m_n(1,\xi_2^j,1)\iso m_n(1,\xi_2^j,1)\circ m_n(1,\xi_2^i,1)$$
with $\vert j-i\vert\geq 2$ we will call {\em the four canonical commutativity 2-cells for $\xi_2$}. 
\end{defn}

\smallskip

Let us denote  
\begin{equation} \eqlabel{path xi}
path(\xi_{n+1}):=
\F(a^\M_\bullet)\circ \xi_2\circ m_2(1,\xi_2^{i_1})\circ m_3(1,\xi_2^{i_2},1)\circ m_4(1,\xi_2^{i_3},1)\circ...\circ m_n(1,\xi_2^{i_{n-1}},1)
  \circ (a^\N_\bullet)^*
	\end{equation} 
the codomain of the 2-cell in part 3) in the above proposition. In view of \coref{path a_i} we now have: 

\begin{cor} \colabel{path a_i}
Given two paths $a_{\crta{a_i}}, a_{\crta{a_i}}':m(a_1,...,a_n)\to m'(a_1,...,a_n)$ as in \rmref{j and j'}, the unique 2-cell 
$\F(a_{\crta{a_i}}) \xi_n\Rightarrow \xi'_n a_{\crta{\F(a_i)}}$ from \equref{delta_n za niz a_i} induces a unique invertible 2-cell
\begin{equation}\eqlabel{delta_n forma nizova}
path\big(\F(a'_{\crta{a_i}})\big) \circ path(\xi_n)\Rightarrow path(\xi'_n)\circ  a_{\crta{\F(a_i)}}.
\end{equation}
\end{cor}



We now may summarize our findings about pseudofunctors of bi-representable bi-multicategories from \ssref{fun-birep-bimult} as follows. 
To recapture, we remind the reader that 
the path of arrows $a_{a_1,...,a_n}:m(a_1,...,a_n)\to m'(a_1,...,a_n)$, 
the 2-cells $\pi:a_{\crta{a_i}}'\Rightarrow a_{\crta{a_i}}$ and 
$\mu_{x,m(-),z}: m(id_x,l_z)\circ a_{x, m(-), z}\Rightarrow m(r_x, id_z)$ for 1-cells  
$l_x:m(m(-),x)\to x$ and $r_x:m(x,m(-))\to x$, which are images by $\Omega_{-,x}$ and $\Omega_{x,-}$ of $id_x$, respectively, 
are as in \rmref{j and j'}.

\begin{prop} \prlabel{psf birep}
Let $\M$ and $\N$  be bi-representable bi-multicategories and $\F:\M\to\N$ a pseudofunctor on their underlying bi-multicategories. 
Then $\F$ induces: 
\begin{itemize}
\item a sequence of pseudonatural transformations of pseudofunctors $\xi_n: m(\F(-),..,\F(-))\to \F(m(-,..,-)): \M_n\to\N_1$ 
and a unary 1-cell $\xi_0: m(-)\to\F(m(-))$ in $\N$, as in \prref{ksi}; 
\item invertible modifications $\delta_n:\F(a_{\crta{(-)}}') \xi_n(\crta{(-)})\Rrightarrow \xi'_n(\crta{(-)}) a_{\crta{\F(-)}}$ for $n\geq 3$ 
%
so that 
for a list $a_1,...,a_{n+1}$ 
the 2-cell $\delta_{n+1}(a_1,...,a_{n+1})$ considered in the form of \equref{delta_n forma nizova} 
equals any sensible composition made of the 2-cells 
$$m(\Id_{id_{\F(a_1)}},\delta_n(a_2,...,a_{n+1})), \quad m(\delta_n(a_1,...,a_n), \Id_{id_{\F(a_{n+1})}}), \quad \delta_n(a_1,..,m_2(a_i,a_{i+1}),..,a_{n+1})$$ 
for $i\in\{1,...,n\}$, the four canonical commutativity 2-cells for $\xi_2$ (from \deref{four commutativities}), $\pi^\M$ and $\pi^\N$ from \equref{tau-a}, and 
naturalities of the transformation $\xi_2$ with respect to $a_{a_1,...,a_n}, a_{a_1,...,a_n}':m(a_1,...,a_n)\to m'(a_1,...,a_n)$, 
\item two additional invertible modifications 
$$\delta^L:\F(l_{-})\circ\xi_2\circ m(\xi_0,id_{-})\Rightarrow l_{\F(-)}\quad\text{and}\quad 
 \delta^R:\F(r_{-})\circ\xi_2\circ m(id_{-},\xi_0)\Rightarrow r_{\F(-)}$$
\end{itemize}
so that the modifications $\delta, \delta^L, \delta^R$ are related as follows:  
any sensible composition made of the 2-cells $\delta_3(x,m(-),y)$ considered in the form of \equref{delta_n forma nizova},  
$m(\Id_{id_{\F(x)}},\delta^L(y))$ and $m(\delta^R(x),\Id_{id_{\F(y)}})$, \, $\mu_{x,m(-),z}$, and 
naturalities of the transformation $\xi_2$ with respect to $l_y, r_x$ and of $a_{\bullet}$ with respect to $\xi_0$ equals the unique 2-cell 
existing by the universal property between 1-cells 
\begin{multline*}
m(\F(x),\F(y))\xrightarrow{\displaystyle \xi_2} \F(m(x,y)) \xrightarrow{\displaystyle \F(m(r^{-1}_x,1))} \F(m(m(x,m(-)),y)) \\
\xrightarrow{\displaystyle \F(a_{x,m(-),y})} \F(m(x,m(m(-),y)))  \xrightarrow{\displaystyle \F(m(1,l_y))} \F(m(x,y)) 
\end{multline*}
and $\xi_2$. 
\end{prop}

\medskip

In the above proposition it only remains to prove the two coherence conditions. They hold by the universal property, in the style of the proof 
of \prref{get modif delta for free}. 

\medskip

From now on for a pseudofunctor $\F:\M\to\N$ between the underlying bi-multicategories of bi-representable bi-multicategories $\M$ and $\N$ we will write 
$(\F, \xi, \delta)$, where $\xi$ stands for short for the pair $\xi=( (\xi_n)_{n\in\Nn_{\geq 2}}, \xi_0)$ consisting of a sequence of pseudonatural transformations 
 $\xi_n$ and the unary 1-cell $\xi_0$, and $\delta=((\delta_n)_{n\in\Nn_{\geq 3}}, \delta^L, \delta^R)$ stands for short for the invertible  modifications as in the above proposition.

\subsubsection{Transformations and modifications for bi-multicategories}

With the detailed proofs we have given so far, the reader should be convinced that the following results hold true. 
They concern the transformations and modifications between pseudofunctors of bi-multicategories $\M$ and $\N$, whereby$\M$ and $\N$ are bi-representable. 
The existence of the highlighted 2-cells follows from the universal property in all cases, and so do hold the identities between the involved 2-cells, 
since the domain of the 1-cell on which the compositions of the mentioned 2-cells act is a universal object. For future use we only record that the 2-cell 
$\Theta_n(\crta a)$ is defined as the unique 2-cell between the right-hand side compositions of 1-cells in the below diagrams existing by the universal 
property for the given 1-cell given in the left-hand side composition of 1-cells and the indicated 2-cells:
\begin{equation} \eqlabel{ind Theta_n}
	\scalebox{0.86}{
\bfig
 \putmorphism(330,450)(1,-1)[ `\F(m(a_1,...,a_n)))`\F(j_{\crta a})]{540}1l
\putmorphism(320,450)(1,0)[\F(a_1),...,\F(a_n) `m(\F(a_1),...,\F(a_n))`j_{\crta{\F(a_i)}}]{1200}1a
 \putmorphism(1500,460)(-1,-1)[\phantom{a_1,...,a_n} ``\xi_n^\F(\crta a)]{550}1r 
\put(830,240){\fbox{$\sigma^\F $}}

\putmorphism(260,450)(0,-1)[ `\G(a_1),...,\G(a_n)`\crta{\theta(a_i)}]{900}1l
\putmorphism(890,-90)(0,-1)[ `\G(m(a_1,...,a_n)))`\theta(m(\crta{a_i}))]{900}1r
\put(380,-270){\fbox{$\theta_{j_{\crta{a_i}}}$}}
 \putmorphism(250,-460)(1,-1)[ ``\G(j_{\crta a})]{540}1l

\efig}
=\quad
	\scalebox{0.86}{
\bfig
\putmorphism(320,450)(1,0)[\F(a_1),...,\F(a_n) `m(\F(a_1),...,\F(a_n))`j_{\crta{\F(a_i)}}]{1200}1a
\putmorphism(320,450)(0,-1)[ `\G(a_1),...,\G(a_n)`\crta{\theta(a_i)}]{400}1l
\putmorphism(1520,450)(0,-1)[ `m(\G(a_1),...,\G(a_n))`m(\crta{\theta(a_i)})]{400}1r
\putmorphism(320,0)(1,0)[\phantom{\F(a_1),...,\F(a_n)} `\phantom{m(\F(a_1),...,\F(a_n))}`j_{\crta{\G(a_i)}}]{1200}1a
\put(830,-230){\fbox{$\sigma^\G $}}
\put(720,300){\fbox{$\Epsilon_{m(\crta{\theta(a_i)})} $}}
 \putmorphism(330,0)(1,-1)[ `\G(m(a_1,...,a_n))).`\G(j_{\crta a})]{540}1l
 \putmorphism(1500,10)(-1,-1)[\phantom{A_1,...,A_n} ``\xi_n^\G(\crta a)]{550}1r 
\efig}
\end{equation}

\begin{prop} \prlabel{transf}
Let $\M$ and $\N$  be bi-representable bi-multicategories and $\F, \G:\M\to\N$  pseudofunctors on their underlying bi-multicategories. 
A pseudonatural transformation $\theta: \F\Rightarrow \G$ of pseudofunctors of bi-multicategories 
(indeed of pseudofunctors $(\F, \xi^\F, \delta^\F)\Rightarrow (\G, \xi^\G,\delta^\G)$) 
induces a sequence of invertible modifications $\Theta_n: \xi_n^\G(-) m(\crta{\theta_{-}})\Rrightarrow\theta_{m(-)}\xi_n^\F(-)$ (between horizontal 
compositions of pseudonatural transformations) for $n\in\Nn_{\geq 1}$ with 2-cell components 
$$
	\scalebox{0.86}{
\bfig
 \putmorphism(-150,470)(1,0)[m(\F(a_1),..,\F(a_n))`m(\G(a_1),..,\G(a_n))`m(\theta_{a_1},...,\theta_{a_n})]{1600}1a
\putmorphism(-160,470)(0,-1)[\phantom{Y_2}` ` \xi_n^\F(\crta a)]{380}1l
\putmorphism(1500,470)(0,-1)[\phantom{Y_2}` `  \xi_n^\G(\crta a)]{380}1r
 \putmorphism(-150,70)(1,0)[\F(m(a_1,..,a_n))`\G(m(a_1,..,a_n))` \theta_{m(\crta a)}]{1600}1b
\put(460,240){\fbox{$\SWarrow \Theta_n(\crta a)$}}
\efig}
$$
meaning that for every list of arrows $\crta f:\crta a\to\crta b$ 
the following identities hold  
$$\scalebox{0.8}{
\bfig
 \putmorphism(450,700)(1,0)[m(\crta{\F(b)}) `m(\crta{\G(b)}) `m(\crta{\theta_{b}})]{680}1a
 \putmorphism(1160,700)(1,0)[\phantom{A\ot B}`\G(m(\crta b)) ` \xi_n^\G(\crta b)]{660}1a

 \putmorphism(-190,250)(1,0)[m(\crta{\F(a)}) `m(\crta{\F(b)}) ` m(\crta{\F(f)})]{640}1a
 \putmorphism(490,250)(1,0)[\phantom{A\ot B}`\F(m(\crta b))` \xi_n^\F(\crta b)]{660}1a
 \putmorphism(1200,250)(1,0)[\phantom{A\ot B}`\G(m(\crta b)) ` \theta_{m(\crta b)}]{600}1a

\putmorphism(450,700)(0,-1)[\phantom{Y_2}``=]{450}1r
\putmorphism(1800,700)(0,-1)[\phantom{Y_2}``=]{450}1r
\put(1020,450){\fbox{$ \Theta(\crta b)$}}

 \putmorphism(-190,-200)(1,0)[m(\crta{\F(a)})`\F(m(\crta{a})) `\xi_n^\F(\crta a)]{640}1a
 \putmorphism(530,-200)(1,0)[\phantom{A'\ot B'}`\F(m(\crta b)) `\F(m(\crta f))]{630}1a
 \putmorphism(1200,-200)(1,0)[\phantom{A\ot B}`\G(m(\crta b)) ` \theta_{m(\crta b)}]{630}1a

\putmorphism(-220,250)(0,-1)[\phantom{Y_2}``=]{450}1l
\putmorphism(1120,250)(0,-1)[\phantom{Y_3}``=]{450}1r
\put(310,0){\fbox{$ \xi_n^\F(\crta f)$}}

\putmorphism(450,-200)(0,-1)[\phantom{Y_2}``=]{450}1r
\putmorphism(1800,-200)(0,-1)[\phantom{Y_2}``=]{450}1r

 \putmorphism(440,-650)(1,0)[\F(m(\crta{a})) `\G(m(\crta a)) `\theta_{m(\crta a)}]{720}1a
 \putmorphism(1340,-650)(1,0)[ `\G(m(\crta b)) `\G(m(\crta f))]{500}1a
\put(940,-440){\fbox{$ \theta(m(\crta f))$}}
\efig}
=
\scalebox{0.8}{  
\bfig

 \putmorphism(-190,700)(1,0)[m(\crta{\F(a)}) `  m(\crta{\F(b)})` m(\crta{\F(f)}) ]{640}1a
 \putmorphism(490,700)(1,0)[\phantom{A\ot B}`m(\crta{\G(b)}) `m(\crta{\theta_{b}})]{660}1a

 \putmorphism(-190,250)(1,0)[m(\crta{\F(a)})`m(\crta{\G(a)})`m(\crta{\theta_{a}})]{660}1a
 \putmorphism(550,250)(1,0)[\phantom{A\ot B}`m(\crta{\G(b)}) `m(\crta{\G(f)})]{680}1a
 \putmorphism(1240,250)(1,0)[\phantom{A'\ot B'}`\G(m(\crta b)) `\xi_n^\G(\crta b)]{650}1a 

\putmorphism(-220,700)(0,-1)[\phantom{Y_2}``=]{450}1r
\putmorphism(1100,700)(0,-1)[\phantom{Y_2}``=]{450}1r
\put(350,450){\fbox{$m(\crta{\theta_{f}})$}}
\put(1000,0){\fbox{$\xi_n^\G(\crta f)$}}

 \putmorphism(-150,-200)(1,0)[m(\crta{\F(a)})`m(\crta{\G(a)})` m(\crta{\theta_{a}})]{600}1a
 \putmorphism(460,-200)(1,0)[\phantom{A''\ot B'}` \G(m(\crta{a})) `\xi_n^\G(\crta a)]{680}1a
 \putmorphism(1160,-200)(1,0)[\phantom{A''\ot B'}`\G(m(\crta b)) ` \G(m(\crta f))]{660}1a

\putmorphism(450,250)(0,-1)[\phantom{Y_2}``=]{450}1l
\putmorphism(1750,250)(0,-1)[\phantom{Y_2}``=]{450}1r
\putmorphism(-180,-200)(0,-1)[\phantom{Y_2}``=]{450}1r
\putmorphism(1100,-200)(0,-1)[\phantom{Y_2}``=]{450}1r
 \putmorphism(-150,-650)(1,0)[m(\crta{\F(a)})`\F(m(\crta{a}))` \xi_n^\F(\crta a)]{600}1a
 \putmorphism(630,-650)(1,0)[`\G(m(\crta{a}))`\theta_{m(\crta a)}]{540}1a
\put(320,-440){\fbox{$ \Theta(\crta a)$}} 
\efig} \vspace{-0,2cm}
$$
whereby in writing the 2-cell $ m(\crta{\theta_{f}})$ on the right top we omit writing the laxity and colaxity pseudofunctor 
structure of $m(-)$. 
For $n=0$ and the empty list, $\Theta_n(\crta a)$ yields a 2-cell 
$$\begin{tikzcd}
	I^\N & I^\N & \\
	{\mathcal F(I^\M)} & {\mathcal G(I^\M)} \\
	{} & {} & {} \\
	\\
	& {}
	\arrow[""{name=0, anchor=center, inner sep=0}, "{id_I}", from=1-1, to=1-2]
	\arrow["{\xi_0^{\mathcal F}}"', from=1-1, to=2-1]
	\arrow["{\xi_0^{\mathcal G}}", from=1-2, to=2-2]
	\arrow[""{name=1, anchor=center, inner sep=0}, "{\theta_I}"', from=2-1, to=2-2]
	\arrow["{\Theta_0}"{description}, draw=none, from=0, to=1]
\end{tikzcd} \vspace{-2,5cm} $$
where $I=m(-)$ in both cases. The above 2-cell components satisfy the following properties. 
\begin{enumerate}
\item ({\em factorization}) They factor as the compositions of 2-cells:  
\begin{itemize}
\item 
$$
	\scalebox{0.86}{
\bfig
  \putmorphism(-850,970)(1,0)[m_n(\crta{\F(a_i)})`m_n(\crta{\G(a_i)})` m_n(\crta{\theta_{a_i}})]{3900}1a
  \putmorphism(-890,950)(2,-1)[``a]{880}1r
  \putmorphism(-1000,970)(0,-1)[``]{1950}1r
  \putmorphism(-1000,920)(0,-1)[``\xi^\F_n]{1900}0r
\putmorphism(0,530)(1,0)[``   m_{n-1}(\theta_{a_1},.., m_2(\theta_{a_i},\theta_{a_{i+1}}),..,\theta_{a_n})]{2200}0a
 \putmorphism(0,470)(1,0)[m_{n-1}(\F(a_1),..,m_2(\F(a_i),\F(a_{i+1})),..\F(a_n))`m_{n-1}(\G(a_1),...,m_2(\G(a_i),\G(a_{i+1})),...\G(a_n))`
   ]{2200}1a
\putmorphism(0,470)(0,-1)[\phantom{Y_2}` ` m_{n-1}(1,\xi_2,1)]{380}1l
\putmorphism(2050,470)(0,-1)[\phantom{Y_2}` `  m_{n-1}(1,\xi_2',1)]{380}1r  
 \putmorphism(0,70)(1,0)[m_{n-1}(\F(a_1),..,\F(m_2(a_i,a_{i+1})),..\F(a_n))`m_{n-1}(\G(a_1),..,\G(m_2(a_i,a_{i+1})),..\G(a_n))` ]{2200}1b
 \putmorphism(0,20)(1,0)[`` m_{n-1}(\crta{\theta_{m(\crta{a_{<}})}}, \theta(m_2(a_i,a_{i+1})),  \crta{\theta_{m(\crta{a_{>}})}})]{2200}0b

\put(580,260){\fbox{$m_{n-1}(1,\Theta_2(a_i,a_{i+1}),1)$}}
 \putmorphism(2050,70)(0,-1)[\phantom{\F(m(a_1,..,a_n))} ` ` \xi_{n-1}^\G]{560}1r
 \putmorphism(0,70)(0,-1)[\phantom{\F(m(a_1,..,a_n))} `\F(m_{n-1}(a_1,..m_2(a_i,a_{i+1}),..,a_n))` \xi_{n-1}^\F]{560}1l  
 \putmorphism(-170,-500)(1,0)[\phantom{m_{n-1}(\F(a_1),..,\F(m_2(a_i,a_{i+1})),..\F(a_n))}`\G(m_{n-1}(a_1,..m_2(a_i,a_{i+1}),..,a_n))` 
   \theta_{(a_1,..m_2(a_i,a_{i+1}),..,a_n)}]{2270}1b
\put(460,-320){\fbox{$\Theta_{n-1}(a_1,..m_2(a_i,a_{i+1}),..,a_n)$}}
 \putmorphism(-850,-970)(1,0)[\F(m_n(\crta{a}))`\G(m_n(\crta{a}))` \theta_{m_n(\crta{a})}]{3900}1b
\put(1000,-810){\fbox{$\theta_{a^*}^{-1}$}}
 \putmorphism(100,-480)(-2,-1)[`` \F(a^*)]{1000}1l  
 \putmorphism(1950,-510)(2,-1)[\phantom{\F(m(a_1,..,a_n))} ` ` \G(a^*)]{980}1r

\putmorphism(3100,950)(-2,-1)[``a]{880}1l 
\put(1000,790){\fbox{$a_{\crta{\theta_{a_i}}}$}}

\putmorphism(3100,970)(0,-1)[``]{1950}1l
\putmorphism(3100,920)(0,-1)[``\xi^\G_n]{1950}0l
\put(2630,-260){\fbox{$\pref{lax mon mult}$}}  
\put(-660,-260){\fbox{$\pref{lax mon mult}$}}
\efig}
$$
for $j=1,...,n-1$, where $a^*$ is a quasi-inverse of the associativity $a$, 
\item by iteration they also factor through compositions of $n-1$ \, 2-cells 
\begin{equation} \eqlabel{decomposition} 
m_{n-1}(\bullet), ..., m_2(\bullet), \Theta_2(a_j',a'_{j+1}),
\end{equation}
where the places $\bullet$ depend on the 2-cell $\Theta_2(a_j',a'_{j+1})$ placed at different positions, for some binary list of objects 
$a_j',a'_{j+1}$ made of the elements of the list $\crta a$ (any such factorization involves either $m_2(\bullet,\Theta_2(a_j',a'_{j+1}))$ 
or $m_2(\Theta_2(a_j',a'_{j+1}),\bullet)$). 
\end{itemize} 
By the universal property all the above factorizations are equal between each other. 
\item ({\em associativity}) 
For another universal object $m'(\crta{a})$ and the induced structure transformations $\xi_n^{'\F}$ and $\xi_n^{'\G}$ for $\F$ and $\G$, respectively, 
there are induced modifications $\Theta_n'$. 
For the associativity $a:m(\crta{a})\to m'(\crta{a})$ the 2-cell components $\Theta_n(\crta a)$ and $\Theta_n'(\crta a)$ obey the associativity condition 
$$\scalebox{0.8}{
\bfig
 \putmorphism(450,700)(1,0)[\F(m(\crta{a})) `\F(m'(\crta a)) `\F(a_{\crta a})]{680}1a
 \putmorphism(1180,700)(1,0)[\phantom{A\ot B}`\G(m'(\crta a)) ` \theta_{m'(\crta a)}]{640}1a

 \putmorphism(-190,250)(1,0)[m(\crta{\F(a_i)}) `\F(m(\crta{a}))` \xi_n^\F]{640}1a
 \putmorphism(490,250)(1,0)[\phantom{A\ot B}`\G(m(\crta a))` \theta_{m(\crta a)}]{660}1a
 \putmorphism(1200,250)(1,0)[\phantom{A\ot B}`\G(m'(\crta a)) ` \G(a_{\crta a})]{600}1a

\putmorphism(450,700)(0,-1)[\phantom{Y_2}``=]{450}1r
\putmorphism(1800,700)(0,-1)[\phantom{Y_2}``=]{450}1r
\put(1050,450){\fbox{$ \theta_{a_{\crta a}}$}}

 \putmorphism(-190,-200)(1,0)[m(\crta{\F(a_i)})`m(\crta{\G(a_i)})  `m(\crta{\theta_{a_i}})]{640}1a
 \putmorphism(480,-200)(1,0)[\phantom{A'\ot B'}`\G(m(\crta a)) `\xi_n^\G]{680}1a
 \putmorphism(1200,-200)(1,0)[\phantom{A\ot B}`\G(m'(\crta a)) ` \G(a_{\crta a})]{630}1a

\putmorphism(-220,250)(0,-1)[\phantom{Y_2}``=]{450}1l
\putmorphism(1120,250)(0,-1)[\phantom{Y_3}``=]{450}1r
\put(310,0){\fbox{$ \Theta_n^{-1}$}}

\putmorphism(450,-200)(0,-1)[\phantom{Y_2}``=]{450}1r
\putmorphism(1800,-200)(0,-1)[\phantom{Y_2}``=]{450}1r

 \putmorphism(440,-650)(1,0)[m(\crta{\G(a_i)}) `m'(\crta{\G(a_i)}) `a_{\crta{\G(a_i)}}]{720}1a
 \putmorphism(1340,-650)(1,0)[ `\G(m'(\crta a)) `\xi_n^{'\G}]{500}1a
\put(1000,-440){\fbox{$\delta_n^\G $}}
\efig}
=
\scalebox{0.8}{  
\bfig

 \putmorphism(-190,700)(1,0)[ m(\crta{\F(a_i)})`\F(m(\crta{a}))` \xi_n^\F]{640}1a
 \putmorphism(490,700)(1,0)[\phantom{A\ot B}`m'(\crta{\G(a_i)}) `\F(a_{\crta a})]{660}1a

 \putmorphism(-190,250)(1,0)[m(\crta{\F(a_i)})`m'(\crta{\F(a_i)})`a_{\crta{\F(a_i)}}]{660}1a
 \putmorphism(550,250)(1,0)[\phantom{A\ot B}`\F(m'(\crta{a})) `\xi_n^{'\F}]{680}1a
 \putmorphism(1270,250)(1,0)[\phantom{A'\ot B'}`\G(m'(\crta a)) `\theta_{m'(\crta a)}]{620}1a 

\putmorphism(-220,700)(0,-1)[\phantom{Y_2}``=]{450}1r
\putmorphism(1170,700)(0,-1)[\phantom{Y_2}``=]{450}1r
\put(350,490){\fbox{$\delta_n^\F$}}
\put(1000,0){\fbox{$\Theta_n'^{-1}$}}

 \putmorphism(-190,-200)(1,0)[m(\crta{\F(a_i)})`m'(\crta{\F(a_i)})` a_{\crta{\F(a_i)}}]{640}1a
 \putmorphism(460,-200)(1,0)[\phantom{A''\ot B'}` m'(\crta{\G(a_i)}) `m'(\crta{\theta_{a_i}})]{680}1a
 \putmorphism(1160,-200)(1,0)[\phantom{A''\ot B'}`\G(m'(\crta a)) ` \xi_n^{'\G}]{660}1a

\putmorphism(450,250)(0,-1)[\phantom{Y_2}``=]{450}1l
\putmorphism(1750,250)(0,-1)[\phantom{Y_2}``=]{450}1r
\putmorphism(-180,-200)(0,-1)[\phantom{Y_2}``=]{450}1r
\putmorphism(1100,-200)(0,-1)[\phantom{Y_2}``=]{450}1r
 \putmorphism(-190,-650)(1,0)[m(\crta{\F(a_i)}) `  m(\crta{\G(a_i)})` m(\crta{\theta_{a_i}})]{620}1a
 \putmorphism(610,-650)(1,0)[`\F(m'(\crta a)). `a_{\crta{\G(a_i)}}]{560}1a
\put(320,-440){\fbox{$a_{\crta{\theta_{a_i}}} $}} 
\efig} \vspace{-0,2cm}
$$
\item ({\em unitality}) 
 For any object $a$ they obey the following unitality conditions: the composites  
$$\begin{tikzcd}[column sep=small]
	{m(I,\mathcal F(a))} & {m(I,\mathcal G(a))} \\
	{m(\mathcal F(I),\mathcal F(a))} & {m(\mathcal G(I),\mathcal G(a))} \\
	{\mathcal F(m(I,a))} & {\mathcal G(m(I,a))} \\
	{\mathcal F(a)} & {\mathcal G(a)}
	\arrow[""{name=0, anchor=center, inner sep=0}, "{m(I,\theta(a))}", from=1-1, to=1-2]
	\arrow["{m(\xi_0^{\mathcal F},1)}"', from=1-1, to=2-1]
	\arrow["{m(\xi_0^{\mathcal G},1)}", from=1-2, to=2-2]
	\arrow[""{name=1, anchor=center, inner sep=0}, "{m(\theta_I,\theta_a)}", from=2-1, to=2-2]
	\arrow["{\xi_2^{\mathcal F}}"', from=2-1, to=3-1]
	\arrow["{\xi_2^{\mathcal G}}", from=2-2, to=3-2]
	\arrow[""{name=2, anchor=center, inner sep=0}, "{\theta_{m(I,a)}}"', from=3-1, to=3-2]
	\arrow["{\mathcal F(l_a)}"', from=3-1, to=4-1]
	\arrow["{\mathcal G(l_a)}", from=3-2, to=4-2]
	\arrow[""{name=3, anchor=center, inner sep=0}, "{\theta_a}"', from=4-1, to=4-2]
	\arrow["{m(\Theta_0,1_{\theta_a})}"{description, pos=0.4}, draw=none, from=0, to=1]
	\arrow["{\Theta_2(I,a)}"{description}, draw=none, from=1, to=2]
	\arrow["{\theta_{l_a}}"{description, pos=0.6}, draw=none, from=2, to=3]
\end{tikzcd} 
\quad\text{and} \qquad 
\begin{tikzcd}[column sep=small]
	{m(\mathcal F(a),I)} & {m(\mathcal G(a),I)} \\
	{m(\mathcal F(a),\mathcal F(I))} & {m(\mathcal G(a),\mathcal G(I))} \\
	{\mathcal F(m(a,I))} & {\mathcal G(m(a,I))} \\
	{\mathcal F(a)} & {\mathcal G(a)}
	\arrow[""{name=0, anchor=center, inner sep=0}, "{m(\theta(a),I)}", from=1-1, to=1-2]
	\arrow["{m(1,\xi_0^{\mathcal F})}"', from=1-1, to=2-1]
	\arrow["{m(1,\xi_0^{\mathcal G})}", from=1-2, to=2-2]
	\arrow[""{name=1, anchor=center, inner sep=0}, "{m(\theta_a,\theta_I)}"', from=2-1, to=2-2]
	\arrow["{\xi_2^{\mathcal F}}"', from=2-1, to=3-1]
	\arrow["{\xi_2^{\mathcal G}}", from=2-2, to=3-2]
	\arrow[""{name=2, anchor=center, inner sep=0}, "{\theta_{m(a,I)}}"', from=3-1, to=3-2]
	\arrow["{\mathcal F(r_a)}"', from=3-1, to=4-1]
	\arrow["{\mathcal G(r_a)}", from=3-2, to=4-2]
	\arrow[""{name=3, anchor=center, inner sep=0}, "{\theta_a}"', from=4-1, to=4-2]
	\arrow["{m(1_{\theta_a},\Theta_0)}"{description}, draw=none, from=0, to=1]
	\arrow["{\Theta_2(a,I)}"{description, pos=0.6}, draw=none, from=1, to=2]
	\arrow["{\theta_{r_a}}"{description, pos=0.6}, draw=none, from=2, to=3]
\end{tikzcd}$$ 
should be equal to the composites 
$$\begin{tikzcd}[column sep=small]
	{m(I,\mathcal F(a))} & {m(I,\mathcal F(a))} & {m(I,\mathcal G(a))} & {m(I,\mathcal G(a))} \\
	{m(\mathcal F(I),\mathcal F(a))} &&& {m(\mathcal G(I),\mathcal G(a))} \\
	{\mathcal F(m(I,a))} &&& {\mathcal G(m(I,a))} \\
	{\mathcal F(a)} & {\mathcal F(a)} & {\mathcal G(a)} & {\mathcal G(a)}
	\arrow[""{name=0, anchor=center, inner sep=0}, "{=}", from=1-1, to=1-2]
	\arrow["{m(\xi_0^{\mathcal F},1)}"', from=1-1, to=2-1]
	\arrow[""{name=1, anchor=center, inner sep=0}, "{m(1_I,\theta(a))}", from=1-2, to=1-3]
	\arrow["{l_{\mathcal F(a)}}"', from=1-2, to=4-2]
	\arrow["{=}", from=1-3, to=1-4]
	\arrow[""{name=2, anchor=center, inner sep=0}, "{l_{\mathcal G(a)}}", from=1-3, to=4-3]
	\arrow["{m(\xi_0^{\mathcal G},1)}", from=1-4, to=2-4]
	\arrow["{\xi_2^{\mathcal F}}"', from=2-1, to=3-1]
	\arrow[""{name=3, anchor=center, inner sep=0}, "{\xi_2^{\mathcal G}}", from=2-4, to=3-4]
	\arrow["{\mathcal F(l_a)}"', from=3-1, to=4-1]
	\arrow["{\mathcal G(l_a)}", from=3-4, to=4-4]
	\arrow[""{name=4, anchor=center, inner sep=0}, "{=}"', from=4-1, to=4-2]
	\arrow[""{name=5, anchor=center, inner sep=0}, "{\theta_a}"', from=4-2, to=4-3]
	\arrow["{=}"', from=4-3, to=4-4]
	\arrow["{\delta^{L,\mathcal F}_a}"{description}, draw=none, from=0, to=4]
	\arrow["{\delta^{L,\mathcal G}_a}"{description}, draw=none, from=2, to=3]
	\arrow["{l_{\theta_a}}"{description}, draw=none, from=5, to=1]
\end{tikzcd}$$ 
and 
\[\begin{tikzcd}[column sep=small]
	{m(\mathcal F(a),I)} & {m(\mathcal F(a),I)} & {m(\mathcal G(a),I)} & {m(\mathcal G(a),I)} \\
	{m(\mathcal F(a),\mathcal F(I))} &&& {m(\mathcal G(a),\mathcal G(I))} \\
	{\mathcal F(m(a,I))} &&& {\mathcal G(m(a,I))} \\
	{\mathcal F(a)} & {\mathcal F(a)} & {\mathcal G(a)} & {\mathcal G(a)}
	\arrow["{=}", from=1-1, to=1-2]
	\arrow["{m(1,\xi_0^{\mathcal F})}"', from=1-1, to=2-1]
	\arrow[""{name=0, anchor=center, inner sep=0}, "{m(\theta(a),1_I)}", from=1-2, to=1-3]
	\arrow[""{name=1, anchor=center, inner sep=0}, "{r_{\mathcal F(a)}}"', from=1-2, to=4-2]
	\arrow["{=}", from=1-3, to=1-4]
	\arrow[""{name=2, anchor=center, inner sep=0}, "{r_{\mathcal G(a)}}", from=1-3, to=4-3]
	\arrow["{m(1,\xi_0^{\mathcal G})}", from=1-4, to=2-4]
	\arrow[""{name=3, anchor=center, inner sep=0}, "{\xi_2^{\mathcal F}}"', from=2-1, to=3-1]
	\arrow[""{name=4, anchor=center, inner sep=0}, "{\xi_2^{\mathcal G}}", from=2-4, to=3-4]
	\arrow["{\mathcal F(r_a)}"', from=3-1, to=4-1]
	\arrow["{\mathcal G(r_a)}", from=3-4, to=4-4]
	\arrow["{=}"', from=4-1, to=4-2]
	\arrow[""{name=5, anchor=center, inner sep=0}, "{\theta_a}"', from=4-2, to=4-3]
	\arrow["{=}"', from=4-3, to=4-4]
	\arrow["{r_{\theta_a}}"{description}, draw=none, from=0, to=5]
	\arrow["{\delta^{R,\mathcal G}_a}"{description}, draw=none, from=2, to=4]
	\arrow["{\delta^{R,\mathcal F}_a}"{description}, draw=none, from=3, to=1]
\end{tikzcd}\]
respectively.  
\end{enumerate}
\end{prop}

We remind the reader that the 2-cells $a_{a_1,...,a_n}, l_a$ and $r_a$ in the points 2) and 3) above are as in \rmref{j and j'}, parts b) and c). 

\begin{rem} \rmlabel{Theta}
Let us remark two things. Firstly, 
in the above proposition we indeed want to work with {\em pseudonatural} transformations $\theta: \F\Rightarrow \G$, since the factorization of 
the invertible modification component $\Theta_n(\crta a)$ via $\Theta_{n-1}(\crta a)$ includes the component $\theta_{a^*}$ 
(the fact that there appears $\theta_{a^*}^{-1}$ is merely due to our convention that we denote all transformations in their oplax form). 

Secondly, in the factorization \equref{decomposition} appear also the 2-cell components $a_{1,\xi^\mathcal F,1}, a_{1,\xi^\mathcal G,1}$, 
$a_{\overline{\theta_{a_i}},m(\overline{\theta_{b_i}}),\overline{\theta_{a_i}}}$, where $\crta b=b_1,...,b_k$ is any subsequence of $a_1,...,a_n$ 
at an ``interruption'' variable $l$. At any $l$ and for any $k$ the following compatibilities hold \\
\begin{tikzcd}
	&& \bullet && \bullet && \bullet \\
	\bullet && \bullet && \bullet && \bullet \\
	\bullet && \bullet && \bullet && \bullet \\
	&& \bullet && \bullet && \bullet
	\arrow[""{name=0, anchor=center, inner sep=0}, "{a_{\overline{\mathcal F(a_i)}}}", from=1-3, to=1-5]
	\arrow["{=}"', from=1-3, to=2-3]
	\arrow["{m'(\overline{\theta_{a_i}},\theta_{m(\overline b)},\overline{\theta_{a_i}})}", from=1-5, to=1-7]
	\arrow["{=}", from=1-7, to=2-7]
	\arrow[""{name=1, anchor=center, inner sep=0}, "{m(1,\xi^\mathcal F,1)}", from=2-1, to=2-3]
	\arrow["{=}"', from=2-1, to=3-1]
	\arrow["{m(\overline{\theta_{a_i}},\theta_{m(\overline b)},\overline{\theta_{a_i}})}", from=2-3, to=2-5]
	\arrow[""{name=2, anchor=center, inner sep=0}, "{a_{\overline{\mathcal G(a_i)},\mathcal G(m(\overline{b})),\overline{\mathcal G(a_i)}}}", from=2-5, to=2-7]
	\arrow["{=}", from=2-5, to=3-5]
	\arrow["{m(\overline{\theta_{a_i}},m(\overline{\theta_{b_i}}),\overline{\theta_{a_i}})}"', from=3-1, to=3-3]
	\arrow[""{name=3, anchor=center, inner sep=0}, "{m(1,\xi^\mathcal G,1)}"', from=3-3, to=3-5]
	\arrow["{=}"', from=3-3, to=4-3]
	\arrow["{a_{\overline{\mathcal G(a_i)},\mathcal G(m(\overline{b})),\overline{\mathcal G(a_i)}}}"', from=3-5, to=3-7]
	\arrow["{=}", from=3-7, to=4-7]
	\arrow["{a_{\overline{\mathcal G(a_i)},m(\overline{\mathcal G(b_i)}),\overline{\mathcal G(a_i)}}}"', from=4-3, to=4-5]
	\arrow[""{name=4, anchor=center, inner sep=0}, "{m'(1,\xi^\mathcal G,1)}"', from=4-5, to=4-7]
	\arrow["{a_{\overline{\theta_{\tilde{a_i}}},m(\overline{\theta_{b_i}}),\overline{\theta_{a_i}}}^{-1}}"{description}, draw=none, from=0, to=2]
	\arrow["{m(1,\Theta_k(\overline b),1)}"{description}, draw=none, from=1, to=3]
	\arrow["{a_{1,\xi^\mathcal G,1}}"{description}, draw=none, from=3, to=4]
\end{tikzcd}
= 
$$\begin{tikzcd}
	\bullet && \bullet && \bullet && \\
	\bullet && \bullet && \bullet && \bullet \\
	\bullet && \bullet && \bullet && \bullet \\
	\bullet && \bullet && \bullet
	\arrow[""{name=0, anchor=center, inner sep=0}, "{m(1,\xi^\mathcal F,1)}", from=1-1, to=1-3]
	\arrow["{=}"', from=1-1, to=2-1]
	\arrow["{a_{\overline{\mathcal F(a_i)}}}", from=1-3, to=1-5]
	\arrow["{=}", from=1-5, to=2-5]
	\arrow["{a_{\overline{\mathcal F(a_i)},m(\overline{\mathcal F(b_i)}),\overline{\mathcal F(a_i)}}}", from=2-1, to=2-3]
	\arrow[""{name=1, anchor=center, inner sep=0}, "{m'(1,\xi^\mathcal F,1)}", from=2-3, to=2-5]
	\arrow["{=}"', from=2-3, to=3-3]
	\arrow["{m'(\overline{\theta_{a_i}},\theta_{m(\overline b)},\overline{\theta_{a_i}})}", from=2-5, to=2-7]
	\arrow["{=}", from=2-7, to=3-7]
	\arrow[""{name=2, anchor=center, inner sep=0}, "{a_{\overline{\mathcal F(a_i)},m(\overline{\mathcal F(b_i)}),\overline{\mathcal F(a_i)}}}", from=3-1, to=3-3]
	\arrow["{=}"', from=3-1, to=4-1]
	\arrow["{m'(\overline{\theta_{a_i}},m(\overline{\theta_{b_i}}),\overline{\theta_{a_i}})}", from=3-3, to=3-5]
	\arrow[""{name=3, anchor=center, inner sep=0}, "{m'(1,\xi^\mathcal G,1)}", from=3-5, to=3-7]
	\arrow["{=}", from=3-5, to=4-5]
	\arrow["{m(\overline{\theta_{a_i}},m(\overline{\theta_{b_i}}),\overline{\theta_{a_i}})}"', from=4-1, to=4-3]
	\arrow[""{name=4, anchor=center, inner sep=0}, "{a_{\overline{\mathcal G(a_i)},m(\overline{\mathcal G(b_i)}),\overline{\mathcal G(a_i)}}}"', from=4-3, to=4-5]
	\arrow["{a_{1,\xi^\mathcal F,1}}"{description}, draw=none, from=0, to=1]
	\arrow["{m'(1,\Theta_k(\overline b),1)}"{description}, draw=none, from=1, to=3]
	\arrow["{a_{\overline{\theta_{a_i}},m(\overline{\theta_{b_i}}),\overline{\theta_{a_i}}}^{-1}}"{description}, draw=none, from=2, to=4]
\end{tikzcd}$$
as they are equivalent to the naturality of the pseudonatural transformation $a$ with respect to the 2-cells $\Theta_k(\crta b)$ 
(the compositions of the two 2-cell components of $a$ present the single 2-cell components of $a$ at the compositions of the corresponding 1-cells). 
\end{rem}

\begin{prop}
For a pseudonatural transformation $\theta: \F\Rightarrow \G$ as in \prref{transf} and a multimap 
$f\in\M(a_1,...,a_n;b)$, the 2-cell components $\theta_f$ in $\N$ factor as follows 
\begin{equation} \eqlabel{theta-f} 
\theta_f=
\begin{tikzcd}
	{\coprod_i \F(a_i)} &&& \\
	&&& {\F(b)} \\
	& {m(\crta{ \F(a_i)})} & {\F(m(\crta{a_i}))} \\
	& {m(\crta{ \G(a_i)})} & {\G(m(\crta{a_i}))} \\
	&&& {\G(b)} \\
	{\coprod_i \G(a_i)}
	\arrow[""{name=0, anchor=center, inner sep=0}, "{\mathcal F(f)}", from=1-1, to=2-4]
	\arrow[""{name=1, anchor=center, inner sep=0}, "{j_{\F(a_i)}}"{description}, from=1-1, to=3-2]
	\arrow[""{name=2, anchor=center, inner sep=0}, "{\F(j)}"{description}, from=1-1, to=3-3]
	\arrow["{\coprod\theta_{a_i}}"', from=1-1, to=6-1]
	\arrow[""{name=3, anchor=center, inner sep=0}, "{\theta_b}", from=2-4, to=5-4]
	\arrow["{\xi_n^\mathcal F}"', from=3-2, to=3-3]
	\arrow[""{name=4, anchor=center, inner sep=0}, "{m(\crta{\theta_{a_i}})}"', from=3-2, to=4-2]
	\arrow["{\F(\hat f)}"{description}, from=3-3, to=2-4]
	\arrow[""{name=5, anchor=center, inner sep=0}, "{\theta_{m(\crta{a_i})}}", from=3-3, to=4-3]
	\arrow["{\xi_n^\mathcal G}", from=4-2, to=4-3]
	\arrow["{j_{\G(a_i)}}"{description}, tail reversed, no head, from=4-2, to=6-1]
	\arrow["{\G(\hat f)}"{description}, from=4-3, to=5-4]
	\arrow[""{name=6, anchor=center, inner sep=0}, "{\G(j)}"{description}, from=6-1, to=4-3]
	\arrow[""{name=7, anchor=center, inner sep=0}, "{\mathcal \G(f)}"', from=6-1, to=5-4]
	\arrow["{\sigma_\F^{-1}}"', between={0.3}{0.9}, Rightarrow, nfold, from=2, to=3-2]
	\arrow["{\F(\Epsilon_f)^{-1}}"{pos=0.7}, between={0.4}{0.7}, Rightarrow, nfold, from=0, to=3-3]
	\arrow["\Epsilon"{description, pos=0.4}, between={0.3}{0.5}, Rightarrow, nfold, from=1, to=6-1]
	\arrow["{\theta_{\hat f}}"{description, pos=0.4}, between={0.1}{0.9}, Rightarrow, nfold, from=3, to=4-3]
	\arrow["{\Theta_n^{-1}(\overline{a_i})}"{description}, between={0.2}{0.8}, Rightarrow, nfold, from=5, to=4]
	\arrow["{\sigma_\G}"'{pos=0.4}, between={0.1}{0.7}, Rightarrow, nfold, from=4-2, to=6]
	\arrow["{\G(\Epsilon_f)}"{pos=0.3}, between={0.2}{0.6}, Rightarrow, nfold, from=4-3, to=7]
\end{tikzcd}
\end{equation}
where $\hat f:m(a_1,...,a_n)\to b$ presents the corresponding unary map of $f$. The 2-cells $\sigma$ are as in \equref{j-xi} and 
$\Epsilon$'s are the associated 2-cells determining 1-cells on universal objects. 
\end{prop}

\begin{proof}
We have a general situation because of bi-representability that given: 1) a 2-cell $g\Rightarrow h:\coprod_i \F(a_i)\to\G(b)$, 
for which we take our $\theta_f$ of interest (on the left below), and 2) 1-cells $\zeta_1, \zeta_2$ (the curvy arrows on the right below) 
acting on a universal object, together with 2-cells (the top and bottom bases of the prism) 
$\sigma_g: \zeta_1\circ j_{\F(a_i)}\Rightarrow g$ and $\sigma_h: \zeta_2\circ j_{\F(a_i)}\Rightarrow h$, 
then on one hand, 
1) there are unique 2-cells $\sigma^\dagger, \tau^\dagger$ as below so that the $\Epsilon_g\diamond\sigma^\dagger=\sigma _g$ and 
$\Epsilon_h\diamond\tau^\dagger=\sigma_h$, and on the other hand, 2) there is a unique  2-cell $\kappa: \zeta_1\Rightarrow\zeta_2$ 
(marked in red, along the outer curved vertical surface on the right) so that $\sigma_h\diamond \kappa\diamond\sigma_g^{-1} = \theta_f$. 
\[\begin{tikzcd}
	{\coprod_i \F(a_i)} && {m(\crta{ \F(a_i)})} \\
	& {\G(b)} \\
	{\coprod_i \F(a_i)} && {m(\crta{ \F(a_i)})} \\
	& {\G(b)}
	\arrow[""{name=0, anchor=center, inner sep=0}, "{j_{\F(a_i)}}", from=1-1, to=1-3]
	\arrow[""{name=1, anchor=center, inner sep=0}, "g"{description}, from=1-1, to=2-2]
	\arrow[equals, nfold, from=1-1, to=3-1]
	\arrow[""{name=2, anchor=center, inner sep=0}, "{\Omega(g)}"{description}, from=1-3, to=2-2]
	\arrow[""{name=3, anchor=center, inner sep=0}, "{\zeta_1}"{description}, curve={height=-30pt}, from=1-3, to=2-2]
	\arrow[equals, nfold, from=1-3, to=3-3]
	\arrow[equals, nfold, from=2-2, to=4-2]
	\arrow[""{name=4, anchor=center, inner sep=0}, "{j_{\F(a_i)}}"{pos=0.7}, dashed, from=3-1, to=3-3]
	\arrow[""{name=5, anchor=center, inner sep=0}, "h"{description}, from=3-1, to=4-2]
	\arrow[""{name=6, anchor=center, inner sep=0}, "{\Omega(h)}"{description}, from=3-3, to=4-2]
	\arrow[""{name=7, anchor=center, inner sep=0}, "{\zeta_2}"{description}, shift left, curve={height=-24pt}, from=3-3, to=4-2]
	\arrow["{\Epsilon_g}", between={0.1}{0.8}, Rightarrow, nfold, from=0, to=1]
	\arrow["{\theta_f}"{description, pos=0.4}, between={0.1}{0.8}, Rightarrow, nfold, from=1, to=5]
	\arrow["{\sigma^\dagger}"', between={0.2}{0.8}, Rightarrow, nfold, from=3, to=2]
	\arrow["\kappa"'{pos=0.2}, color={rgb,255:red,214;green,92;blue,92}, between={0.1}{0.4}, Rightarrow, nfold, from=3, to=7]
	\arrow["{\Epsilon_h}", between={0.2}{0.8}, Rightarrow, nfold, from=4, to=5]
	\arrow["{\tau^\dagger}"', between={0.2}{0.8}, Rightarrow, nfold, from=7, to=6]
\end{tikzcd}\]
Now, for the 1-cells $\zeta_1, \zeta_2$ and for the prism bases $\sigma_g, \sigma_h$ take the following from the above diagram \equref{theta-f}: 
$$\zeta_1:=\G(\hat f)\circ \theta_{m(\crta{a_i})}\circ\xi^\F \qquad\text{and}\qquad 
\zeta_2:=\G(\hat f)\circ \xi^\G\circ m(\crta{\theta_{a_i}})$$ 
(from the edges of the inner square above), and the concatenations 
$$\sigma_g:= \F(\Epsilon_f)\diamond \sigma_\F\diamond \theta_{\hat f}^{-1}\qquad\text{and}\qquad 
\sigma_h:= \Epsilon\diamond \sigma_\G\diamond \G(\Epsilon_f).$$
Now we know a 2-cell $\kappa$ exists, which concatenated with the bases $\sigma$'s gives $\theta_f$. 
To prove the claim it only remains to show that $\kappa$ equals $[\Theta_n^{-1}(\crta{a_i})\vert \Id_{\G(\hat f)}]$. 

Let us rewrite the above prism by partitioning its bases into their constituting pieces:
\[\begin{tikzcd}
	{\crta{\F(a_i)}} && {m(\crta{ \F(a_i)})} && \\
	&& {\F(m(\crta a))} & {\G(m(\crta a))} \\
	& {\F(b)} &&& {\G(b)} \\
	\\
	{\crta{\F(a_i)}} && {m(\crta{ \F(a_i)})} \\
	& {\crta{\G(a_i)}} & {m(\crta{\G(a_i)})} \\
	&&& {\G(m(\crta a))} \\
	&&&& {\G(b)}
	\arrow["{j_{\F(a_i)}}", from=1-1, to=1-3]
	\arrow[""{name=0, anchor=center, inner sep=0}, "{\F(j_{a_i})}"'{pos=0.4}, color={rgb,255:red,92;green,92;blue,214}, from=1-1, to=2-3]
	\arrow[""{name=1, anchor=center, inner sep=0}, "{\F(f)}"', color={rgb,255:red,92;green,214;blue,92}, from=1-1, to=3-2]
	\arrow[color={rgb,255:red,92;green,214;blue,92}, equals, nfold, from=1-1, to=5-1]
	\arrow["{\xi^\F}", from=1-3, to=2-3]
	\arrow[""{name=2, anchor=center, inner sep=0}, "{\theta_{m(\crta a)}}"{pos=0.4}, color={rgb,255:red,92;green,92;blue,214}, from=2-3, to=2-4]
	\arrow["{\F(\hat f)}", from=2-3, to=3-2]
	\arrow[""{name=3, anchor=center, inner sep=0}, "{\G(\hat f)}"{pos=0.7}, from=2-4, to=3-5]
	\arrow[draw={rgb,255:red,92;green,92;blue,214}, equals, nfold, dashed, from=2-4, to=7-4]
	\arrow[""{name=4, anchor=center, inner sep=0}, "{\theta_b}"', color={rgb,255:red,92;green,214;blue,92}, from=3-2, to=3-5]
	\arrow[draw={rgb,255:red,92;green,214;blue,92}, equals, nfold, from=3-5, to=8-5]
	\arrow["{j_{\F(a_i)}}"{pos=0.7}, dashed, from=5-1, to=5-3]
	\arrow[""{name=5, anchor=center, inner sep=0}, "{\crta{\theta_{a_i}}}"'{pos=0.2}, color={rgb,255:red,92;green,92;blue,214}, from=5-1, to=6-2]
	\arrow["{m(\crta{\theta_{a_i}})}"', from=5-3, to=6-3]
	\arrow["{j_{\G(a_i)}}", from=6-2, to=6-3]
	\arrow[""{name=6, anchor=center, inner sep=0}, "{\G(j_{a_i})}"', color={rgb,255:red,92;green,92;blue,214}, curve={height=6pt}, from=6-2, to=7-4]
	\arrow[""{name=7, anchor=center, inner sep=0}, "{\G(f)}"', color={rgb,255:red,92;green,214;blue,92}, curve={height=30pt}, from=6-2, to=8-5]
	\arrow[""{name=8, anchor=center, inner sep=0}, "{\xi^\G}", from=6-3, to=7-4]
	\arrow[""{name=9, anchor=center, inner sep=0}, "{\G(\hat f)}"{pos=0.7}, from=7-4, to=8-5]
	\arrow["{\sigma^\F}"{description, pos=0.3}, between={0}{0.8}, Rightarrow, nfold, from=1-3, to=0]
	\arrow["{\F(\Epsilon_f)}"{description, pos=0.4}, between={0}{0.8}, Rightarrow, nfold, from=2-3, to=1]
	\arrow["{\theta_{j_{a_i}}}"'{pos=0.3}, color={rgb,255:red,92;green,92;blue,214}, between={0.2}{0.5}, Rightarrow, nfold, from=2-3, to=6]
	\arrow["{(\Theta_n(\crta a))^{-1}}", color={rgb,255:red,214;green,153;blue,92}, between={0.2}{0.8}, Rightarrow, nfold, from=2, to=8]
	\arrow["{(\theta_{\hat f})^{-1}}"{description, pos=0.4}, between={0}{0.8}, Rightarrow, nfold, from=2-4, to=4]
	\arrow["{\Id_{\G(\hat f)}}"{description}, between={0.2}{0.8}, Rightarrow, nfold, from=3, to=9]
	\arrow["\Epsilon"{description}, between={0}{0.8}, Rightarrow, nfold, from=5-3, to=5]
	\arrow["{\sigma^\G}", between={0}{0.8}, Rightarrow, nfold, from=6-3, to=6]
	\arrow["{\G(\Epsilon_f)}"', between={0}{0.8}, Rightarrow, nfold, from=7-4, to=7]
\end{tikzcd}\]
The front vertical surface (with five green edges and one blue) is the 2-cell $\theta_f$. 
The blue surface $\theta_{m(\crta a)}\circ \F(j_{a_i})\Rightarrow \G(j_{a_i})\circ \crta{\theta_{a_i}}$ presents the 2-cell component $\theta_{j_{a_i}}$ that 
we insert; it parts the whole prism in two subprisms: the one ``behind'' it is the prism defining the 2-cell $\Theta_n(\crta a)$ marked in orange, due to 
\equref{ind Theta_n}. The front 
subprism commutes by the naturality of the transformation $\theta$ with respect to the 2-cell $\Epsilon_f$, as it is equivalent to the equation 
$$\frac{[\F(\Epsilon_f)\vert \Id_{\theta_b}]}{\theta_f}=\threefrac{[\Id_{\F(j_{a_i})}\vert\theta_{\hat f}]}{[\theta_{j_{a_i}}\vert\Id_{\G(\hat f)}]}
{[\Id_{\crta{\theta_{a_i}}}\vert\G(\Epsilon_f)]}$$
(up to the structure 2-cell $\F^2$). So, the whole (outer) prism commutes whereby $[\Theta_n^{-1}(\crta{a_i})\vert \Id_{\G(\hat f)}]$ is 
at the place of $\kappa$ (the concatenation on the right in the vertical back plane), so by uniqueness of such $\kappa$ they are equal. 
\qed\end{proof}

\begin{prop} \prlabel{modif}
Given pseudonatural transformations $(\theta, \Theta), (\theta', \Theta'):(\F, \xi^\F, \delta^\F)\Rightarrow (\G, \xi^\G,\delta^\G):\M\to\N$ of 
 pseudofunctors of (bi-representable) bi-multicategories as in \prref{transf}, the 2-cell components  of a modification 
$\Xi: \theta\Rrightarrow\theta': \F\Rightarrow \G$ of transformations of pseudofunctors of bi-multicategories 
obey 
\begin{equation} \eqlabel{modif-rep}
\scalebox{0.86}{
\bfig
\putmorphism(-150,500)(1,0)[\F(m(\crta a))`\G(m(\crta a))`\theta_{m(\crta a)}]{600}1a

 \putmorphism(-850,50)(1,0)[m(\crta{\F(a_i)}) `\phantom{F(A)}`\xi^\F_n(\crta a)]{640}1a
 \putmorphism(-150,50)(1,0)[\F(m(\crta a))`\G(m(\crta a))`\theta'_{m(\crta a)}]{600}1a

\putmorphism(-180,500)(0,-1)[\phantom{Y_2}``=]{450}1l
\putmorphism(560,500)(0,-1)[\phantom{Y_2}``=]{450}1r
\put(40,290){\fbox{$\Xi_{m(\crta a)}$}}

 \putmorphism(-850,-400)(1,0)[m(\crta{\F(a_i)}) `\phantom{A'\ot B'}` m(\crta{\theta'_{a_i}})]{680}1a
\putmorphism(-150,-400)(1,0)[m(\crta{\G(a_i)})`\G(m(\crta a)) `\xi_n^\G(\crta a)]{640}1a

\putmorphism(-780,50)(0,-1)[\phantom{Y_2}``=]{450}1l
\putmorphism(560,50)(0,-1)[\phantom{Y_3}``=]{450}1r
\put(-300,-200){\fbox{$\Theta_n'^{-1}(\crta a)$}}

\efig}
\quad=\quad
\scalebox{0.86}{
\bfig
 \putmorphism(100,500)(1,0)[m(\crta{\F(a_i)})`\F(m(\crta a)) `\xi^\F_n(\crta a)]{680}1a
\putmorphism(750,500)(1,0)[\phantom{\F(m(\crta a))}`\G(m(\crta a))`\theta_{m(\crta a)}]{600}1a
 \putmorphism(100,50)(1,0)[m(\crta{\F(a_i)})`m(\crta{\G(a_i)}) `m(\crta{\theta_{a_i}})]{680}1a
 \putmorphism(790,50)(1,0)[\phantom{\F(m(\crta a))}`\G(m(\crta a))`\xi^\G_n(\crta a)]{600}1a

\putmorphism(50,500)(0,-1)[\phantom{Y_2}``=]{450}1r
\putmorphism(1300,500)(0,-1)[\phantom{Y_2}``=]{450}1r
\put(610,260){\fbox{$\Theta_n^{-1}(\crta a)$}}
\put(230,-160){\fbox{$m(\crta{\Xi_i})$}}

 \putmorphism(90,-400)(1,0)[m(\crta{\F(a_i)})` m(\crta{\G(a_i)}) `m(\crta{\theta'_{a_i}})]{640}1a
\putmorphism(50,50)(0,-1)[\phantom{Y_2}``=]{450}1l
\putmorphism(720,50)(0,-1)[\phantom{Y_3}``=]{450}1r
\efig}
\end{equation}
and
\[\begin{tikzcd}
	I & I & \\
	{\mathcal F(I)} & {\mathcal G(I) } \\
	{\mathcal F(I)} & {\mathcal G(I)} & {} \\
	\\
	& {}
	\arrow[""{name=0, anchor=center, inner sep=0}, "{=}", from=1-1, to=1-2]
	\arrow["{\xi_0^{\mathcal F}}"', from=1-1, to=2-1]
	\arrow["{\xi_0^{\mathcal G}}", from=1-2, to=2-2]
	\arrow[""{name=1, anchor=center, inner sep=0}, "{\theta_I}"', from=2-1, to=2-2]
	\arrow["{=}"', from=2-1, to=3-1]
	\arrow["{=}", from=2-2, to=3-2]
	\arrow[""{name=2, anchor=center, inner sep=0}, "{\theta_I'}"', from=3-1, to=3-2]
	\arrow["{\Theta_0}"{description}, draw=none, from=0, to=1]
	\arrow["{\Xi_I}"{description}, draw=none, from=1, to=2]
\end{tikzcd} \hspace{-1,2cm} 
=
\begin{tikzcd}
	I & I & \\
	{\mathcal F(I)} & {\mathcal G(I). } \\
	&& {} \\
	\\
	& {}
	\arrow[""{name=0, anchor=center, inner sep=0}, "{=}", from=1-1, to=1-2]
	\arrow["{\xi_0^{\mathcal F}}"', from=1-1, to=2-1]
	\arrow["{\xi_0^{\mathcal G}}", from=1-2, to=2-2]
	\arrow[""{name=1, anchor=center, inner sep=0}, "{\theta_I'}"', from=2-1, to=2-2]
	\arrow["{\Theta_0'}"{description}, draw=none, from=0, to=1]
\end{tikzcd}\]
\end{prop}

\vspace{-1,2cm}

From the above results it is evident that given bi-representable bi-multicategories $\M$ and $\N$, the 
full sub-bicategory of $\BiMlt(\M,\N)$ of pseudofunctors of {\em bi-representable} bi-multicategories that one would naturally define 
coincides with the whole $\BiMlt(\M,\N)$. 





\subsection{Enrichment of monoidal bicategories} \sslabel{mon bicat}

We denote by $\MBic$ the (necessarily 3-dimensional) category 
whose objects are monoidal bicategories and that for two such objects $\B$ and $\B'$ is equipped with the hom-bicategory $\MBic(\B,\B')$ 
consisting of {\em lax monoidal} pseudofunctors $\B\to\B'$, monoidal pseudonatural transformations, and their modifications.

\subsubsection{The functor bicategory of monoidal bicategories} \ssslabel{fun-bicat-mon}

The constituting cells of the bicategory $\MBic(\B,\B')$ were defined in \cite[Section 3]{GPS}, \cite[Section 1]{DS} and can be found also in \cite[Section 2.2]{MV}, \cite[Section 2.3]{SP}. We will use the definitions from the latter reference: their definition of pseudofunctors coincides with the former definitions, while transformations and modifications are defined as a variation of the corresponding definitions in the former references. 

For reader's convenience we write down here the 2-cell components of the invertible modifications defining the respective transforms and list the labels of axioms (from \cite[Section 2.3]{SP}) they need to obey: 
\begin{itemize} 
\item a lax monoidal pseudofunctor $F$: 
\begin{equation} \eqlabel{omega mon ps fun}
\scalebox{0.82}{
\bfig
\putmorphism(-390,450)(0,-1)[\phantom{Y_2}` `=]{450}1r 

\putmorphism(-390,450)(1,0)[(F(A)F(B))F(C)` F(AB)F(C)`F^2_{A,B}F(C)]{1170}1a 
\putmorphism(840,450)(1,0)[\phantom{A''\ot B'}`F((AB)C) ` F^2_{AB,C}]{1170}1a 
 \putmorphism(2040,450)(1,0)[\phantom{A''\ot B'}` F(A(BC)) `F(\alpha_{A,B,C})]{1160}1a 

 \putmorphism(-390,0)(1,0)[(F(A)F(B))F(C)` F(A)(F(B)F(C)) `\alpha_{F(A),F(B),F(C)}]{1350}1b 
 \putmorphism(1120,0)(1,0)[\phantom{A''\ot B'}` F(A)F(BC) `F(A)F^2_{B,C}]{1040}1b 
 \putmorphism(2230,0)(1,0)[\phantom{A''\ot B'}` F(A(BC)) `F^2_{A,BC}]{1020}1b 

\putmorphism(3200,450)(0,-1)[\phantom{Y_2}``=]{450}1r 

\put(1200,230){\fbox{$\omega_{A,B,C}$}}
\efig}
\end{equation}

\begin{equation} \eqlabel{gamma mon ps fun}
\scalebox{0.82}{
\bfig
\putmorphism(-390,450)(0,-1)[\phantom{Y_2}` `=]{450}1r 

\putmorphism(-390,450)(1,0)[IF(A)` F(I)F(A) `F^0F(A)]{740}1a 
\putmorphism(370,450)(1,0)[\phantom{A''\ot B'}` F(IA) ` F^2_{I,A} ]{700}1a 
 \putmorphism(1020,450)(1,0)[\phantom{A''\ot B'}` F(A) `F(\lambda_A) ]{760}1a 

 \putmorphism(-390,0)(1,0)[IF(A)` F(A); `\lambda_{F(A)}]{2230}1b 

\putmorphism(1820,450)(0,-1)[\phantom{Y_2}``=]{450}1r 

\put(700,210){\fbox{$\gamma_A$}}
\efig}
\end{equation}

\begin{equation} \eqlabel{delta mon ps fun}
\scalebox{0.82}{
\bfig
\putmorphism(-390,450)(0,-1)[\phantom{Y_2}` `=]{450}1r 

\putmorphism(-390,450)(1,0)[F(A)I` F(A)F(I) `F(A)F^0]{740}1a 
\putmorphism(370,450)(1,0)[\phantom{A''\ot B'}` F(AI) ` F^2_{A,I} ]{700}1a 
 \putmorphism(1020,450)(1,0)[\phantom{A''\ot B'}` F(A) `F(\rho_A) ]{760}1a 

 \putmorphism(-390,0)(1,0)[F(A)I` F(A) `\rho_{F(A)}]{2230}1b 

\putmorphism(1820,450)(0,-1)[\phantom{Y_2}``=]{450}1r 

\put(700,210){\fbox{$\delta_A$}}
\efig}
\end{equation}
with the axioms $(HTA1)$ and $(HTA2)$, 
\item a lax monoidal pseudonatural transformation $\beta:F\Rightarrow G$:
\begin{equation} \eqlabel{beta's}
\scalebox{0.86}{
\bfig
 \putmorphism(-570,500)(1,0)[F(A)\ot F(B)`G(A)\ot G(B)`\beta(A)\ot\beta(B)]{1140}1a
 \putmorphism(740,500)(1,0)[\phantom{F(f)}`G(A\ot B) `G^2_{A,B}]{700}1a
 \putmorphism(-570,120)(1,0)[F(A)\ot F(B)`F(A\ot B)`F^2_{A,B}]{940}1a
 \putmorphism(460,120)(1,0)[\phantom{G(B)}`G(A\ot B) `\beta(A\ot B)]{960}1a
\putmorphism(-580,500)(0,-1)[\phantom{Y_2}``=]{380}1r
\putmorphism(1440,500)(0,-1)[\phantom{Y_2}``=]{380}1r
\put(280,310){\fbox{$\beta^2_{A,B}$}}
\efig}
\qquad\quad
\scalebox{0.86}{
\bfig
 \putmorphism(0,500)(1,0)[I`I`=]{540}1a
 \putmorphism(480,500)(1,0)[\phantom{F(f)}`I `G^0]{600}1a
 \putmorphism(0,120)(1,0)[I `F(I)`F^0]{500}1a
 \putmorphism(460,120)(1,0)[\phantom{G(B)}`G(I) `\beta_I]{620}1a
\putmorphism(0,500)(0,-1)[\phantom{Y_2}``=]{380}1r
\putmorphism(1080,500)(0,-1)[\phantom{Y_2}``=]{380}1r
\put(420,310){\fbox{$\beta^0$}}
\efig}
\end{equation} 
with the axioms $(MBTA1)-(MBTA3)$, 
\item and a monoidal modification $\Xi:\beta\Rrightarrow\beta'$:
$$
\scalebox{0.86}{
\bfig
 \putmorphism(-570,500)(1,0)[F(A\ot B)`G(A\ot B)`\beta(A\ot B)]{940}1a
 \putmorphism(-570,120)(1,0)[F(A\ot B)`G(A\ot B)`\beta'(A\ot B)]{940}1a
\putmorphism(-580,500)(0,-1)[\phantom{Y_2}``=]{380}1l
\putmorphism(400,500)(0,-1)[\phantom{Y_2}``=]{380}1r
\put(-220,330){\fbox{$\Xi_{A,B}$}}
\efig}
$$
satisfying the axioms $(BMBM1)-(BMBM2)$. 
\end{itemize}


\bigskip

Let us highlight some of the identities that a monoidal transformation $\beta$ obeys. \\
{\em 4-ary compatibility with associativity} 
\begin{equation} \eqlabel{4-ary compat}
\begin{tikzcd}
	&& \bullet & \bullet & \bullet \\
	\bullet && \bullet & \bullet & \bullet \\
	\bullet && \bullet & \bullet & \bullet \\
	&& \bullet & \bullet & \bullet
	\arrow[""{name=0, anchor=center, inner sep=0}, "{a_{F_{AB},F_C,F_D}}", from=1-3, to=1-4]
	\arrow["{=}"', from=1-3, to=2-3]
	\arrow["{\beta_{AB}(\beta_C\beta_D)}", from=1-4, to=1-5]
	\arrow["{=}", from=1-5, to=2-5]
	\arrow[""{name=1, anchor=center, inner sep=0}, "{(\xi^F_{A,B}1_{FC})1_{FD}}", from=2-1, to=2-3]
	\arrow["{=}"', from=2-1, to=3-1]
	\arrow["{(\beta_{AB}\beta_C)\beta_D}", from=2-3, to=2-4]
	\arrow[""{name=2, anchor=center, inner sep=0}, "{a_{G_{AB},G_C,G_D}}"', from=2-4, to=2-5]
	\arrow["{=}", from=2-4, to=3-4]
	\arrow["{((\beta_A\beta_B)\beta_C)\beta_D}"', from=3-1, to=3-3]
	\arrow[""{name=3, anchor=center, inner sep=0}, "{(\xi^G_{A,B}1_{GC})1_{GD}}", from=3-3, to=3-4]
	\arrow["{=}"', from=3-3, to=4-3]
	\arrow["{a_{G_{AB},G_C,G_D}}"', from=3-4, to=3-5]
	\arrow["{=}", from=3-5, to=4-5]
	\arrow["{a_{F_{AB},F_C,F_D}}"', from=4-3, to=4-4]
	\arrow[""{name=4, anchor=center, inner sep=0}, "{\xi^G_{A,B}(1_C1_D)}"', from=4-4, to=4-5]
	\arrow["{a_{\beta_{AB},\beta_C,\beta_D}^{-1}}"{description, pos=0.4}, draw=none, from=0, to=2]
	\arrow["{((\beta^2_{A,B})^{-1}1_C)1_D}"{description, pos=0.3}, draw=none, from=1, to=3]
	\arrow["{a_{\xi^G_{A,B},1,1}}"{description}, draw=none, from=3, to=4]
\end{tikzcd}
=\quad
\begin{tikzcd}
	\bullet & \bullet & \bullet && \\
	\bullet & \bullet & \bullet && \bullet \\
	\bullet & \bullet && \bullet & \bullet \\
	\bullet & \bullet && \bullet
	\arrow[""{name=0, anchor=center, inner sep=0}, "{(\xi^F_{A,B}1_C)1_D}", from=1-1, to=1-2]
	\arrow["{=}"', from=1-1, to=2-1]
	\arrow["{a_{F_{AB},F_C,F_D}}", from=1-2, to=1-3]
	\arrow["{=}", from=1-3, to=2-3]
	\arrow["{a_{F_A F_B,F_C,F_D}}"', from=2-1, to=2-2]
	\arrow[""{name=1, anchor=center, inner sep=0}, "{\xi^F_{A,B}(1_C1_D)}", from=2-2, to=2-3]
	\arrow["{=}"', from=2-2, to=3-2]
	\arrow["{\beta_{AB}(\beta_C\beta_D)}", from=2-3, to=2-5]
	\arrow["{=}", from=2-5, to=3-5]
	\arrow[""{name=2, anchor=center, inner sep=0}, "{a_{F_A F_B,F_C,F_D}}", from=3-1, to=3-2]
	\arrow["{=}"', from=3-1, to=4-1]
	\arrow["{(\beta_A\beta_B)(\beta_C\beta_D)}"', from=3-2, to=3-4]
	\arrow[""{name=3, anchor=center, inner sep=0}, "{\xi^G_{A,B}(1_C1_D)}"', from=3-4, to=3-5]
	\arrow["{=}", from=3-4, to=4-4]
	\arrow["{((\beta_A\beta_B)\beta_C)\beta_D}"', from=4-1, to=4-2]
	\arrow[""{name=4, anchor=center, inner sep=0}, "{a_{G_A G_B,G_C,G_D}}"', from=4-2, to=4-4]
	\arrow["{a_{\xi^F,1,1}}"{description, pos=0.4}, draw=none, from=0, to=1]
	\arrow["{(\beta^2_{A,B})^{-1}(1_C1_D)}"{description}, draw=none, from=1, to=3]
	\arrow["{a_{\beta_A\beta_B,\beta_C,\beta_D}^{-1}}"{description}, draw=none, from=2, to=4]
\end{tikzcd}
\end{equation}
and similarly for $\beta^2$ and $\xi$ applied in the second and third variable. 
They are equivalent to the naturality of the pseudonatural transformation $a$ with respect to the 2-cell components of $\beta^2$ 
(the compositions of the two 2-cell components of $a$ present the single 2-cell components of $a$ at the compositions of the corresponding 1-cells). 

Similar identities hold for the following pairs of triples of 2-cell components: 
$$L(\theta_{a^*}^{-1}(\crta b), \,\,F(a^*_{\crta f}), \,\,\theta_{m_n(\crta f)})=
R(\theta_{m_{n-1}(\crta f)}, \,\,G(a^*_{\crta f}), \,\,\theta_{a^*}^{-1}(\crta a)),\vspace{-0,5cm}$$ 
\begin{equation} \eqlabel{triples}
L(a_{\crta{\theta_{b_i}}},\,\,  a_{\crta{F(f)}}, \,\, m_{n-1}(\theta_{f_1},...,m_2(\theta_{f_i},\theta_{f_{i+1}}),...,\theta_{f_n}))=
R(m_n(\crta{\theta_{f_i}}), \,\,a_{\crta{G(f)}}, \,\, a_{\crta{\theta_{a_i}}}).
\end{equation}
Here $L(\bullet,\bullet,\bullet)$ and $R(\bullet,\bullet,\bullet)$ denote the left- and right-hand side triples of 2-cells as in \equref{4-ary compat}. 
  
And finally:  
\begin{equation} \eqlabel{pi-beta}
\begin{tikzcd}
	& {} & \bullet & \bullet & \bullet \\
	& \bullet & \bullet & \bullet & \bullet \\
	\bullet & \bullet & \bullet & \bullet \\
	\bullet & \bullet & \bullet & \bullet & \bullet \\
	& \bullet & \bullet & {} & \bullet
	\arrow[""{name=0, anchor=center, inner sep=0}, "{F(Aa_{B,C,D})}", from=1-3, to=1-4]
	\arrow["{=}"', from=1-3, to=2-3]
	\arrow["{\beta_{A(B(CD))}}", from=1-4, to=1-5]
	\arrow["{=}", from=1-5, to=2-5]
	\arrow[""{name=1, anchor=center, inner sep=0}, "{F(a_{A,BC,D})}", from=2-2, to=2-3]
	\arrow["{=}"', from=2-2, to=3-2]
	\arrow["{\beta_{A((BC)D)}}", from=2-3, to=2-4]
	\arrow[""{name=2, anchor=center, inner sep=0}, "{G(Aa_{B,C,D})}"', from=2-4, to=2-5]
	\arrow["{=}", from=2-4, to=3-4]
	\arrow[""{name=3, anchor=center, inner sep=0}, "{F(a_{A,B,C}D)}", from=3-1, to=3-2]
	\arrow["{=}"', from=3-1, to=4-1]
	\arrow["{\beta_{(A(BC))D}}", from=3-2, to=3-3]
	\arrow[""{name=4, anchor=center, inner sep=0}, "{G(a_{A,BC,D})}", from=3-3, to=3-4]
	\arrow["{=}", from=3-3, to=4-3]
	\arrow["{\beta_{((AB)C)D}}", from=4-1, to=4-2]
	\arrow[""{name=5, anchor=center, inner sep=0}, "{G(a_{A,B,C}D)}", from=4-2, to=4-3]
	\arrow["{=}"', from=4-2, to=5-2]
	\arrow[""{name=6, anchor=center, inner sep=0}, "{G(a_{A,BC,D})}", from=4-3, to=4-4]
	\arrow["{G(a_{A,BC,D})}", from=4-4, to=4-5]
	\arrow["{=}", from=4-5, to=5-5]
	\arrow["{G(a_{AB,C,D})}"', from=5-2, to=5-3]
	\arrow["{G(a_{A,B,CD})}"', from=5-3, to=5-5]
	\arrow["{\beta_{Aa_{B,C,D}}}"{description}, draw=none, from=0, to=2]
	\arrow["{\beta_{a_{A,BC,D}}}"{description}, draw=none, from=1, to=4]
	\arrow["{\beta_{a_{A,B,C}D}}"{description}, draw=none, from=3, to=5]
	\arrow["{G(\pi)}"{description}, draw=none, from=6, to=5-4]
\end{tikzcd}
=\quad
\begin{tikzcd}
	\bullet & \bullet & \bullet & \bullet & \\
	\bullet & \bullet && \bullet & \bullet \\
	\bullet & \bullet & \bullet && \bullet \\
	\bullet & \bullet & \bullet \\
	&&& {}
	\arrow["{F(a_{A,B,C}D)}", from=1-1, to=1-2]
	\arrow["{=}"', from=1-1, to=2-1]
	\arrow[""{name=0, anchor=center, inner sep=0}, "{F(a_{A,BC,D})}", from=1-2, to=1-3]
	\arrow["{F(Aa_{B,C,D})}", from=1-3, to=1-4]
	\arrow["{=}", from=1-4, to=2-4]
	\arrow["{F(a_{AB,C,D})}"', from=2-1, to=2-2]
	\arrow[""{name=1, anchor=center, inner sep=0}, "{F(a_{A,B,CD})}", from=2-2, to=2-4]
	\arrow["{=}"', from=2-2, to=3-2]
	\arrow["{\beta_{A(B(CD))}}", from=2-4, to=2-5]
	\arrow["{=}", from=2-5, to=3-5]
	\arrow[""{name=2, anchor=center, inner sep=0}, "{F(a_{AB,C,D})}", from=3-1, to=3-2]
	\arrow["{=}"', from=3-1, to=4-1]
	\arrow["{\beta_{(AB)(CD)}}", from=3-2, to=3-3]
	\arrow[""{name=3, anchor=center, inner sep=0}, "{G(a_{A,B,CD})}", from=3-3, to=3-5]
	\arrow["{=}", from=3-3, to=4-3]
	\arrow["{\beta_{((AB)C)D}}"', from=4-1, to=4-2]
	\arrow[""{name=4, anchor=center, inner sep=0}, "{G(a_{A,B,C}D)}"', from=4-2, to=4-3]
	\arrow["{F(\pi)}"{description}, draw=none, from=0, to=1]
	\arrow["{\beta_{a_{A,B,CD}}}"{description}, draw=none, from=1, to=3]
	\arrow["{\beta_{a_{AB,C,D}}}"{description}, draw=none, from=2, to=4]
\end{tikzcd}
\end{equation}
The two compositions of the 2-cell components of $\beta$ appearing in the above diagrams present the single 2-cell component of $\beta$ at the composition of the associativity 1-cells, then the identity holds by naturality of the transformation $\beta$ with respect to the 2-cell $\pi_{A,B,C,D}$.

\subsubsection{$\mathsf{Bicat}$-enrichment of monoidal bicategories}

To prove that $\MBic$ is a $\mathsf{Bicat}$-category 
we first argue that monoidal bicategories make a left closed 
multicategory, so there is a self-enrichment by \prref{Yau}. This yields an associative and unital binary multimorphism 
$\MBic(\B',\B''),\MBic(\B,\B')\to\MBic(\B,\B'')$ of monoidal bicategories, so in particular, an associative and unital binary multimorphism in $\mathsf{Bicat}$, as desired. 

\smallskip

To prove that monoidal bicategories make a left closed multicategory, we consider the following steps: 
\begin{itemize}
\item for monoidal bicategories $\B$ and $\B'$ it is easily shown that $\MBic(\B,\B')$ makes a full sub-bicategory of $\Bicat(\B,\B')$; 
then 
a monoidal structure of $\MBic(\B,\B')$ is induced pointwise (for two monoidal pseudofunctors $F,F':\B\to\B'$ it is 
$(F\ot F')(B):=F(B)\ot F'(B')$, and similarly for higher cells in $\MBic(\B,\B')$; 
the interchange between the monoidal structure and the vertical composition of pseudonatural transformations is induced by the interchange of the monoidal structure of $\B'$ and its 1-cells, and so on);
\item the closedness isomorphisms $\M_{n+1}(\A_1,..., \A_n, \B; \C)\iso \M_n( \A_1,..., \A_n; \M_1(\B;\C))$ are proved by using the 
corresponding isomorphisms from the known closedness of the multicategory $\mathsf{Bicat}$ and adding the check of monoidality of the 
cells at all levels; hereby an $n$-ary multifunctor is monoidal if all its $n$-unary constituting pseudofunctors are monoidal (it is easily seen that an $(n+1)$-ary multifunctor on the left is monoidal if and only if an $n$-ary multifunctor on the right is monoidal; a pseudofunctor 
$F:\A_i\to\M_1(\B;\C)$ is monoidal by $F(A\ot A')(B)=F(A)(B)\ot F(A')(B)$); 
\item that substitution 
$\Pi_{j=1}^{m_1}\A_{1j}\times...\times\Pi_{j=1}^{m_n}\A_{nj}\stackrel{F_1\times...\times F_n}{\longrightarrow}\B_1\times...\times\B_n
\stackrel{G}{\to}\C$
for monoidal pseudofunctors $F_1,...,F_n,G$ is a monoidal pseudofunctor is straightforwardly proved;
\item that evaluation $ev: \M_1(\A;\B),\A\to\B$ is a monoidal binary pseudofunctor is also straghtforwardly proved.
\end{itemize}

\subsection{The two hom-bicategories are biequivalent} \sslabel{equiv homs}

In \ssref{fun-bicat} we showed that for bi-representable bi-multicategories $\M$ and $\N$ the bicategory of pseudofunctors between them 
coincides with the bicategory $\BiMlt(\M,\N)$ of pseudofunctors of their underlying bi-multicategories. To stress that we are dealing with 
bi-representable bi-multicategories and their bicategory of pseudofunctors, in this subsection we are going to use the notation $\RepBM(\M,\N)$ 
for the latter bicategory. 

We are now going to prove that there is a biequivalence pseudofunctor 
$$\U:\RepBM(\M,\N)\to\MBic(\ul{\M},\ul{\N}).$$


Given a pseudofunctor $(\F, \xi, \delta):\M\to\N$ of bi-representable bi-multicategories, from \prref{psf birep} for $n=3$ it is clear that 
$\F$ restricted to the underlying monoidal bicategory $\ul{\M}$ of $\M$ equipped with the pseudonatural transformations $\xi_2, \xi_0$ and 
invertible modifications $\delta_3, \delta^L, \delta^R$ determines a lax monoidal pseudofunctor 
of monoidal bicategories (see \equref{omega mon ps fun} - \equref{delta mon ps fun}).  
Observe that $m'(x,y,z)$ is given by $m^R(x,y,z)$ from \rmref{j and j'} part a). The two coherence identities from \prref{psf birep} 
correspond to the axioms $(HTA1)$ and $(HTA2)$ from \cite[Definition 3.1]{GPS} (they agree with the two axioms from \cite[Definition 2.5]{SP}). 

Given a pseudonatural transformation $(\theta, (\Theta_n)_{n\in\Nn_0}): (\F, \xi^\F, \delta^\F)\Rightarrow (\G, \xi^\G,\delta^\G)$ of pseudofunctors of 
bi-multicategories. 
The invertible modification $\Theta_2$ and the 2-cell $\Theta_0$ correspond to the modification $\beta^2$ and the 2-cell $\beta^0$ from \equref{beta's}, respectively, 
the axiom from point 2) of \prref{transf} for $n=3$ corresponds to the axiom MBTA1 from \cite[Definition 2.7]{SP} 
(consider the modifications $\Theta_3$ and $\Theta_3'$ as $\Theta_3^L$ and $\Theta_3^R$, respectively, that is to say, 
in the form of their factorizations through $\Theta_2$ and $m_2(\Theta,1)$, and through $\Theta_2$ and $m_2(1,\Theta)$, respectively), 
and the axioms from point 3) of \prref{transf} correspond to the axioms MBTA2 and  MBTA3 of \cite[Definition 2.7]{SP}, 
so that we obtain a monoidal pseudonatural transformation of monoidal pseudofunctors of monoidal bicategories. 

Finally, the two axioms from \prref{modif} of a modification $\Xi:(\theta, \Theta)\Rrightarrow (\theta', \Theta')$ between pseudonatural transformations 
for sequences of length $2$ coincide with the two axioms from \cite[Definition 2.8]{SP}, and we indeed have a monoidal modification on the side of monoidal bicategories, as desired.

\medskip

We now construct a quasi-inverse for the underlying 2-functor $\U$. 

\medskip

On the level of objects, take a lax monoidal pseudofunctor $F:\ul{\M}\to\ul{\N}$ with the pseudofunctor structure $F^2, F^0$, lax monoidal pseudofunctor data 
$\xi^2, \xi^0$ and invertible modifications $\omega, \gamma, \delta$ as in \equref{omega mon ps fun} - \equref{delta mon ps fun}. By iterating $\xi^2$ 
construct $\xi^L_{A_1,..,A_n}: \textcircled{L} F(A_i)\to F(\textcircled{L} A_i)$ and define 
$$\F_{A_1, ..., A_n;B}: \dul{\M}_n(A_1, ..., A_n;B) \to \dul{\N}_n(F(A_1), ..., F(A_n); F(B))$$ 
for $n> 1$ as the composition functor 
$$\ul{\M}(\underset{i}{\textcircled{L}}A_i,B) \xrightarrow{F_{\underset{i}{\textcircled{L}}A_i,B}} \ul{\N}(F(\underset{i}{\textcircled{L}}A_i),F(B))  
 \xrightarrow{\ul{\N}(\xi^L_{A_1,..,A_n},1)} \ul{\N}(\underset{i}{\textcircled{L}}F(A_i),F(B)) $$ 
and as $F_{A,B}$ for $n=1$. Thus for $f:\coprod_i A_i\to B$ with its corresponding unary map $\hat f:\underset{i}{\textcircled{L}}A_i\to B$, it is 
$$\F(f)=F(\hat f)\circ\xi^L(\crta{A_i})\circ j: \coprod_i F(A_i) \to F(B).$$ 
The desired pseudofunctor structure $\F^2, \F^0$ we define by setting the latter to be $F^0$. Recalling from \equref{def subst} how we defined substitution, we 
define the component 2-cells of $\F^2$ at multimaps $f_1,...,f_n, g$ mapping
$$\F(g)\comp \langle\F(f_1),...,\F(f_n)\rangle \Rightarrow\F(g\circ \langle f_1,...,f_n\rangle)$$
via the corresponding unary maps $\hat f_i:\underset{ij}{\textcircled{L}}A_{ij}\to B_i, i=1,...,n$ and 
$\hat g:\underset{i}{\textcircled{L}}B_i\to C$ and the compositions (in the unary part of $\N$)
\begin{multline} \label{F2}
\big(F(\hat g)\comp\xi(\crta{B_i})\big)\circ \big(\underset{i}{\textcircled{L}} \big(F(\hat f_i)\comp\xi_j(\crta{A_{ij}})\big) \big) 
\stackrel{1\comp int}{\Longrightarrow }  
\big(F(\hat g)\comp\xi(\crta{B_i})\big)\circ \big(\underset{i}{\textcircled{L}} F(\hat f_i)\big) \comp\big(\underset{i}{\textcircled{L}}\xi_j(\crta{A_{ij}})\big)  \\
\stackrel{1\comp\xi^L_{\crta{\hat f_i}}\circ 1}{\Longrightarrow }  
F(\hat g)\comp F(\underset{i}{\textcircled{L}} \hat f_i) \circ \xi_i(\crta{A_{ij}}) \circ\big(\underset{i}{\textcircled{L}}\xi_j(\crta{A_{ij}})\big)  
\stackrel{F^2_{\hat g,\underset{i}{\textcircled{L}} \hat f_i}\circ 1\circ 1} {\Longrightarrow } 
F(\hat g\circ \big(\underset{i}{\textcircled{L}} \hat f_i\big)) \circ \xi_i(\crta{A_{ij}}) \circ\big(\underset{i}{\textcircled{L}}\xi_j(\crta{A_{ij}})\big) 
\end{multline}
mapping $\underset{i}{\textcircled{L}}\underset{j}{\textcircled{L}}F(A_{ij})\to F(C)$.

The hexagonal law for $\F$ can be drawn in terms of $F$ as the pasting of: 1) the following commuting diagram 
\[\begin{tikzcd}
	{\big(F(f)\circ\big(l_i\circ(\underset{i}{\textcircled{L}}F(g_i))\big)\big)\circ\big(l_{ij}\circ(\underset{ij}{\textcircled{L}}F(h_{ij}))\big)} && {F(f)\circ\big(\big(l_i\circ(\underset{i}{\textcircled{L}}F(g_i))\big)\circ\big(l_{ij}\circ(\underset{ij}{\textcircled{L}}F(h_{ij}))\big)\big)} \\
	{\big(F(f)\circ \big(F(\underset{i}{\textcircled{L}}g_i)\circ r_i\big)\big)\circ\big( F(\underset{ij}{\textcircled{L}}h_{ij})\circ r_{ij}\big)} && {F(f)\circ\big(\big(F(\underset{i}{\textcircled{L}}g_i)\circ r_i\big)\circ \big( F(\underset{ij}{\textcircled{L}}h_{ij})\circ r_{ij}\big)\big)}
	\arrow["{\tilde a^{-1}}", Rightarrow, nfold, from=1-1, to=1-3]
	\arrow["{(1\circ\xi^L_{left})\circ\tilde\xi^L_{right}}"', Rightarrow, nfold, from=1-1, to=2-1]
	\arrow["{1\circ(\xi^L_{left}\circ\tilde\xi^L_{right})}", Rightarrow, nfold, from=1-3, to=2-3]
	\arrow["{a^{-1}}", Rightarrow, nfold, from=2-1, to=2-3]
\end{tikzcd}\]
2) a similar one where in the vertical direction $r_i$ and $F(\underset{ij}{\textcircled{L}}h_{ij})$ interchange places via a naturality 2-cell, and 
3) the hexagonal law for $F$ from below (composed with the identity on $r_i' \circ r_{ij}$), for all three we know they commute, so the desired whole diagram commutes. (In the above diagram $\tilde a$ is the one from 
\equref{tilde a}, the ``noise factors'' $l$ and $r$ are: $l_i=\xi_i, \, l_{ij}=\big(\underset{i}{\textcircled{L}}\xi_{ij}\big)\circ\omega_b, \, r_i=F(\omega_b), 
\, r_{ij}=\xi_{ij}$ evaluated at the corresponding objects, and the 2-cells $(1\circ\xi^L_{left})\circ\tilde\xi^L_{right}$ and 
$1\circ(\xi^L_{left}\circ\tilde\xi^L_{right})$ are unprecise short annotations for the  generally written 2-cell $1\circ\xi^L_{left}\circ\tilde\xi^L_{right}$, 
whose part $\xi^L_{left}\circ\tilde\xi^L_{right}$ is the following composition of 2-cells
\begin{multline*}
 \big(l_i\circ(\underset{i}{\textcircled{L}}F(g_i))\big)\circ\big(l_{ij}\circ(\underset{ij}{\textcircled{L}}F(h_{ij}))\big)
\stackrel{\xi_{\crta{F(g_i)}}\circ 1}{\Longrightarrow} F(\underset{i}{\textcircled{L}}g_i)\circ l_i'\circ l_{ij}\circ
   \big(\underset{ij}{\textcircled{L}} F(h_{ij})\big)
\stackrel{1\circ  nat\circ 1}{\Longrightarrow} F(\underset{i}{\textcircled{L}}g_i)\circ r_i \circ r_{ij}'\circ \underset{ij}{\textcircled{L}}F(h_{ij}) \\
\stackrel{ 1\circ \xi_{\crta{F(h_{ij})}} }{\Longrightarrow} 
\big(\big(F(\underset{i}{\textcircled{L}}g_i)\circ r_i\big)\circ \big( F(\underset{ij}{\textcircled{L}}h_{ij})\circ r_{ij}\big)\big)
\end{multline*}
being 
$\omega_b$ from \equref{omega-a}, and $l_i'$ and $r_{ij}'$ differ from $l_i$ and $r_{ij}$ in that they are evaluated at different objects.) 
Similarly one shows the two unitality laws for $\F$. 

\medskip

With the above data we already have a pseudofunctor of bi-multicategories $\dul{\M}\to\dul{\N}$ and by \prref{psf birep} we know it respects the bi-representability 
in a lax way. To be explicit, let us comment to which data the modifications associated to $F$ correspond. 
We know that the given modifications $\omega, \gamma, \delta$ satisfy the axioms $(HTA1)$ and $(HTA2)$ from \cite[Definition 3.1]{GPS}.  
The 2-cell component $\omega_{A,B,C}$ take to be the 2-cell component of the modification $\delta_3$ from \prref{get modif delta for free}, and any side of the identity $(HTA1)$ 
can be taken for the 2-cell component of the modification $\delta_4$, so that the identity presents the compatibility condition from the 
second bullet item of \prref{psf birep}. 
The 2-cell components $\gamma_A, \delta_A$ take to be the 2-cell components of the modifications $\delta^L, \delta^R$. 
Then the last coherence required in \prref{psf birep} holds by the axiom $(HTA2)$. Finally, observe that the modification $\delta_{k+n+l;1}$ in 
\equref{delta_n forma modif- sa x i b} is given by the identity, as the equivalences $\Omega$ are identities. 

\medskip

On 1-cells, take a monoidal pseudonatural transformation $\beta$ with an invertible modification $\beta^2$ and the 2-cell $\beta^0$ and satisfying axioms 
MBTA1 - MBTA3 from \cite[Definition 2.7]{SP}. We extend $\beta$ to a pseudonatural transformation $\tilde\beta$ in $\dul{\M}$ with invertible modifications 
$\Theta_n$ as follows. For $f\in\dul{\M}(A_1,...,A_n;B)$ we have invertible 2-cells $\beta_{\hat f}$ in $\ul{\N}$ as in the diagram bellow, on the right, with 
 $\hat f:\textcircled{L} A_i\to B$ in $\ul{\M}$ the corresponding unary map, and we define the following composite 
\begin{equation} \eqlabel{beta-f} 
\tilde\beta_f:=
\begin{tikzcd}
	{\coprod_i F(A_i)} &&& {F(B)} \\
	& {m(\crta{ F(A_i)})} & {F(m(\crta{A_i}))} \\
	& {m(\crta{ G(A_i)})} & {G(m(\crta{A_i}))} \\
	{\coprod_i G(A_i)} &&& {G(B)}
	\arrow[""{name=0, anchor=center, inner sep=0}, "{\mathcal F(f)}", from=1-1, to=1-4]
	\arrow[""{name=1, anchor=center, inner sep=0}, "j", from=1-1, to=2-2]
	\arrow["{\coprod\beta_{A_i}}"', from=1-1, to=4-1]
	\arrow[""{name=2, anchor=center, inner sep=0}, "{\beta_B}", from=1-4, to=4-4]
	\arrow[""{name=3, anchor=center, inner sep=0}, "{\xi_n^\mathcal F}", from=2-2, to=2-3]
	\arrow["{m(\crta{\beta_{A_i}})}"', from=2-2, to=3-2]
	\arrow["{F(\hat f)}", from=2-3, to=1-4]
	\arrow["{\beta_{m(\crta{A_i})}}", from=2-3, to=3-3]
	\arrow[""{name=4, anchor=center, inner sep=0}, "{\xi_n^\mathcal G}"', from=3-2, to=3-3]
	\arrow["{j'}", tail reversed, no head, from=3-2, to=4-1]
	\arrow["{G(\hat f)}"', from=3-3, to=4-4]
	\arrow[""{name=5, anchor=center, inner sep=0}, "{\mathcal G(f)}"', from=4-1, to=4-4]
	\arrow["{=}"{description, pos=0.4}, draw=none, from=0, to=3]
	\arrow["\Epsilon"{description}, between={0.2}{0.8}, Rightarrow, nfold, from=1, to=4-1]
	\arrow["{\beta_{\hat f}}"{description, pos=0.4}, between={0.1}{0.8}, Rightarrow, nfold, from=2, to=3-3]
	\arrow["{\Theta_n^{-1}(\overline{A_i})}"{description}, between={0.2}{0.9}, Rightarrow, nfold, from=3, to=4]
	\arrow["{=}"{description}, draw=none, from=4, to=5]
\end{tikzcd}
\end{equation}
as the structure 2-cell of $\tilde\beta$ at $f$ in the bi-multicategory $\dul{\N}$. Here $\Theta_2=\beta^2$ (and $\Theta_0=\beta^0$) and for higher $n$ 
the components of $\Theta_n$ are defined by iteration as prescribed by the factorizations given in the point 1) of \prref{transf} (relying on the 4-ary compatibilities with associativity \equref{4-ary compat}, recall \rmref{Theta}). To obtain a transformation in bi-multicategories, we should check the 
associativity and unit law for $\tilde\beta$ 
from \cite[Definition 4.2.33]{Phil}. For this it will be helpful to first prove that $\Theta_n$ is a modification. We know this is true for $n=2$ by monoidality 
of $\beta$. For higher $n$ it is proved by induction, whereby $\xi_n(\crta f)$ factors through $m_{n-1}(1,\xi_2(f_i,f_{i+1}),1)$ and 
$\xi_{n-1}(f_1,..m_2(f_i,f_{i+1}),..,f_n)$ in a fully analogous way as it does $\Theta_n(\crta a)$ in the above cited proposition. In this proof one uses: 
the induction hypothesis for $n=2$ and for $n-1$, and the identities \equref{triples}.

\smallskip

To prove the associativity law for $\tilde\beta$, instead of drawing here the diagram computation that would take much space, we lay out the steps for such proof. 
One starts by writing out the transformation axiom for $\beta$ relating the 2-cell components $\beta_g, \beta_{m(\crta{f_i})}$ and 
$\beta_{g\circ(m(\crta{f_i}))}$ (it also contains the 2-cell components of $F^2$ and $G^2$ at $m(\crta{f_i})$ and $g$), 
paste to it the 2-cell $\Theta_n^{-1}(\crta{A_{ij}})$ from below and $\xi^F_{\crta{f_i}}$ from 
above on both sides. Then, on one side identify the subpasting of 2-cells $R(\xi^F_{\crta{f_i}}, \,\,\beta_{m(\crta{f_i})}, \,\,
\Theta_n^{-1}(\crta{A_{ij}}))$ (recall the notation from \equref{triples}) and use the modification axiom for $\Theta_n$ to substitute that subpasting with 
the pasting $L(\Theta_n^{-1}(\crta{B_i}), \,\,
m(\crta{\beta_{\crta{f_i}}}), \,\,\xi^G_{\crta{f_i}})$. In what is obtained, the pasting of $\beta_g$ and $\Theta_n^{-1}(\crta{B_i})$ 
represents $\tilde\beta_g$, the 2-cell $m(\crta{\beta_{\crta{f_i}}})$ represents in a condensed form $m(\tilde\beta_f)$ (forgetting the 
interchangers between composition and monoidal product), and the pasting of $\xi^G_{\crta{f_i}}$ and the component of $G^2$ represents $\G^2$. This 
completes the proof of the associativity axiom, as on the other side of the equation one has: 1) the pasting of $\xi^F_{\crta{f_i}}$ and the component of 
$F^2$ representing $\F^2$, and 2) the pasting of $\beta_{g\circ(m(\crta{f_i}))}$ and $\Theta_n^{-1}(\crta{A_{ij}})$ representing 
$\tilde\beta_{g\circ(m(\crta{f_i}))}$. (In this computation we ignored the associativity $a$: in a rigorous proof one would have to 
deal with it by pooling its naturality 2-cell components across the whole computation.)

The unit law for $\tilde\beta$ is proved in the same lines and we have that $\tilde\beta$ is a transformation of bi-multicategories.  

By \prref{transf} we now know $\tilde\beta$ is also a transformation of bi-representable bi-multicategories, let us outline the correspondence of the axioms. 
The law MBTA1 corresponds to the equation in point 2) of \prref{transf} for $n=3$, and the unity laws MBTA2 and MBTA3 from \cite[Definition 2.7]{SP} 
yield the unity laws from point 3) of \prref{transf}. The identity for higher $n$ is proven by induction 
using the 4-ary compatibilities with associativity \equref{4-ary compat} and \equref{pi-beta}.

\medskip

The modifications in $\MBic(\ul{\M},\ul{\N})$ yield clearly the case $n=2$ for modifications in $\RepBM(\M,\N)$. The identity 
\equref{modif-rep} for the induced modification $\Xi$ evaluated at $m_n(\crta{a_i})$ is shown by induction as follows. Write $\Theta'_n(\crta a)$ 
as a factorization  through $\Theta'_{n-1}(\crta a)$ and $\Theta'_2(\crta a)$ as in \prref{transf} part 1). Compose this with $\Xi_{m_n(\crta a)}$ 
at the 1-cell edge $\theta'_{m_n(\crta a)}$ of $\theta'_{a^*}$ (from the factorization of $\Theta'_n(\crta a)$). Based on the modification axiom, 
commute $\Xi_{m_n(\crta a)}$ with $\theta'_{a^*}$  to get $\theta_{a^*}$ and $\Xi_{m_{n-1}(\crta a)}$ (here $a: m_n(\crta a)\to 
m_{n-1}(a_1,...,m_2(a_i,a_{i+1}),...,a_n)$ and $a^*$ is its quasi-inverse). In the obtained total pasting recognize the left-hand side of \equref{modif-rep} for $n-1$ 
incrusted in the rest $R$ of the 2-cells. Then the total pasting is equal to the right-hand side of \equref{modif-rep} for $n-1$ 
(containing $m_{n-1}(\crta{\Xi_{a_i}})$) incrusted in $R$ in the exactly same way. 
In the latter total pasting, which let us call $\R$, move $m_{n-1}(\crta{\Xi_{a_i}})$ 
first along $m_{n-1}(1,\Theta_2(a_i,a_{i+1}),1)$ (because of the modification axiom at every $a_i$), and then along $a_{\crta{\theta_{a_i}}}$ 
(because of the naturality of $a$ at every variable via the single-variable 2-cell $\Xi_{a_i}$). What we obtained by this pulling the 2-cell 
components of $\Xi$ along the total pasting $\R$, is $\Theta_n(\crta a)$ pasted with $m_n(\crta{\Xi_{a_i}})$, proving the identity 
\equref{modif-rep} for general $n$. 

\medskip

This finishes the construction of a quasi-inverse $\U^*$ for $\U:\RepBM(\M,\N)\to\MBic(\ul{\M},\ul{\N})$.

\bigskip

Let us check the composition $\U^*\circ\U$. 

\smallskip

On a pseudofunctor $(\F, \xi, \delta)$ in $\RepBM(\M,\N)$ the only thing there is left to compare is the image $\F(f)$ for $f:a_1,...,a_n\to b$ and the structure 
$\F^2$ with their respective counterparts via $\U^*\circ\U$. 
Let $\U(\F)=F$, then we have $\Omega(f):\underset{i}{\textcircled{L}}a_i\to b$ and $(\U^*\circ\U)(\F)(f)=
\Sigma\big(F\big(\Omega(f)\big)\circ\xi^L(\crta{a_i})\big)$. Here $F\big(\Omega(f)\big)=\F\big(\Omega(f)\big)$ and  by \prref{lax mon mult} there is an 
invertible 2-cell $\xi_n(\crta{a_i})\Rightarrow\xi^L(\crta{a_i})$. Moreover, from \equref{modif delta} we have an invertible 2-cell 
$\delta_{\crta{a_i};b}:\F\big(\Omega(f)\big)\circ\xi_n(\crta{a_i})\Rightarrow\Omega\big(\F(f)\big)$, so there is  an invertible 2-cell $(\U^*\circ\U) (\F)(f)
\Rightarrow\Sigma\big(\Omega\big(\F(f)\big)\big)\Rightarrow\F(f)$, as $\Sigma$ is an equivalence of categories. 

In computing the image of the structure 2-cell $\F^2_{g,f_1,...,f_n}$ by $\U^*\circ\U$ one finds that a composition similar to (\ref{F2}), with 
$F^2_{g,\underset{i}{\textcircled{L}} f_i}$ substituted by $\Omega^\N(\F^2_{g,\underset{i}{\textcircled{L}} f_i})$, corresponds to the 
composite 2-cell in the diagram 
\[\begin{tikzcd}
	\bullet &&&&&& \\
	& \bullet &&&& \bullet \\
	\bullet && \bullet && \bullet && \bullet \\
	&& \bullet && \bullet \\
	& \bullet &&&& \bullet \\
	\bullet
	\arrow["{subst^\M_{1;1}}", curve={height=-6pt}, from=1-1, to=2-6]
	\arrow["{\F^1_{\bullet;\bullet}\times\F^1_{\bullet;\bullet}}"', from=1-1, to=3-1]
	\arrow[""{name=0, anchor=center, inner sep=0}, "{1\times\tilde\omega^\M}"'{pos=0.2}, from=2-2, to=1-1]
	\arrow[""{name=1, anchor=center, inner sep=0}, "{\F_{\bullet;\bullet}}", curve={height=-12pt}, from=2-6, to=3-7]
	\arrow["{1\times\N_1(\xi_j(\underset{i}{\textcircled{L}}\xi_j)\omega_a;1)}"', from=3-1, to=6-1]
	\arrow["{\Omega^\M_\bullet\times\Pi\Omega^\M_\bullet}"'{pos=0.4}, curve={height=-6pt}, from=3-3, to=2-2]
	\arrow["{subst^\M_{n;k_1,..,k_n}}", from=3-3, to=3-5]
	\arrow[""{name=2, anchor=center, inner sep=0}, "{\F_{\bullet;\bullet}^n\times\Pi_{i=1}^n\F^{k_1,..,k_n}_{\bullet;\bullet}}"', from=3-3, to=4-3]
	\arrow[""{name=3, anchor=center, inner sep=0}, "{\Omega^\M_\bullet}"'{pos=0.6}, from=3-5, to=2-6]
	\arrow[""{name=4, anchor=center, inner sep=0}, "{\F_{\bullet;\bullet}}", from=3-5, to=4-5]
	\arrow["{\N_1(\xi_j(\underset{i}{\textcircled{L}}\xi_j)\omega_a;1)}", curve={height=-12pt}, from=3-7, to=5-6]
	\arrow["{subst^\N_{n;k_1,..,k_n}}"', from=4-3, to=4-5]
	\arrow["{\Omega^\N_\bullet\times\Pi\Omega^\N_\bullet}"{pos=0.3}, curve={height=6pt}, from=4-3, to=5-2]
	\arrow[""{name=5, anchor=center, inner sep=0}, "{\Omega^\N_\bullet}", from=4-5, to=5-6]
	\arrow[""{name=6, anchor=center, inner sep=0}, "{1\times\tilde\omega^\N}"{pos=0.2}, from=5-2, to=6-1]
	\arrow["{subst^\N_{1;1}}"', curve={height=12pt}, from=6-1, to=5-6]
	\arrow["{\tilde\delta_\bullet}"{description, pos=0.3}, between={0.2}{0.4}, Rightarrow, nfold, from=0, to=6]
	\arrow["{\delta^{comp}_\bullet}"{description}, between={0.3}{0.7}, equals, nfold, 2tail reversed, from=1, to=5]
	\arrow["{\F_{\bullet;\bullet}^2}"{description}, between={0.2}{0.8}, Rightarrow, nfold, from=2, to=4]
	\arrow["{subst^\M_{\Omega_\bullet}}"{description}, between={0.3}{0.7}, equals, nfold, 2tail reversed, from=3, to=0]
	\arrow["{subst^\N_{\Omega_\bullet}}"{description}, between={0.3}{0.7}, equals, nfold, 2tail reversed, from=5, to=6]
\end{tikzcd}\]
evaluated at the object $g,f_1,...,f_n$ of $\M_n(b_1,...,b_n,c)\times\Pi_{i=1}^{n}\M_{k_i}(\crta{a_i};b_i)$. 
Here the bullets in the indexes denote the relevant objects in $\M$ and $\N$, the horizontal inner arrows $subst^\M_{n;k_1,..,k_n}$ denote the functors 
\equref{cat-multi} and the pentagons $subst^\M_{\Omega_\bullet}$ relying on them represent the diagrams \equref{def subst-diag} (up to $\Omega$), 
and the left and the right diagrams $\tilde\delta_\bullet$ and $\delta^{comp}_\bullet$ are the obvious (inverse) versions of the natural transformations 
\equref{modif delta}. Pulling back this composite by $\U^*$, {\em i.e.} by $\Sigma$, to the setting of $\RepBM(\M,\N)$ one comes back to the original 
structure 2-cell $\F^2_{g,f_1,...,f_n}$, as desired, since applying $\U^*$ means canceling the surrounding four natural transformations containing $\Omega$, 
by applying the analogous surrounding four natural transformations containing $\Sigma$ instead of $\Omega$. We conclude that $\U^*\circ\U$ on objects is 
isomorphic to identity.

\medskip

On a pseudonatural transformation $(\theta, (\Theta_n)_{n\in\Nn_0})$  in $\RepBM(\M,\N)$ one should check the relation between the 
structure 2-cells $\widetilde{\F(\theta)}_f$ for $f:a_1,...,a_n\to b$ 
and $\theta_f$, as well as between the modifications $\Theta_n$ and $\widetilde{\Theta_n}$ reconstructed from $\Theta_2$ by the factorization. 
For the latter identity there is then actually nothing to check. For the former identity note that the construction \equref{beta-f} corresponds almost by identity 
to the factorization \equref{theta-f} of the original $\theta_f$: they only differ in that the upper and lower trapeze with longer bases $\F(f)$ and $\G(f)$ 
in $\theta_f$ are non-trivial invertible 2-cells, while in $\widetilde{\F(\theta)}_f$ they are identities. Then these two transformation components differ by 
composing with invertible 2-cells, this makes that the 1-cells $\theta$ and $\widetilde{\F(\theta)}$ in the hom-bicategory differ by an isomorphism.

The level of modifications is the easiest, as the 2-cell components $\Xi_{m_n(\crta a)}$ precisely correspond to the 2-cell components 
$\Xi_{\underset{i}{\textcircled{L}} a_i}$, making $\U^*\circ\U$ identity on 2-cells.  

\medskip

From the construction it is clear that $\U\circ\U^*=\Id$, and we have that $\U:\RepBM(\M,\N)\to\MBic(\ul{\M},\ul{\N})$ is indeed a bi-equivalence of bicategories.

\subsection{The two enriched categories are equivalent} \sslabel{equiva enrich}

Let $\RepBM$ denote the category whose objects are bi-representable bi-multicategories, and whose hom-bicategories for two such objects $\M, \N$ 
are the bicategories $\BiMlt(\M,\N)=\RepBM(\M,\N)$ described in \ssref{fun-bicat}. From \ssref{mon bicat} we know we have a $\mathsf{Bicat}$-enriched 
category $\MBic$ of monoidal bicategories with hom-bicategories $\MBic(\B,\B')$. The bi-equivalence $\U:\RepBM(\M,\N)\to\MBic(\ul{\M},\ul{\N})$ 
induces then a $\mathsf{Bicat}$-enrichment from $\MBic$ onto $\RepBM$ with the composition $c$ defined by the following commuting diagram 
\[\begin{tikzcd}
	{\RepBM(\mathcal N,\mathcal L),\RepBM(\mathcal M,\mathcal N)} &&& {\RepBM(\mathcal M,\mathcal L)} \\
	{\MBic(\ul{\mathcal N},\ul{\mathcal L}),\MBic(\ul{\mathcal M},\ul{\mathcal N})} &&& {\MBic(\ul{\mathcal M},\ul{\mathcal L}).}
	\arrow["c", from=1-1, to=1-4]
	\arrow["{\mathcal U,\mathcal U}"', from=1-1, to=2-1]
	\arrow["{\mathcal U^*}", tail reversed, no head, from=1-4, to=2-4]
	\arrow["subst", from=2-1, to=2-4]
\end{tikzcd}\]
This way we obtain:

\begin{thm} \thlabel{main}
The two $\mathsf{Bicat}$-enriched categories $\MBic$ and $\RepBM$ are equivalent. 
\end{thm}

\section{Applications}

We have worked with the weakest notions of monoidal bicategories and bi-representable bi-multicategories to establish equivalence between their respective 
$\mathsf{Bicat}$-enriched categories. Having seen that the equivalence in this weakest case indeed exists, now we can easily draw conclusions on the 
correspondences between stricter versions of the notions. We do this in the first subsection below. Next, in \ssref{fun-birep-bimult} we saw that 
pseudofunctors of bi-multicategories for bi-representable bi-multicategories $\M, \N$ naturally induce {\em lax monoidal} pseudofunctors on the induced 
monoidal bicategories  $\ul{\M}, \ul{\N}$. In the second subsection below we will formulate fully monoidal version of the enriched equivalence. At last, 
in \ssref{Slat} we extend the formulation of Slattery's result on a bi-multicategory structure on the Kleisli bicategory of a pseudocommutative pseudomonad 
on a 2-multicategory. Namely, our developings contribute with the notion of bi-representability that was not included in that setting.

\subsection{Stricter versions of the result} \sslabel{Stricter}

In \ssref{2-mult} we recalled 2-multicategories: they are $\Cat_1$-multicategories and restricting to their unary 
multihom-categories gives their underlying 2-categories. 

\smallskip

By {\em trivial representability} on a bi-multicategory 
we mean the version of bi-representability in which the equivalence $\Omega$ is the identity (recall the fourth row of Table \ref{table1}). 
Then necessarily the monoidal bicategory corresponding to a trivially representable bi-multicategory is strictly monoidal. 

From our construction we can claim that there are the following correspondences: 
\begin{center}
strictly monoidal 2-categories $\leftrightarrow$ trivially representable 2-multicategories \\
(bicategorically)  monoidal bicategories $\leftrightarrow$ bi-representable bi-multicategories.
\end{center}
Let us now investigate the intermediate combinations of structures.

By a {\em bi-representable 2-multicategory} we mean a 
bi-representable bi-multicategory whose underlying bi-multicategory is a 2-multicategory. Similarly, by a {\em bicategorically 
monoidal 2-category} we mean a monoidal bicategory whose underlying bicategory is a 2-category. However, it is not true that 
bi-representable 2-multicategories correspond via an equivalence to bicategorically monoidal 2-categories. Let us have a closer 
look at this, as well as at the corresponding notion to Gray monoidal 2-categories. 


\smallskip

Verifying the proof of associativity $\tilde\alpha$ in the bi-multicategory induced from a monoidal bicategory $\B$ in \ssref{assoc}, we see that 
$\tilde\alpha$ was defined through the associativity $\alpha$ of the bicategory and $\tilde\gamma$. The latter is built from the transformations $\gamma, \theta, \omega, \Epsilon$. The first of them is the multi-variable version of the monoidality interchange $\xi$ of $\B$, $\omega$ is defined in terms of $\alpha$, and the remaining two are built out of $\gamma$ and $\alpha$. In conclusion, when constructing a bi-representable bi-multicategory from a monoidal bicategory: 
\begin{itemize}
\item the nature of bi-representability ({\em i.e.} of the equivalence $\Omega$) is the same as of the monoidal associativity $a$ (and unitalities $l,r$), and 
\item a bi-multicategory is weak if either the associativity $\alpha$ or the monoidal interchange $\xi$ is non-trivial, though, 
the underlying bicategory is of the same nature as the associativity $\alpha$ (recall \equref{tilde alfa}). 
\end{itemize} 
This leads to the classification of cases in the following Table: 
\begin{table}[H]
\begin{center}
\begin{tabular}{ c c c c | c } 
 & \hspace{0,3cm} assoc. $\alpha$, & mon. assoc. $a$, & mon. & induced \\ [0.5ex]
type of bicat. & \hspace{0,3cm} unit. $\lambda, \rho$ & unit. $l,r$ & interch. $\xi$ & bi-multicat. \\ [0.5ex]
\hline
strictly mon. 2-cat. &  \textcolor{verde}{$\id$'s} &  \textcolor{rojo}{$\id$'s} &  $\id$ & \textcolor{rojo}{tr.-rep.} \textcolor{verde}{2-mult.} \\ [1ex]   
Gray mon. 2-cat. &  \textcolor{verde}{$\id$'s} &  \textcolor{rojo}{$\id$'s} & \textcolor{indigo2}{$\iso$} & \textcolor{rojo}{tr.-rep.} 
   \textcolor{indigo2}{bi-mult.}  \\ 
 & & & & (\textcolor{verde}{with u.l. 2-cat.}) \\ [1ex]
bicat. mon. 2-cat. &  \textcolor{verde}{$\id$'s}  &  \textcolor{rojo}{$\simeq$} &  \textcolor{indigo2}{$\iso$} & \textcolor{rojo}{bi-rep.} 
\textcolor{indigo2}{bi-mult.} \\ 
 & & & & (\textcolor{verde}{with u.l. 2-cat.}) \\ [1ex]
bicat. mon. bicat. &  $\simeq$ &  \textcolor{rojo}{$\simeq$} &  \textcolor{indigo2}{$\iso$} & \textcolor{rojo}{bi-rep.} \textcolor{indigo2}{bi-mult.}\\ [1ex]   
\end{tabular}
\caption{From a monoidal bicategory $\B$ to a bi-representable bi-multicategory}
\label{table2}
\end{center}
\end{table}

It remains to check the converse assignment for the middle two cases. The summary in \ssref{cliq} agrees with the former bullet-item above in the following: 
bi-representability ({\em i.e.} the equivalence $\Omega$) and the monoidal associativity and unitalities $a, l,r$ are of the same nature. We now 
show how the fact that we deal with a bi-multicategory $\M$ having an underlying 2-category implies existence of a non-trivial monoidality interchange 
for the induced monoidal bicategory $\ul{\M}$. Apply the diagram \equref{assoc bi-mc} to the triple of multimaps $( h, \langle g_1, g_2\rangle, 
\langle\langle f_{11},f_{12}\rangle, \langle f_{21},f_{22}\rangle\rangle)$. Set $f_1=f_{11}\ot f_{12}$ and  $f_1=f_{21}\ot f_{22}$ (here $\ot$ is induced 
by the universal object for a binary list of objects) and take $h$ to be identity. Then the diagram gives a non-trivial 2-cell 
$$(g_1\circ f_1)\ot(g_2\circ f_2)\Rightarrow (g_1\ot g_2)\circ(f_1\ot f_2)$$
which is the monoidality interchange for $\ul{\M}$. (It necessarily agrees with the 2-cell we constructed in \equref{mon int}, by the universal property 
in $\M$.) Thus, if we are given a bi-multicategory $\M$ having an underlying 2-category, we obtain a monoidal 2-category $\ul{\M}$ with a non-trivial 
monoidality interchange, which is: strictly monoidal (if $\M$ is trivially representable), or bicategorically monoidal (if $\M$ is bi-representable). 
This confirms that the assignment in the middle two rows in Table \ref{table2} is indeed an equivalence, and we may summarize the equivalence assignments we 
found as follows:
\begin{table}[H]
\begin{center}
\begin{tabular}{| c  | c |} 
 \hline 
 \multicolumn{2}{|c|}{There is equivalence between:} \\ [1ex] 
 \hline\hline 
strictly monoidal 2-categories & trivially representable 2-multicategories  \\ [1ex]
\hline 
Gray monoidal 2-category & trivially representable bi-multicategories \\
 & with underlying 2-categories \\ [1ex]
\hline 
(bicategorically)  monoidal 2-categories & bi-representable bi-multicategories \\
 & with underlying 2-categories \\ [1ex]
\hline 
(bicategorically)  monoidal bicategories & bi-representable bi-multicategories \\
\hline 
\end{tabular}
\caption{Equivalences between versions of monoidal bicategories and bi-representable bi-multicategories}
\label{table3}
\end{center}
\end{table}

\bigskip

By strictifying the two groups of structures in \thref{main}, we find that the hom-bicategories $\RepBM(\M,\N)$ from \ssref{fun-bicat} restrict to 
hom-2-categories for trivially representable 2-multicategories and in this case one has an enrichment over the multicategory $\mathsf{2Cat}$ of 2-categories, 
as well as in the case of strictly monoidal 2-categories. Similarly as in \thref{main} we have: 


\begin{cor} 
The $\mathsf{2Cat}$-enriched categories 
of strictly monoidal 2-categories and of trivially representable 2-multicategories are equivalent. 
\end{cor}

The higher transforms for monoidal 2-categories (in all of the above three sorts of monoidality): monoidal 2-functors, monoidal 2-transformations and monoidal modifications form a 2-category. On the other hand, for trivially representable or bi-representable bi-multicategories having underlying 2-categories (in both 
cases), any choice of pseudofunctors, transformations and modifications forms a proper hom-{\em bicategory}, since 
the composition of transformations is expressed in terms of the composition of their 1-cell components living in the bi-multicategory, as we explained at the beginning of \ssref{fun-bicat}, which is non-trivially associative. For this reason we may say that strictifying the two groups of structures in \thref{main} 
to the middle two cases of strictness we obtain: 

\begin{cor} 
The $\mathsf{Bicat}$-enriched categories of Gray monoidal 2-categories and of trivially representable bi-multicategories having underlying 2-categories are equivalent. 

The $\mathsf{Bicat}$-enriched categories of bicategorically monoidal 2-categories and of bi-representable bi-multicategories having underlying 2-categories are equivalent. 
\end{cor}

\subsubsection{Bi-representable multicategories} \ssslabel{birep mult}

In \cite[Appendix A]{Camp} the author constructed a bi-representable multicategory of pseudodouble categories with pseudodouble functors. 
This is a 1-dimensional multicategory representable in a bicategorical way, as the representability correspondence is an equivalence rather than an isomorphism. Independently, we obtained a similar result for double categories and a subclass of lax double functors in 
\cite[Section 6.1]{Fem:Bif}.   
(We proved an equivalence indicating bi-representability by the Cartesian product at the level of 
binary multihoms 
$$q\x\Lax_{hop}^{st\x u}(\Aa\times\Bb,\Cc) \simeq \Lax_{hop}^{ud}(\Aa\times\Bb,\Cc)$$
and in \cite{Fem:Gray} we constructed a closed multicategory of double categories. The supra-script $u$ stands for {\em unitary}, while $d$ 
stands for {\em decomposable}, meaning that for the lax double functor structure $\gamma$ the globular 2-cells $\gamma_{(1_{A},g)(f,1_{B'})}$ 
for 1-cells $f,g$ are invertible. In particular, the equivalence holds when the objects of the above hom-double categories are pseudodouble functors.) 

\begin{rem}
Denote by $\F:q\x\Lax_{hop}^{st\x u}(\Aa\times\Bb,\Cc) \to \Lax_{hop}^{ud}(\Aa\times\Bb,\Cc)$ the above equivalence 
and by $\G$ its quasi-inverse. 
In \cite[Section 5.3]{Fem:Bif} we obtained that the horizontal strict transformation $\kappa: \Id \Rightarrow \G\F$ is defined via the 
unitors $\iota$ of the lax double quasi-functors, 
while $\lambda:\F\G\Rightarrow\Id$ is defined via (the 2-cell components of) the lax double functor structures $\gamma_{(1_{A},g)(f,1_{B'})}$. 
This explains why the equivalence is obtained for unitary and decomposable lax double functors. In particular, restricting $\F$ to 
{\em strict} double functors would yield to an isomorphism above and consequently to proper representability of the multicategory of 
double categories by the Cartesian product, recovering the known Cartesian monoidal product on the category of double categories and 
double functors. 
\end{rem}

\smallskip

By the third row of Table \ref{table3} we obtain: 

\begin{cor} \colabel{mixed}
Bi-representable multicategories are equivalent to bicategorically monoidal categories. 

In particular, the categories of double categories and unitary and decomposable lax double functors (respectively, pseudodouble functors), 
as well as of pseudodouble categories and pseudodouble functors are Cartesian monoidal in a bicategorical sense. 
\end{cor}

\subsection{Equivalence entailing pseudomonoidal transformations} \sslabel{fully mon}

In \prref{ksi} we proved that a pseudofunctor $\F:\M\to\N$ of the underlying bi-multicategories of bi-representable bi-multicategories 
$\M$ and $\N$ induces pseudonatural transformations $\xi_n: m(\F(-),..,\F(-))\to \F(m(-,..,-))$, and in \prref{psf birep} (via the modifications $\delta_n, \delta^L, \delta^R$) we showed that the $\xi_n$ make $\F$ a {\em lax representable} 
pseudofunctor. Requiring 
for pseudofunctors $\F:\M\to\N$ that their induced transformations $\xi_n$ be pseudonatural equivalences (meaning that there are 
pseudonatural transformations $\zeta_n: \F(m(-,..,-))\to m(\F(-),..,\F(-))$ and invertible modifications $\xi_n\circ\zeta_n\Rrightarrow\Id$ and 
$\zeta_n\circ\xi_n\Rrightarrow\Id$, 
or equivalently, that the 1-cell components $\xi_n(a_1,...,a_n)$ are equivalences), we obtain the pseudo version of 
\coref{respect univ o}:

\begin{cor} \colabel{fully mon}
Let $\M$ and $\N$ be  bi-representable bi-multicategories. Pseudofunctors $\F:\M\to\N$ between their underlying bi-multicategories 
equipped with pseudonatural equivalences $\xi_n: m(\F(-),..,\F(-))\to \F(m(-,..,-))$ 
preserve the universal objects (in a pseudo-representable way) up to equivalence. 
\end{cor}

On the other hand, 
for monoidal bicategories $\B$ and $\B'$ the hom-bicategories $\MBic(\B,\B')$ have for objects 
 {\em lax monoidal} pseudofunctors $\B\to\B'$. Let us set the following notation:
$\MBic_{lx.m}:=\MBic$ and $\RepBM_{lx.r}:=\RepBM$. The versions of $\mathsf{Bicat}$-enriched categories $\MBic$ and $\RepBM$ whose 
pseudofunctors are {\em pseudomonoidal}, respectively {\em pseudo-representable} let us denote similarly by $\MBic_{ps.m}$ and 
$\RepBM_{ps.r}$. Then there is a commuting diagram
\begin{equation} \eqlabel{commut diag}
\begin{tikzcd}
	{\RepBM_{ps.r}} &&& {\MBic_{ps.m}} \\
	{\RepBM_{lx.r}} &&& {\MBic_{lx.m}}
	\arrow["\simeq", from=1-1, to=1-4]
	\arrow["forget", from=1-4, to=2-4]
	\arrow["forget", tail reversed, no head, from=2-1, to=1-1]
	\arrow["\simeq"', from=2-1, to=2-4]
\end{tikzcd}
\end{equation}
where the equivalences on the horizontal arrows are those from \thref{main} (and its obvious variant).

\subsection{Bi-representable Kleisli bi-multicategories} \sslabel{Slat}

In his Ph.D. thesis \cite[Theorem 6.8]{Slat} Andrew Slattery proved, in the setting of relative pseudomonads, that given a pseudocommutative pseudomonad 
$T$ on a 2-multicategory $\M$ there is a structure of a bi-multicategory on the Kleisli bicategory $\Kl_T(\ul{\M})$. A pseudocommutative pseudomonad  
$T$ on $\M$ is in particular strong, where the notion of {\em strength on a bi-multicategory}, \cite[Definition 5.3.]{Slat}, means that 
\begin{itemize}
\item for every object $a$ there is an object $T(a)$ and a unit map $i_a:a\to T(a)$, and 
\item for any lists of objects $a_1,...,a_{i-1}, x, a_{i+1},...,a_n, \, y$ there is a family of functors 
$$st^i:\M(a_1,...,a_{i-1}, x, a_{i+1},...,a_n; T(y)) \to \M(a_1,...,a_{i-1}, T(x), a_{i+1},...,a_n; T(y))$$ 
pseudonatural in every argument. 
\end{itemize}
On the other hand, there is a known notion of strength on a {\em monoidal bicategory} (see \cite[Definition 9]{HF}, \cite[Definition 4.3]{HF2}). 
With our equivalence between monoidal bicategories and bi-representable bi-multicategories we are in a position to upgrade Slattery's (non-relative version of the) result in the following two forms.

\begin{thm} \thlabel{Slat}
Let $\B$ be a monoidal bicategory and suppose there is a pseudocommutative pseudomonad $T$ on $\B$. Then there is a structure of a 
bi-representable bi-multicategory on the Kleisli bicategory $\Kl_T(\B)$.

Let $\M$ be a bi-representable bi-multicategory and suppose there is a pseudocommutative pseudomonad $T$ on $\M$. Then there is a structure 
of a bi-representable bi-multicategory on the Kleisli bicategory $\Kl_T(\ul{\M})$ of the underlying bicategory of $\M$.
\end{thm}

\begin{proof}
By \cite[Theorems 7.3]{HF2} the Kleisli bicategory $\Kl_T(\B)$ is monoidal, and by what we proved in \seref{mon bic - rep} there is a bi-representable bi-multicategory structure on it. 

By our result in \seref{bi-bi->mon.bic}, $\ul{\M}$ is a monoidal bicategory. The pseudocommutative pseudomonad $T$ on $\M$ induces 
a pseudocommutative pseudomonad $T$ on $\ul{\M}$ (same notation for simplicity; a strength $t:a\ot T(b)\to T(a\ot b)$ is the image by 
$\Omega\circ st^2\circ\Sigma$ of $i_{a\ot b}:a\ot b\to T(a\ot b)$ as in the diagram below, where we identify $m(a,b)=a\ot b$)  
\[\begin{tikzcd}
	{\M_2(a,b; T(m(a,b)))} && {\M_1(a\ot b; T(a\ot b))} \\
	{\M_2(a,b; T(m(a,b)))} && {\M_1(a\ot T(b); T(a\ot b))}
	\arrow["{st^2}"', from=1-1, to=2-1]
	\arrow["\Sigma"', from=1-3, to=1-1]
	\arrow["{\crta{st^2}}", from=1-3, to=2-3]
	\arrow["\Omega"', from=2-1, to=2-3]
\end{tikzcd}\]
and the second statement follows from the first. 
\qed\end{proof}

\bigskip
\bigskip

{\large{\bf Acknowledgements.}} The author thanks to Nathanael Arkor for many inspiring and constructive conversations. In particular, 
for suggesting to consider $\mathsf{Bicat}$-enrichment in the equivalence, rather than fully weak tricategories, and for pointing out 
to the Slattery's result that we address in \ssref{Slat}. Special thanks to Pawel Sobocinski and the Laboratory for Compositional 
Systems and Methods at TalTech (Tallinn University of Technology) for hosting me during my two-week visit, which provided a great 
opportunity for the exchange of ideas and fruitful discussions, and thus supported this research.

\bigskip

{\large{\bf Appendix}}

    \begin{landscape}
				\centering
\[\begin{tikzcd}[column sep=small]
	{\big(\textcircled{L}_{i=1}^{n-1} f_i(h_ik_i)\big)\ot f_n(h_nk_n)} & {(\textcircled{L}_{i=1}^{n-1} f_i)(\textcircled{L}_{i=1}^{n-1} h_ik_i)\ot f_n(h_nk_n)} && {f(\textcircled{L}_{i=1}^n h_ik_i)} \\
	{\big(\textcircled{L}_{i=1}^{n-1} (f_i h_i)k_i\big)\ot f_n(h_nk_n)} & {\big(\textcircled{L}_{i=1}^{n-1} f_i\big)\big(\big(\textcircled{L}_{i=1}^{n-1} h_i\big)\big(\textcircled{L}_{i=1}^{n-1} k_i\big)\big)\ot f_n(h_n k_n)} && {f(hk)} \\
	{\big(\textcircled{L}_{i=1}^{n-1} f_i h_i\big)\big(\textcircled{L}_{i=1}^{n-1}k_i\big)\ot f_n(h_nk_n)} & {\big(\big(\textcircled{L}_{i=1}^{n-1} f_i\big)\big(\textcircled{L}_{i=1}^{n-1} h_i\big)\big)\big(\textcircled{L}_{i=1}^{n-1} k_i\big)\ot f_n (h_nk_n)} & {f\big(\big(\textcircled{L}_{i=1}^{n-1} h_i\big)\big(\textcircled{L}_{i=1}^{n-1} k_i\big)\ot (h_n k_n)\big)} \\
	{\big(\textcircled{L}_{i=1}^{n-1} f_i h_i\big)\big(\textcircled{L}_{i=1}^{n-1}k_i\big)\ot (f_n h_n)k_n} & {\big(\big(\textcircled{L}_{i=1}^{n-1} f_i\big)\big(\textcircled{L}_{i=1}^{n-1} h_i\big)\big)\big(\textcircled{L}_{i=1}^{n-1} k_i\big)\ot (f_n h_n) k_n} \\
	& {\big(\big(\textcircled{L}_{i=1}^{n-1} f_i\big)\big(\textcircled{L}_{i=1}^{n-1} h_i\big)\ot (f_n h_n)\big)\big((\textcircled{L}_{i=1}^{n-1} k_i) \ot k_n\big)} && {(fh)k}
	\arrow[""{name=0, anchor=center, inner sep=0}, "{\gamma_{.,..}\ot 1}", Rightarrow, nfold, from=1-1, to=1-2]
	\arrow["{\underset{i}{\textcircled{L}}\alpha\ot 1}"', Rightarrow, nfold, from=1-1, to=2-1]
	\arrow["{\xi_{., ..}}", Rightarrow, nfold, from=1-2, to=1-4]
	\arrow[""{name=1, anchor=center, inner sep=0}, "{1\gamma\ot 1}", Rightarrow, nfold, from=1-2, to=2-2]
	\arrow["{1\gamma}", Rightarrow, nfold, from=1-4, to=2-4]
	\arrow["{\gamma_{..,.}\ot 1}"', Rightarrow, nfold, from=2-1, to=3-1]
	\arrow["{\alpha\ot 1}", Rightarrow, nfold, from=2-2, to=3-2]
	\arrow[""{name=2, anchor=center, inner sep=0}, "{\xi_{.,..}}", Rightarrow, nfold, from=2-2, to=3-3]
	\arrow["\alpha", Rightarrow, nfold, from=2-4, to=5-4]
	\arrow[""{name=3, anchor=center, inner sep=0}, "{\gamma 1\ot 1}"', Rightarrow, nfold, from=3-1, to=3-2]
	\arrow["{1\ot  \alpha}"', Rightarrow, nfold, from=3-1, to=4-1]
	\arrow["{1\ot\alpha}", Rightarrow, nfold, from=3-2, to=4-2]
	\arrow[""{name=4, anchor=center, inner sep=0}, "{1\xi}", Rightarrow, nfold, from=3-3, to=2-4]
	\arrow[""{name=5, anchor=center, inner sep=0}, "{\gamma 1\ot 1}"', Rightarrow, nfold, from=4-1, to=4-2]
	\arrow["{\xi_{..,.}}"', Rightarrow, nfold, from=4-2, to=5-2]
	\arrow[""{name=6, anchor=center, inner sep=0}, "{\xi 1}", equals, nfold, 2tail reversed, from=5-4, to=5-2]
	\arrow["{hex_{n-1}\ot 1}"{description, pos=0.3}, draw=none, from=0, to=3]
	\arrow["{nat_\xi w.r.t. \gamma}"{description}, draw=none, from=1, to=4]
	\arrow["{hex_2}"{description}, draw=none, from=2, to=6]
	\arrow["{obv.}"{description}, draw=none, from=3, to=5]
\end{tikzcd}\]
\captionof{table}{The $i$-ary hexagon for $\gamma$ from \sssref{comput pentag}}
    \end{landscape}
\label{hex_i}


\begin{thebibliography}{99}


\bibitem{Bou} J. Bourke, {\em Skew structures in 2-category theory and homotopy theory}, J. Homotopy Relat. Struct. {\bf 12}/1 (2017), 31--81.

\bibitem{BL} J. Bourke, G. Lobbia, {\em A skew approach to enrichment for Gray-categories}, 
Advances in Mathematics {\bf 434} (2023), https://doi.org/10.1016/j.aim.2023.109327.


\bibitem{Camp} A. Campbell, {\em How strict is strictification?}, J. Pure Appl. Algebra {\bf 223}/7 (2019), 2948--2976. 


\bibitem{DS} B. Day, R. Street, {\em Monoidal Bicategories and Hopf Algebroids}, Adv. Math. {\bf 129}/1 (1997), 99--157. 


\bibitem{FS} A. D. Fairbanks, M. Shulman, {\em Doubly Weak Double Categories}, {\bf 34}/26 (2026). 


\bibitem{FMS} P.F. Faul, G. Manuell, J. Siqueira, {\em 2-Dimensional bifunctor theorems and distributive laws}, Theory Appl. Categ. {\bf 37} 
(2021), 1149--1175. 


\bibitem{Fem:Bif} B. Femi\'c, {\em Bifunctor Theorem and strictification tensor product for double categories with lax double functors}, 
Theory Appl. Categ. {\bf 39} (2023), 824--873. 




\bibitem{Fem:Gray} B. Femi\'c, {\em Gray (skew) multicategories for double and Gray-categories}, arxiv.org/abs/2408.00561.  


\bibitem{Fem:Prem} B. Femi\'c, {\em Premonoidal and Kleisli double categories}, arxiv.org/abs/2401.17494v3.

\bibitem{FLK} F. Foltz, Ch. Lair, G. M. Kelly, {\em Algebraic categories with few monoidal biclosed structures or none}, 
Journal of Pure and Applied Algebra {\bf 17}/2 (1980), 171--177.  


\bibitem{GGV} N. Gambino, R. Garner, C. Vasilakopoulou, {\em Monoidal Kleisli Bicategories and the Arithmetic Product of 
Coloured Symmetric Sequences},  Doc. Math. {\bf 29}/3 (2024), 627--702. 


\bibitem{GPS} R. Gordon, A. J. Power, R. Street, {\em Coherence for tricategories}, Memoirs Amer. Math. Soc. {\bf 117} (1995).


\bibitem{GP:Fram} M. Grandis, R. Par\'e, {\em Intercategories: a framework for three-dimensional category theory}, J. Pure Appl.
Algebra {\bf 221}/5 (2017), 999--1054. 


\bibitem{Gray} J. W. Gray, {\em Formal category theory: adjointness for 2-categories}, Lecture Notes in Mathematics
{\bf 391}, Springer-Verlag, Berlin-New York (1974) 1, 19, 27. 


\bibitem{Herm} C. Hermida, {\em Representable multicategories}, Adv. Math. {\bf 151}/2 (2000), 164--225.



\bibitem{JY} N. Johnson, D. Yau, {\em Homotopy Theory of Enriched Mackey Functors}, 
London Mathematical Society Lecture Note Series {\bf 492}, Cambridge University Press (2025). 


\bibitem{JS} A. Joyal, R. Street, {\em Geometry of Tensor Calculus I}, Advances in Mathematics 88 (1991), 55--112.


\bibitem{LSR} D. R. Licata, M. Shulman, M. Riley, {\em A fibrational framework for substructural and modal logics}, In FSCD, 2017.


\bibitem{Lob} G. Lobbia, {\em A finite approach to representable multicategories and related structures}, 
Theory and Applications of Categories {\bf 45}/10 (2026), 332--390.


\bibitem{McL} S. Mac Lane, {\em Categories for the Working Mathematicians}, Graduate Texts in Mathematics, Springer-Verlag (1971).


\bibitem{MV} J. Moeller, Ch. Vasilakopoulou, {\em Monoidal Grothendieck Construction}, 
Theory and Applications of Categories {\bf 35}/31 (2020), 1159--1207.


\bibitem{HF2} H. Paquet, P. Saville, {\em Effectful semantics in bicategories: strong, commutative, and concurrent pseudomonads}, 
LICS '24: Proceedings of the 39th Annual ACM/IEEE Symposium on Logic in Computer Science Article {\bf 61}, 1 -- 15, 
https://doi.org/10.1145/3661814.3662130. 


\bibitem{HF} H. Paquet, P. Saville, {\em Strong pseudomonads and premonoidal bicategories}, arXiv:2304.11014 (choose v1). 


\bibitem{Phil} P. Saville, {\em Cartesian closed bicategories: type theory and coherence}, PhD thesis (2019). 


\bibitem{SP} Ch. J. Schommer-Pries, {\em The classification of two-dimensional extended topological field theories}, PhD thesis? (2009).


\bibitem{Shul-sl}  M. Shulman, {\em Extraordinary multicategories}, slides from a talk at the 2010 CMS Summer Meeting. 


\bibitem{Slat} A. Slattery, {\em Commutativity of Relative Pseudomonads}, University of Leeds, PhD thesis (2024). 
https://etheses.whiterose.ac.uk/id/eprint/35372/


\bibitem{Str80} R. Street, {\em Fibrations in bicategories}, Cahiers Topologie G\'eom. Diff\'erentielle {\bf 21}/2 (1980), 111--160.


\bibitem{Ver} D. Verity, {\em Enriched categories, internal categories and change of base}, Repr. Theory Appl. Categ. {\bf 20} (2011), 1--266.

\end{thebibliography}
\end{document}